\documentclass[12pt]{amsart}
\usepackage{mathrsfs}
\usepackage{cases}
\usepackage{epic}
\usepackage{amsfonts}
\usepackage{tikz}
\usepackage{tikz-cd}
\usepackage{graphicx}
\usepackage{amsmath}
\usepackage{amssymb, upgreek}
\usepackage{bm}
\usepackage{latexsym,todonotes}
\usepackage{pdflscape}
\usepackage[all]{xypic}
\usepackage[all]{xy}
\usepackage{tikz-cd}
\usepackage{color}
\usepackage{colordvi}
\usepackage{multicol}
\usepackage{mathtools}
\usepackage[normalem]{ulem}
\usepackage[linktocpage=true]{hyperref}
\hypersetup{colorlinks,linkcolor=blue,urlcolor=cyan,citecolor=blue}
\usepackage{graphicx}
\NewDocumentCommand{\sslash}{s}{%
	\IfBooleanTF{#1}
	{\big/\mkern-7mu\big/}
	{/\mkern-6mu/}%
}
\makeatletter
\newsavebox{\@brx}
\newcommand{\llangle}[1][]{\savebox{\@brx}{\(\m@th{#1\langle}\)}%
	\mathopen{\copy\@brx\kern-0.5\wd\@brx\usebox{\@brx}}}
\newcommand{\rrangle}[1][]{\savebox{\@brx}{\(\m@th{#1\rangle}\)}%
	\mathclose{\copy\@brx\kern-0.5\wd\@brx\usebox{\@brx}}}
\makeatother

\tikzcdset{scale cd/.style={every label/.append style={scale=#1},
		cells={nodes={scale=#1}}}}

\DeclareMathAlphabet{\mathbbb}{U}{bbold}{m}{n}
\DeclareMathOperator{\dimv}{\underline{\dim}}

\makeatletter

\newcommand{\Rmnum}[1]{\textup{\expandafter\@slowromancap\romannumeral #1@}}
\makeatother

\allowdisplaybreaks

\begin{document}
	\input xy
	\xyoption{all}
	\newcommand{\iLa}{\Lambda^{\imath}}
	\newcommand{\iadd}{\operatorname{iadd}\nolimits}
	\renewcommand{\mod}{\operatorname{mod}\nolimits}
	\newcommand{\fproj}{\operatorname{f.proj}\nolimits}
	\newcommand{\Fac}{\operatorname{Fac}\nolimits}
	\newcommand{\ci}{{\I}_{\btau}}
	\newcommand{\proj}{\operatorname{proj}\nolimits}
	\newcommand{\inj}{\operatorname{inj}\nolimits}
	\newcommand{\rad}{\operatorname{rad}\nolimits}
	\newcommand{\Span}{\operatorname{Span}\nolimits}
	\newcommand{\soc}{\operatorname{soc}\nolimits}
	\newcommand{\ind}{\operatorname{inj.dim}\nolimits}
	\newcommand{\Ginj}{\operatorname{Ginj}\nolimits}
	\newcommand{\res}{\operatorname{res}\nolimits}
	\newcommand{\np}{\operatorname{np}\nolimits}
	\newcommand{\Mor}{\operatorname{Mor}\nolimits}
	\newcommand{\Mod}{\operatorname{Mod}\nolimits}
	\newcommand{\End}{\operatorname{End}\nolimits}
	\newcommand{\lf}{\operatorname{l.f.}\nolimits}
	\newcommand{\Iso}{\operatorname{Iso}\nolimits}
	\newcommand{\Aut}{\operatorname{Aut}\nolimits}
	\newcommand{\Rep}{\operatorname{Rep}\nolimits}
	
	\newcommand{\colim}{\operatorname{colim}\nolimits}
	\newcommand{\gldim}{\operatorname{gl.dim}\nolimits}
	\newcommand{\cone}{\operatorname{cone}\nolimits}
	\newcommand{\rep}{\operatorname{rep}\nolimits}
	\newcommand{\Ext}{\operatorname{Ext}\nolimits}
	\newcommand{\Tor}{\operatorname{Tor}\nolimits}
	\newcommand{\Hom}{\operatorname{Hom}\nolimits}
	\newcommand{\Top}{\operatorname{top}\nolimits}
	\newcommand{\Coker}{\operatorname{Coker}\nolimits}
	\newcommand{\thick}{\operatorname{thick}\nolimits}
	\newcommand{\rank}{\operatorname{rank}\nolimits}
	\newcommand{\Gproj}{\operatorname{Gproj}\nolimits}
	\newcommand{\Len}{\operatorname{Length}\nolimits}
	\newcommand{\RHom}{\operatorname{RHom}\nolimits}
	\renewcommand{\deg}{\operatorname{deg}\nolimits}
	\renewcommand{\Im}{\operatorname{Im}\nolimits}
	\newcommand{\Ker}{\operatorname{Ker}\nolimits}
	\newcommand{\Coh}{\operatorname{Coh}\nolimits}
	\newcommand{\Id}{\operatorname{Id}\nolimits}
	\newcommand{\Qcoh}{\operatorname{Qch}\nolimits}
	\newcommand{\CM}{\operatorname{CM}\nolimits}
	\newcommand{\sgn}{\operatorname{sgn}\nolimits}
	\newcommand{\utMH}{\operatorname{\cm\ch(\iLa)}\nolimits}
	\newcommand{\GL}{\operatorname{GL}}
	\newcommand{\Perv}{\operatorname{Perv}}
	
	\newcommand{\IC}{\operatorname{IC}}
	\def \hU{\widehat{\U}}
	\def \hUi{\widehat{\U}^\imath}
	\newcommand{\bb}{\psi_*}
	\newcommand{\bvs}{{\boldsymbol{\varsigma}}}
	\def \ba{\mathbf{a}}
	\newcommand{\vs}{\varsigma}
	\def \bfk {\mathbf{k}}

	\def \bd{\mathbf{d}}
	\newcommand{\e}{{\bf 1}}
	\newcommand{\EE}{E^*}
	\newcommand{\dbl}{\operatorname{dbl}\nolimits}
	\newcommand{\ga}{\gamma}
	\newcommand{\tM}{\cm\widetilde{\ch}}
	\newcommand{\la}{\lambda}
	
	\newcommand{\For}{\operatorname{{\bf F}or}\nolimits}
	\newcommand{\coker}{\operatorname{Coker}\nolimits}
	\newcommand{\rankv}{\operatorname{\underline{rank}}\nolimits}
	\newcommand{\diag}{{\operatorname{diag}\nolimits}}
	\newcommand{\swa}{{\operatorname{swap}\nolimits}}
	\newcommand{\supp}{{\operatorname{supp}}}
	
	\renewcommand{\Vec}{{\operatorname{Vec}\nolimits}}
	\newcommand{\pd}{\operatorname{proj.dim}\nolimits}
	\newcommand{\gr}{\operatorname{gr}\nolimits}
	\newcommand{\id}{\operatorname{id}\nolimits}
	\newcommand{\aut}{\operatorname{Aut}\nolimits}
	\newcommand{\Gr}{\operatorname{Gr}\nolimits}
	
	\newcommand{\pdim}{\operatorname{proj.dim}\nolimits}
	\newcommand{\idim}{\operatorname{inj.dim}\nolimits}
	\newcommand{\Gd}{\operatorname{G.dim}\nolimits}
	\newcommand{\Ind}{\operatorname{Ind}\nolimits}
	\newcommand{\add}{\operatorname{add}\nolimits}
	\newcommand{\pr}{\operatorname{pr}\nolimits}
	\newcommand{\oR}{\operatorname{R}\nolimits}
	\newcommand{\oL}{\operatorname{L}\nolimits}
	\def \brW{\mathrm{Br}(W_\btau)}
	\newcommand{\Perf}{{\mathfrak Perf}}
	\newcommand{\cc}{{\mathcal C}}
	\newcommand{\gc}{{\mathcal GC}}
	\newcommand{\ce}{{\mathcal E}}
	\newcommand{\calI}{{\mathcal I}}
	\newcommand{\cs}{{\mathcal S}}
	\newcommand{\cf}{{\mathcal F}}
	\newcommand{\cx}{{\mathcal X}}
	\newcommand{\cy}{{\mathcal Y}}
	\newcommand{\ct}{{\mathcal T}}
	\newcommand{\cu}{{\mathcal U}}
	\newcommand{\cv}{{\mathcal V}}
	\newcommand{\cn}{{\mathcal N}}
	\newcommand{\mcr}{{\mathcal R}}
	\newcommand{\ch}{{\mathcal H}}
	\newcommand{\ca}{{\mathcal A}}
	\newcommand{\cb}{{\mathcal B}}
	\newcommand{\cj}{{\mathcal J}}
	\newcommand{\cl}{{\mathcal L}}
	\newcommand{\cm}{{\mathcal M}}
	\newcommand{\cp}{{\mathcal P}}
	\newcommand{\cg}{{\mathcal G}}
	\newcommand{\cw}{{\mathcal W}}
	\newcommand{\co}{{\mathcal O}}
	\newcommand{\cq}{{\mathcal Q}}
	\newcommand{\cd}{{\mathcal D}}
	\newcommand{\ck}{{\mathcal K}}
	\newcommand{\calr}{{\mathcal R}}
	\newcommand{\cz}{{\mathcal Z}}
	\newcommand{\ol}{\overline}
	\newcommand{\ul}{\underline}
	\newcommand{\st}{[1]}
	\newcommand{\ow}{\widetilde}
	\renewcommand{\P}{\mathbf{P}}
	\newcommand{\pic}{\operatorname{Pic}\nolimits}
	\newcommand{\Spec}{\operatorname{Spec}\nolimits}
	\newcommand{\Fr}{\mathrm{Fr}}
	\newcommand{\Gp}{\mathrm{Gp}}

	%Theorem for the introduciton
	\newtheorem{innercustomthm}{{\bf Theorem}}
	\newenvironment{customthm}[1]
	{\renewcommand\theinnercustomthm{#1}\innercustomthm}
	{\endinnercustomthm}
	
	\newtheorem{innercustomcor}{{\bf Corollary}}
	\newenvironment{customcor}[1]
	{\renewcommand\theinnercustomcor{#1}\innercustomcor}
	{\endinnercustomthm}
	
	\newtheorem{innercustomprop}{{\bf Proposition}}
	\newenvironment{customprop}[1]
	{\renewcommand\theinnercustomprop{#1}\innercustomprop}
	{\endinnercustomthm}
	
	\newtheorem{theorem}{Theorem}[section]
	\newtheorem{acknowledgement}[theorem]{Acknowledgement}
	\newtheorem{algorithm}[theorem]{Algorithm}
	\newtheorem{axiom}[theorem]{Axiom}
	\newtheorem{case}[theorem]{Case}
	\newtheorem{claim}[theorem]{Claim}
	\newtheorem{conclusion}[theorem]{Conclusion}
	\newtheorem{condition}[theorem]{Condition}
	\newtheorem{conjecture}[theorem]{Conjecture}
	\newtheorem{construction}[theorem]{Construction}
	\newtheorem{corollary}[theorem]{Corollary}
	\newtheorem{criterion}[theorem]{Criterion}
	\newtheorem{definition}[theorem]{Definition}
	\newtheorem{example}[theorem]{Example}
	\newtheorem{assumption}[theorem]{Assumption}
	\newtheorem{lemma}[theorem]{Lemma}
	\newtheorem{notation}[theorem]{Notation}
	\newtheorem{problem}[theorem]{Problem}
	\newtheorem{proposition}[theorem]{Proposition}
	\newtheorem{solution}[theorem]{Solution}
	\newtheorem{summary}[theorem]{Summary}
	\newtheorem{hypothesis}[theorem]{Hypothesis}
	\newtheorem*{thm}{Theorem}
	
	\theoremstyle{remark}
	\newtheorem{remark}[theorem]{Remark}
	
	\def \Br{\mathrm{Br}}
	\newcommand{\tK}{K}
	
	\newcommand{\tk}{\widetilde{k}}
	\newcommand{\tU}{\widetilde{{\mathbf U}}}
	\newcommand{\Ui}{{\mathbf U}^\imath}
	\newcommand{\tUi}{\widetilde{{\mathbf U}}^\imath}
	\newcommand{\qbinom}[2]{\begin{bmatrix} #1\\#2 \end{bmatrix} }
	\newcommand{\ov}{\overline}
	\newcommand{\tMHg}{\operatorname{\widetilde{\ch}(Q,\btau)}\nolimits}
	\newcommand{\tMHgop}{\operatorname{\widetilde{\ch}(Q^{op},\btau)}\nolimits}
	
	\newcommand{\rMHg}{\operatorname{\ch_{\rm{red}}(Q,\btau)}\nolimits}
	\newcommand{\dg}{\operatorname{dg}\nolimits}
	\def \fu{{\mathfrak{u}}}
	\def \fv{{\mathfrak{v}}}
	\def \sqq{{\mathbbb{v}}}
	\def \bp{{\mathbf p}}
	\def \bv{{\mathbf v}}
	\def \bw{{\mathbf w}}
	\def \bA{{\mathbf A}}
	\def \bL{{\mathbf L}}
	\def \bF{{\mathbf F}}
	\def \bS{{\mathbf S}}
	\def \bC{{\mathbf C}}
	\def \bU{{\mathbf U}}
	\def \U{{\mathbf U}}
	\def \btau{\varpi}
	\def \La{\Lambda}
	\def \Res{\Delta}
	\newcommand{\ev}{\bar{0}}
	\newcommand{\odd}{\bar{1}}
	\def \fk{\mathfrak{k}}
	\def \ff{\mathfrak{f}}
	\def \fp{{\mathfrak{P}}}
	\def \fg{\mathfrak{g}}
	\def \fn{\mathfrak{n}}
	\def \gr{\mathfrak{gr}}
	\def \Z{\mathbb{Z}}
	\def \F{\mathbb{F}}
	\def \D{\mathbb{D}}
	\def \C{\mathbb{C}}
	\def \N{\mathbb{N}}
	\def \Q{\mathbb{Q}}
	\def \G{\mathbb{G}}
	\def \P{\mathbb{P}}
	\def \K{\mathbb{K}}
	\def \E{\mathbb{K}}
	\def \I{\mathbb{I}}
	
	\def \eps{\varepsilon}
	\def \BH{\mathbb{H}}
	\def \btau{\varrho}
	\def \cv{\varpi}
	
	\def \tR{\widetilde{\bf R}}
	\def \tRZ{\widetilde{\bf R}_\cz}
	\def \hR{\widehat{\bf R}}
	\def \hRZ{\widehat{\bf R}_\cz}
	\def\tRi{\widetilde{\bf R}^\imath}
	\def\hRi{\widehat{\bf R}^\imath}
	\def\tRiZ{\widetilde{\bf R}^\imath_\cz}
	\def\reg{\mathrm{reg}}
	
	\def\hRiZ{\widehat{\bf R}^\imath_\cz}
	\def \tTT{\widetilde{\mathbf{T}}}
	\def \TT{\mathbf{T}}
	\def \br{\mathbf{r}}
	\def \bp{{\mathbf p}}
	\def \tS{\texttt{S}}
	\def \bq{{\bm q}}
	\def \bvt{{v}}
	\def \bs{{ r}}
	\def \tt{{v}}
	\def \k{k}
	\def \bnu{\bm{\nu}}
	\def\bc{\mathbf{c}}
	\def \ts{\textup{\texttt{s}}}
	\def \tt{\textup{\texttt{t}}}
	\def \tr{\textup{\texttt{r}}}
	\def \tc{\textup{\texttt{c}}}
	\def \tg{\textup{\texttt{g}}}
	\def \bW{\mathbf{W}}
	\def \bV{\mathbf{V}}

	\newcommand{\browntext}[1]{\textcolor{brown}{#1}}
	\newcommand{\greentext}[1]{\textcolor{green}{#1}}
	\newcommand{\redtext}[1]{\textcolor{red}{#1}}
	\newcommand{\bluetext}[1]{\textcolor{blue}{#1}}
	\newcommand{\brown}[1]{\browntext{ #1}}
	\newcommand{\green}[1]{\greentext{ #1}}
	\newcommand{\red}[1]{\redtext{ #1}}
	\newcommand{\blue}[1]{\bluetext{ #1}}
	\numberwithin{equation}{section}
	\renewcommand{\theequation}{\thesection.\arabic{equation}}
	
	%todo
	\newcommand{\wtodo}{\rightarrowdo[inline,color=orange!20, caption={}]}
	\newcommand{\lutodo}{\rightarrowdo[inline,color=green!20, caption={}]}
	\def \tT{\widetilde{\mathcal T}}
	
	\def \tTL{\tT(\iLa)}
	\def \iH{\widetilde{\ch}}
	
	%\title[Dual canonical bases arising from quantum symmetric pairs]{Dual canonical bases arising from quantum symmetric pairs}
	
	\title[Dual and double canonical bases of quantum groups]{Dual and double canonical bases of quantum groups}
	
	\author[Ming Lu]{Ming Lu}
	\address{Department of Mathematics, Sichuan University, Chengdu 610064, P.R.China}
	\email{luming@scu.edu.cn}

	\author[Xiaolong Pan]{Xiaolong Pan}
	\address{Department of Mathematics, Sichuan University, Chengdu 610064, P.R.China}
	\email{xiaolong\_pan@stu.scu.edu.cn}


	\begin{abstract}
		Qin established the geometric realization of entire quantum groups via perverse sheaves, which further give rise to dual canonical bases with integral and positive structure constants for quantum groups of  type ADE. In this paper, we prove that the dual canonical bases of (Drinfeld double) quantum groups coincide with Berenstein--Greenstein's double canonical bases, by reinterpreting their intricate algebraic construction via the geometry of NKS quiver varieties. This result settles several conjectures therein, including those on positivity and invariance under braid group actions.
	\end{abstract}

	\maketitle
	\setcounter{tocdepth}{1}
	\tableofcontents
	
	%%%%%%%
	\section{Introduction}
	
	Ringel \cite{Rin90} first introduced Hall algebras of Dynkin quivers to realize the positive part $\U^+$ of the quantum group $\U$. Inspired by Ringel's pioneering work, Lusztig \cite{Lus90,Lus91,Lus93} utilized perverse sheaves on the representation variety of a quiver $Q$ to give a geometric realization of $\U^+$, and further constructed the canonical basis of $\U^+$ from simple perverse sheaves; see also \cite{Ka91} for the crystal basis construction. These seminal works are widely recognized as the earliest examples of categorifying the half quantum groups.
	
	Bridgeland \cite{Br13} later constructed a Hall algebra from $\Z_2$-graded complexes of projective representations of a quiver, which yields a realization of the Drinfeld double $\tU$ of the quantum group $\U$. The Hall algebra realization 
	has paved ways for fruitful understanding of quantum groups, such as braid group actions and PBW bases; see \cite{Rin3}. Lusztig \cite{Lus90} provided an elementary construction of the canonical basis via the PBW basis, and the Hall algebra interpretation of this construction can be found in \cite{DDPW}.
	
	Motivated by Bridgeland's construction \cite{Br13}, Qin \cite{Qin} established a geometric realization of the Drinfeld double $\tU$ of type ADE using the (dual) quantum Grothendieck rings of cyclic quiver varieties. A key advantage of this geometric approach is the construction of an integral and positive basis for $\tU$, which includes the dual canonical bases (up to mild rescaling) of the half quantum groups $\U^{\pm}$ \cite{Lus90} as subsets. This basis is referred to as the {\em dual canonical basis} in \cite{LW21b,LP25,LP26}. Qin's work builds on Hernandez-Leclerc's construction \cite{HL15} (who realized the half quantum group $\U^+$ via graded Nakajima quiver varieties) and the concept of quantum Grothendieck rings introduced by Nakajima \cite{Na04} and Varagnolo--Vasserot \cite{VV}.
	
	The (universal) $\imath$quantum group $\tUi$ is defined as a subalgebra of $\tU$ associated with a Satake diagram, and the pair $(\tU, \tUi)$ forms a quantum symmetric pair (cf. \cite{Let99, Ko14,LW19}). A central reduction of $\tUi$ yields Letzter's $\imath$quantum group $\Ui =\Ui_\bvs$ with parameters $\bvs$ \cite{Let99,Ko14}. Letzter's $\imath$quantum groups can be regarded as a vast generalization of Drinfeld-Jimbo quantum groups, and quantum groups are $\imath$quantum groups of diagonal type. In \cite{LW19,LW21b}, the first author and Wang introduced the $\imath$Hall algebra and the Grothendieck ring of perverse sheaves, associated with a quiver with involution $(Q,\varrho)$, to realize $\tUi$. In the geometric setting, perverse sheaves further give rise to an integral and positive basis for $\tUi$, also called the dual canonical basis in \cite{LW21b,LP26}.
	
	Recently, we conducted an in-depth study of the dual canonical basis for the quantum group $\tU$ and the $\imath$quantum group $\tUi$ of type ADE from both Hall algebra and quiver variety perspectives (see \cite{LP25,LP26}). In these works, we present a new construction of the dual canonical basis in the framework of $\imath$Hall algebras, introduce Fourier transforms for $\imath$Hall algebras and the Grothendieck ring of perverse sheaves, and prove that the dual canonical basis is invariant under braid group actions and Fourier transforms. We also establish the positivity of the coefficients in the transition matrix from the Hall basis (and PBW basis) to the dual canonical basis. Since quantum groups are $\imath$quantum groups of diagonal type, all these results hold (and are new) for the classical case of quantum groups.
	
	Unlike Lusztig's canonical basis for $\U^+$, the dual canonical basis is defined for the entire $\tU$ and $\tUi$. In fact, to construct canonical bases for the whole quantum groups, Lusztig introduced modified quantum groups $\dot{\U}$ (see \cite{Lus93}). In contrast to Lusztig's approach, Berenstein and Greenstein \cite{BG17} constructed the double canonical basis for $\tU$ by combining Lusztig's dual canonical bases of $\U^{\pm}$ and performing a series of intricate operations (see \S\ref{sec:BG}). They conjecture that the double canonical basis possesses favorable structural properties and propose numerous related conjectures {\em loc. cit.}, such as the positivity of the expansion of the natural product $b_{-}b_{+}$ (where $b_{\pm}$ are dual canonical basis elements of $\U^\pm$) in terms of the double canonical basis, and the invariance of the double canonical basis under Lusztig's braid group actions.
	
	In this paper, we address the coincidence of the dual canonical basis and the double canonical basis for $\tU$. To this end, we reinterpret Berenstein--Greenstein's construction of the double canonical basis using the geometry of NKS quiver varieties associated with the NKS regular singular categories $\mathcal{R}$ and $\mathcal{S}$ (called cyclic quiver varieties in \cite{Qin}). The key observation is that the applications of Lusztig's Lemma in \cite{BG17} correspond to pullbacks along the convolution diagram used to define the restriction functor of $\tR$ (see \eqref{eq:embedding by j and i}). By utilizing a $\C^*$-action on cyclic quiver varieties (as defined in \cite{Na04}), we show that these pullbacks preserve intersection cohomology (IC) sheaves, thereby inducing embeddings of quantum Grothendieck rings. The upper triangular relations in the definition of the double canonical basis then follow naturally from the intrinsic properties of IC sheaves.
	
	\begin{customthm}{{\bf A}} [Theorem~\ref{iCB is DCB}] \label{thm H}
		The double canonical basis of $\tU$ coincides with the dual canonical basis. In particular, it is invariant under Lusztig's braid group actions, and its structure constants lie in $\N[v^{\frac{1}{2}},v^{-\frac{1}{2}}]$.
	\end{customthm}
	
	Using this main result, we answer several conjectures proposed in \cite{BG17} (see Corollaries~\ref{coro: QG double CB braid}, \ref{coro: QG double CB positive} and Proposition~\ref{QG dCB invariant under *}). In particular, the double ($=$dual) canonical basis is invariant under taking the anti-involution $\cdot^*$ and Chevalley involution $\omega$; see Proposition \ref{QG dCB invariant under *}, Corollary \ref{cor:dCB-chevalley}. 
	
	We also provide explicit formulas for the dual canonical basis of $\tU_v(\mathfrak{sl}_2)$ in Section \ref{sec:dCB rank I}, which can be compared with the formula for the double canonical basis of $\tU_v(\mathfrak{sl}_2)$ given in \cite{BG17}.
	
	In \cite{Lus90}, Lusztig gave an elementary construction of the canonical basis of $\U^+$ via the PBW basis, and proved that the entries of the transition matrix from the canonical basis of $\U^+$ to the PBW basis $\{E^\bc\mid \bc\in \N^{|\Phi^+|}\}$ (see Proposition \ref{prop:PBW-QG}) belong to $\N[v^{\frac{1}{2}},v^{-\frac{1}{2}}]$. Inspired by Lusztig's result, we prove that the coefficients of the transition matrix from the PBW basis $\{F^{\ba}E^{\bc}K_\mu K'_{\nu}\mid \ba,\bc\in\N^{|\Phi^+|},\mu,\nu\in\Z^\I\}$ (see \eqref{PBW}) to the double ($=$dual) canonical basis of $\tU$ also lie in $\N[v^{\frac{1}{2}},v^{-\frac{1}{2}}]$ (see Corollary \ref{cor:positive-trans-PBW-dCB}). This result generalizes (the dual version of) Lusztig's result on $\U^+$ to the entire quantum group $\tU$.
	
	It is well known that quantum groups are bialgebras (and even Hopf algebras). Lusztig realized $\U^+$ as a bialgebra using perverse sheaves, and further established that the structure constants of the canonical basis of ${\U}^+$ with respect to the coproduct are integral and positive (see \cite{Lus90}).
	
	From the above discussion, the double ($=$dual) canonical basis of $\widetilde{\U}$ %(and also of $\widehat{\U}$) 
	can be viewed as a natural extension of Lusztig's dual canonical basis for $\U^+$. We propose the following conjecture:
	\begin{conjecture}
		The structure constants of the double (dual) canonical basis with respect to the coproduct also lie in $\mathbb{N}[v^{\frac{1}{2}},v^{-\frac{1}{2}}]$.
	\end{conjecture}
	This conjecture has been verified by direct computation for $\tU_v(\mathfrak{sl}_2)$ in \cite{CZ26}, building on the work of \cite{LP25,LP26} and the present paper. A general proof via direct computation is infeasible; following Lusztig's approach in \cite{Lus90}, we instead turn to geometric realization. However, the geometric realization of quantum groups used in this paper (see \cite{Qin,VV}) only captures the algebra structure, which motivates us to further develop the geometric framework for quantum groups introduced in \cite{Qin,VV}, and in particular, to provide a geometric interpretation of the coproduct for $\tU$.
	
	There have been various connections of quantum groups to cluster algebras in the works \cite{GLS13,HL15, KKKO18,Qin} and references therein. In fact, using the cluster algebra method, Shen \cite{Sh22} constructed a basis of the quantum group $\U$ which is integral, positive, and also invariant under braid group actions. It is interesting to compare this basis with the dual (double) canonical basis as they share these nice properties. For quantum $\mathfrak{sl}_2$, one can see that they coincide with each other (\cite{Sh22}), but it is mysterious for higher rank.

	The paper is organized as follows: Section \ref{sec:QG and iQG} reviews the basics of quantum groups, adopting the presentation in \cite{BG17}; Section \ref{NKS QV section} recapitulates Qin's construction of quantum groups via perverse sheaves on cyclic quiver varieties, using the general language of Nakajima-Keller-Scherozke (NKS) quiver varieties \cite{KS16,Sch19} following \cite{LW21b,LP26}; Section \ref{sec:QG double CB} is dedicated to proving the coincidence of the dual and double canonical bases; explicit formulas for the dual (double) canonical basis of $\tU_v(\mathfrak{sl}_2)$ are given in Section \ref{sec:dCB rank I}.
	
	\vspace{2mm}
	
	\noindent{\bf Acknowledgments.}
	We thank Fan Qin for helpful discussions on quiver varieties. We are grateful to Weiqiang Wang for his collaboration on related projects, as well as for his insightful comments and stimulating discussions. ML is partially supported by the National Natural Science Foundation of China (Grant No. 12171333).
	
	%%%%%%%%%%%%
	\section{Quantum groups}\label{sec:QG and iQG}

	%\subsection{Quantum groups}
	%\label{subsec:QG}
	
	%Let $Q=(Q_0,Q_1)$ be a Dynkin quiver with vertex set $Q_0= \I$.
	%Let $n_{ij}$ be the number of edges connecting vertex $i$ and $j$. 
	Let $\I=\{1,\dots,n\}$ be the index set. 
	Let $C=(c_{ij})_{i,j \in \I}$ be the Cartan matrix of of a simply-laced semi-simple Lie algebra $\fg$.
	Let $\Delta^+=\{\alpha_i\mid i\in\I\}$ be the set of simple roots of $\fg$, and denote the root lattice by $\Z^{\I}:=\Z\alpha_1\oplus\cdots\oplus\Z\alpha_n$. Denote by $\alpha_i^\vee$ ($i\in\I$) the coroots. 
	Let $\Phi^+$ be the set of positive roots. The simple reflection $s_i:\Z^{\I}\rightarrow\Z^{\I}$ is defined to be $s_i(\alpha_j)=\alpha_j-c_{ij}\alpha_i$, for $i,j\in \I$.
	Denote the Weyl group by $W =\langle s_i\mid i\in \I\rangle$.
	
	%Let $\btau$ be an involution of $Q$. We assume that $c_{i,\btau i}=0$ for all $i$, which always hold  for the {\em Dynkin} $\imath$quivers. We denote by $\bs_{i}$ the following element of order 2 in the Weyl group $W$, i.e.,
	%\begin{align}
	%\label{def:simple reflection}
	%\bs_i= \left\{
	%\begin{array}{ll}
	%s_{i}, & \text{ if } \btau i=i;
	%\\
	%s_is_{\btau i}, & \text{ if } \btau i\neq i.
	%\end{array}
	%\right.
	%\end{align}
	%It is well known (cf., e.g., \cite{KP11}) that the {\rm restricted Weyl group} associated to the quasi-split symmetric pair $(\fg, \fg^\theta)$ can be identified with the following subgroup $W_\btau$ of $W$:
	%\begin{align}
	%  \label{eq:Wtau}
	%W_{\btau} =\{w\in W\mid \btau w =w \btau\}
	%\end{align}
	%where $\btau$ is regarded as an automorphism of $\Aut(C)$. Moreover, it admits the following property.

	Let $v$ be an indeterminate. Write $[A, B]=AB-BA$, and $[A,B]_v=AB-vBA$. Denote, for $r,m \in \N$,
	\[
	[r]=\frac{v^r-v^{-r}}{v-v^{-1}},
	\quad
	[r]!=\prod_{i=1}^r [i], \quad \qbinom{m}{r} =\frac{[m][m-1]\ldots [m-r+1]}{[r]!}.
	\]
	Following \cite{BG17}, the Drinfeld double $\hU := \hU_v(\fg)$ is defined to be the $\Q(v^{1/2})$-algebra generated by $E_i,F_i, \tK_i,\tK_i'$, $i\in \I$, %\red{where $\tK_i, \tK_i'$ are invertible,?}
	subject to the following relations:  for $i, j \in \I$,
	\begin{align}
		[E_i,F_j]= \delta_{ij}(v^{-1}-v) (\tK_i-\tK_i'),  &\qquad [\tK_i,\tK_j]=[\tK_i,\tK_j']  =[\tK_i',\tK_j']=0,
		\label{eq:KK}
		\\
		\tK_i E_j=v^{c_{ij}} E_j \tK_i, & \qquad \tK_i F_j=v^{-c_{ij}} F_j \tK_i,
		\label{eq:EK}
		\\
		\tK_i' E_j=v^{-c_{ij}} E_j \tK_i', & \qquad \tK_i' F_j=v^{c_{ij}} F_j \tK_i',
		\label{eq:K2}
	\end{align}
	and for $i\neq j \in \I$,
	\begin{align}
		& \sum_{r=0}^{1-c_{ij}} (-1)^r \left[ \begin{array}{c} 1-c_{ij} \\r \end{array} \right]  E_i^r E_j  E_i^{1-c_{ij}-r}=0,
		\label{eq:serre1} \\
		& \sum_{r=0}^{1-c_{ij}} (-1)^r \left[ \begin{array}{c} 1-c_{ij} \\r \end{array} \right]  F_i^r F_j  F_i^{1-c_{ij}-r}=0.
		\label{eq:serre2}
	\end{align}
	%\red{Note that $\tK_i \tK_i'$ are central in $\hU$ for all $i$. %\brown{Consider the necessity for $K_i,K_i'$ to be invertible.}
	%The comultiplication $\Delta: \widehat{\U} \rightarrow \widehat{\U} \otimes \widehat{\U}$ is given by
	%\begin{align}  \label{eq:Delta}
	%\begin{split}
	%\Delta(E_i)  = E_i \otimes 1 + \tK_i \otimes E_i, & \quad \Delta(F_i) = 1 \otimes F_i + F_i \otimes \tK_{i}', \\
	%\Delta(\tK_{i}) = \tK_{i} \otimes \tK_{i}, & \quad \Delta(\tK_{i}') = \tK_{i}' \otimes \tK_{i}'.
	%\end{split}
	%\end{align}
	%}
	
	We define $\tU=\tU_v(\fg)$ to be the $\Q(v^{1/2})$-algebra constructed from $\widehat{\U}$ by making $\tK_i,\tK_i'$ ($i\in\I$) invertible. Then $\tU$ and $\hU$ are $\Z^\I$-graded by setting $\deg E_i=\alpha_i$, $\deg F_i=-\alpha_i$, $\deg K_i=0=\deg K_i'$. %Moreover,
	%let $\tU_\mu$ be the homogeneous subspace of degree $\mu$.  $\tU=\oplus_{\mu\in\Z^\I} \tU_\mu$ and $\hU=\oplus_{\mu\in\Z^\I} \hU_\mu$. 
	
	\begin{remark}
		Note that $\tU$ defined here is different but isomorphic to the one given in \cite{LW19}. 
		In fact, denote 
		\begin{align}
			\label{eq:Udj-gen}
			\ce_i=\frac{E_i}{v^{-1}-v},\qquad \cf_i=\frac{F_i}{v-v^{-1}},\qquad \forall i\in\I.
		\end{align} 
		Then the presentation of $\tU$ given in \cite{LW19} is generated by $\ce_i,\cf_i,K_i,K_i'$ ($i\in\I$).
	\end{remark}

	The quantum group $\bU$ is defined to be quotient algebra of $\hU$ (also $\tU$) modulo the ideal generated by $K_iK_i'-1$ ($i\in\I$). In fact, $\U$ 
	is 
	the $\Q(v^{1/2})$-algebra generated by $E_i,F_i, K_i, K_i^{-1}$, $i\in \I$, subject to the  relations modified from \eqref{eq:KK}--\eqref{eq:serre2} with $\tK_i'$ replaced by $K_i^{-1}$. %The comultiplication $\Delta$ is obtained by modifying \eqref{eq:Delta} with $\tK_i'$ replaced by  $K_i^{-1}$. 
	
	Let $\U^+$ be the subalgebra of $\U$ (also of $\tU$ and $\hU$) generated by $E_i$ $(i\in \I)$, $\U^0$ be the subalgebra of $\widehat{\bU}$ generated by $\tK_i^{\pm1}$ $(i\in \I)$, and $\U^-$ (also of $\tU$ and $\hU$) be the subalgebra of $\U$ generated by $F_i$ $(i\in \I)$, respectively.
	The Cartan subalgebras $\tU^0$ (resp. $\hU^0$) of $\tU$ (resp. $\hU$) are defined similarly. Then the algebras $\hU$, $\widetilde{\bU}$ and $\bU$ have triangular decompositions:
	\begin{align*}
		\hU=\U^+\otimes\hU^0\otimes \U^-,\qquad 
		\widetilde{\bU} =\U^+\otimes \widetilde{\bU}^0\otimes\U^-,
		\qquad
		\bU &=\bU^+\otimes \bU^0\otimes\bU^-.
	\end{align*}
	Clearly, %${\bU}^+\cong\U^+\cong\U^+$, ${\bU}^-\cong \U^-\cong\U^-$, 
	$\hU^0\cong\Q(v^{1/2})[\tK_i,\tK_i'\mid i\in\I]$, $\tU^0\cong \Q(v^{1/2})[\tK_i^{\pm1},(\tK_i')^{\pm1}\mid i\in\I]$ and $\bU^0\cong\Q(v^{1/2})[K_i^{\pm1}\mid i\in\I]$. Note that 
	${\bU}^0 \cong \hU^0/(\tK_i \tK_i' -1 \mid   i\in \I)$. 
	For any $\mu=\sum_{i\in\I}m_i\alpha_i\in\Z^\I$, we denote by $K_\mu=\prod_{i\in\I} K_i^{m_i}$, $K_\mu'=\prod_{i\in\I} (K_i')^{m_i}$ in $\tU$ (or $\U$); we can view $\tK_\mu,\tK_\mu'$ in $\hU$ if $\mu\in\N^\I$.
	
	\begin{lemma}[cf. \cite{BG17}]\label{QG bar-involution def}
		There exists an anti-involution $u\mapsto \ov{u}$ on $\hU$ (also $\tU$, $\U$) given by $\ov{v^{1/2}}=v^{-1/2}$, $\ov{E_i}=E_i$, $\ov{F_i}=F_i$, and $\ov{K_i}=K_i$, $\ov{K_i'}=K_i'$, for $i\in\I$.
	\end{lemma}

	Let $\Br(W)$ be the braid group associated to the Weyl group $W$, generated by $t_i$ ($i\in\I$).
	Lusztig introduced braid group symmetries $T_{i,e}',T_{i,e}''$ for $i\in\I$ and $e=\pm1$, on the Drinfeld-Jimbo quantum group $\U$ \cite[\S37.1.3]{Lus93}. These braid group symmetries can be lifted to  the Drinfeld double $\tU$; see \cite[Propositions 6.20–6.21]{LW21a}, which are denoted by $\widetilde{T}_{i,e}',\widetilde{T}_{i,e}''$. We shall use the following modified version of $\widetilde{T}_i:=\widetilde{T}_{i,1}''$ (compare with \cite[Theorem 1.13]{BG17}).
	\begin{proposition}
		%[\cite{Lus90a}]
		\label{prop:braid1}
		There exists an automorphism $\widetilde{T}_{i}$, for $i\in \I$, on $\tU$ such that
		\begin{align*}
			&\widetilde{T}_{i}(K_\mu)= K_{s_i(\mu)},
			\qquad \widetilde{T}_{i}(K'_\mu)= K'_{s_i(\mu)},\;\;\forall \mu\in \Z^\I,\\
			&\widetilde{T}_{i}(E_i)=v(K_i')^{-1}F_i,\qquad \widetilde{T}_{i}(F_i)=v^{-1}E_iK_i^{-1},\\
			&\widetilde{T}_i(E_j)=E_j,\qquad \widetilde{T}_i(F_j)=F_j,\qquad \text{ if }c_{ij}=0,
			\\
			&\widetilde{T}_{i}(E_j)=\frac{v^{\frac{1}{2}}E_iE_j-v^{-\frac{1}{2}}E_jE_i}{v-v^{-1}},\qquad  \widetilde{T}_{i}(F_j)=\frac{v^{\frac{1}{2}}F_iF_j-v^{-\frac{1}{2}}F_jF_i}{v-v^{-1}}, \text{ if } c_{ij}=-1. %\qquad  j\neq i.
		\end{align*}
		Moreover, there exists a group homomorphism $\Br(W)\rightarrow \Aut(\tU)$, $t_i\mapsto \widetilde{T}_i$ for $i\in\I$.    
	\end{proposition}
	
	Hence, we can define
	\begin{align}\widetilde{T}_w 
		:= \widetilde{T}_{i_1}\cdots
		\widetilde{T}_{i_r} \in \Aut(\tU),
	\end{align}
	where $w = s_{i_1}
	\cdots s_{i_r}$
	is any reduced expression of $w \in W$.
	
	\begin{lemma}
		\label{lem:QGbraid-bar}
		The braid group actions $\widetilde{T}_{i}$ commute with the bar-involution, i.e., $\ov{\widetilde{T}_i(u)}=\widetilde{T}_i(\ov{u})$ for any $u\in\tU$.
	\end{lemma}
	
	\begin{proof}
		It is enough to check $\ov{\widetilde{T}_i(u)}=\widetilde{T}_i(\ov{u})$ for $u=E_i,F_i,K_i,K_i'$, which is obvious by Proposition \ref{prop:braid1}. 
	\end{proof}

	\begin{proposition}[\text{\cite[Proposition 1.10]{Lus90a}}]
		\label{prop:PBW-QG}
		Let $w_0$ be the longest element of $W$ and fix a reduced expression ${w_0}=s_{i_1} \cdots s_{i_l}$ for the longest element $\omega_0\in W$. Set
		\begin{align}
			\label{eq:F-root}
			F_{\beta_k}=\widetilde{T}_{i_1}^{-1}\widetilde{T}_{i_2}^{-1}\cdots \widetilde{T}_{i_{k-1}}^{-1}(F_{i_k}),\qquad 
			E_{\beta_k}=\widetilde{T}_{i_1}^{-1}\widetilde{T}_{i_2}^{-1}\cdots \widetilde{T}_{i_{k-1}}^{-1}(E_{i_k}),
		\end{align}
		for $1\leq k\leq l,a\in \N$. Then the monomials 
		$$F^{\ba}=v^{\frac{1}{2}n_{\mathbf{a}}}F_{\beta_1}^{a_1} F_{\beta_2}^{a_2}\cdots F_{\beta_l}^{a_l}, \qquad \ba=(a_1,\ldots, a_l)\in \N^l$$ 
		form a $\Q(v^{1/2})$-basis for $\U^-$, where 
		\[n_\ba=\sum_{1\leq k<l\leq N}(\beta_k,\beta_l)a_ka_l.\]
		Similarly, the monomials $$E^{\ba}=v^{\frac{1}{2}n_\bc}E_{\beta_1}^{a_1} E_{\beta_2}^{a_2}\cdots E_{\beta_l}^{a_l}, \qquad \ba=(a_1,\ldots, a_l)\in \N^l$$ 
		form a $\Q(v^{1/2})$-basis for $\U^+$.
		Moreover, the monomials \begin{align}
			\label{PBW}
			F^{\ba}E^{\bc}K_\mu K'_{\nu},\qquad \ba,\bc\in\N^l,\mu,\nu\in\Z^\I
		\end{align}
		form a $\Q(v^{1/2})$-basis for $\tU$. Similar result holds for $\hU$. 
	\end{proposition}
	
	Let $\cz:=\Z[v^{\frac12},v^{-\frac12}]$. Then the integral form $\tU_\cz$ of $\tU$ is defined to be $\cz$-free module with a basis given by \eqref{PBW}. It is known that $\tU_\cz$ is an algebra; cf. \cite[Proposition 4.10]{LP25}.

	%%%%%%%%%%%%%
	\section{Quantum groups via quiver varieties}\label{NKS QV section}
	
	In this section, we review the geometric realization of quantum groups given in \cite{Qin}, also see \cite{LP26}. 
	In this section, we assume $k$ is an algebraically closed field of characteristic zero, and denote by $\mod(k)$ the category of $k$-vector spaces. For a $k$-linear category $\mathcal{C}$, we denote by $\mod(\mathcal{C})$ the category of $\mathcal{C}$–modules, i.e. $k$-linear contravariant functors $\mathcal{C}\to\mod(k)$. Furthermore, for a category $\mathcal{C}$, we denote by $\mathcal{C}_0$ the set of objects of $\mathcal{C}$.
	
	For a quiver $Q=(Q_0,Q_1,\ts,\tt)$, we denote by $kQ$ its path algebra, and by $\mod(kQ)$ the category of finite-dimensional left  $kQ$-modules. The bounded derived category is $\cd_Q=\cd^b(\mod(kQ))$, and its suspension functor is $\Sigma$. The simple modules are $\tS_i$ ($i\in Q_0$).
	
	%%%%%%%
	\subsection{NKS categories}
	\label{subsec:NKS}
	
	Let $Q=(Q_0=\I,Q_1)$ be a Dynkin quiver and set
	\[\ov{Q}_1=\{\bar{h}:j\rightarrow i\mid (h:i\rightarrow j)\in Q_1\}.\] 
	Define the \emph{repetition quiver} $\Z Q$ of $Q$ as follows:
	
	$\triangleright$ the set of vertices is $\{(i,p)\in Q_0\times \Z\}$;
	
	$\triangleright$ an arrow $(\alpha,p):(i,p)\rightarrow (j,p)$ and an arrow $(\bar{\alpha},p):(j,p-1)\rightarrow (i,p)$ are given, for any arrow $\alpha:i\rightarrow j$ in $Q$ and any vertex $(i,p)$.
	
	Define the automorphism $\tau$ of $\Z Q$ to be the shift by one unit to the left, i.e., $\tau(i,p)=(i,p-1)$ for all $(i,p)\in Q_0\times \Z$.
	By a slight abuse of notation, associated to $\beta:y\rightarrow x$ in $\Z Q$, we denote by $\bar{\beta}$ the arrow that runs from $\tau x\rightarrow y$. %For convenience, we denote
	%\begin{align}
	%\label{eq:Omega-def}
	%   \Omega:=\{(\alpha,p)\mid \alpha\in Q_1,p\in\Z\},
	%   \qquad
	%  \ov{\Omega}:=\{(\ov{\alpha},p)\mid \alpha\in Q_1,p\in\Z\}.
	%\end{align}
	Let $k(\Z Q)$ be the \emph{mesh category} of $\Z Q$.
	By a theorem of Happel \cite{Ha2}, there is an equivalence
	\begin{align}
		\label{eq:Happelfunctor}
		H: k(\Z Q)\stackrel{\simeq}{\longrightarrow} \Ind \cd_Q,
	\end{align}
	where $\Ind \cd_Q$ denotes the category of indecomposable objects in the bounded derived category $\cd_Q=\cd^b(\mod(kQ))$.
	Using this equivalence, we label once and for all the vertices of $\Z Q$ by the isoclasses of indecomposable objects of $\cd_Q$. %The functor $H$ is an equivalence if and only if $Q$ is of Dynkin type.
	Note that the action of $\tau$ on $\Z Q$ corresponds to the action of the AR-translation on $\cd_Q$, and this explains the notation $\tau$.

	Let 
	\begin{align}
		\label{eq:C}
		C=
		\{\text{the vertices labeled by }\Sigma^j\tS_i, \text{ for all } i\in Q_0 \text{ and } j\in\Z \}.
	\end{align}
	Let $\Z Q_C$ be the quiver constructed from $\Z Q$ by adding to every vertex $c \in C$ a new object denoted by $\sigma c$ together with arrows $\tau c \rightarrow \sigma c$ and $\sigma c \rightarrow c$; we refer to $\sigma c$, for $c\in C$, as {\em frozen vertices}.

	The \emph{graded NKS category} $\mcr^{\gr}_C$ is defined to be the $k$-linear category with
	\begin{itemize}
		\item[$\triangleright$] objects: the vertices of $\Z Q_C$;
		\item[$\triangleright$] morphisms: $k$-linear combinations of paths in $\Z Q_C$ modulo the ideal spanned by the mesh relations
		$\sum_{\alpha:y\rightarrow x}\alpha\bar{\alpha},$
		where the sum runs through all arrows of $\Z Q_C$ ending at $x\in\Z Q$ (including the new arrow $\sigma x \rightarrow x$ if $x\in C$).
	\end{itemize}
	%These categories were formulated in \cite{KS16, Sch19}; here and below NKS stands for Nakajima-Keller-Scherotzke. The work \cite{KS16} was in turn motivated by \cite{Na01,Na04,HL15,LeP13}; also cf. \cite{Qin}.

	Let $F=\Sigma^2:\cd_Q\rightarrow \cd_Q$.
	Note that the subset $C$ in \eqref{eq:C} is $F$-invariant. The isomorphism $F$ of $\Z Q$ can be uniquely lifted to an isomorphism of $\Z Q_C$  by setting $F(\sigma c)=\sigma(Fc)$ for any $c\in C$, and then the functor $F$ of the mesh category $k(\Z Q)$ can be uniquely lifted to $\mcr^{\gr}_C$, which is also denoted by $F$.
	
	Let
	\[
	\mcr=\mcr_{C,F}:=\mcr^{\gr}_C/F,
	\]
	and let $\cs=\cs_{C,F}$ be the full subcategory of $\mcr$ formed by all $\sigma c$ ($c\in C$), following \cite{Sch19}. Then $\mcr$ and $\cs$ are called the \emph{regular NKS category} and the \emph{singular NKS category} of the pair $(F,C)$. %The quotient category 
	%$$\cp:=\mcr/( \cs ),$$ which is equivalent to $k(\Z Q/F)$ is called the \emph{preprojective NKS category}.  By our assumption, $\cd_Q/F$ is a triangulated category and $\Ind \cd_Q/F\simeq \cp$. 
	
	Let $Q^{\dbl} =Q\sqcup  Q'$,  where $Q'$ is an identical copy of $Q$.
	The {\em double framed quiver} $Q^\sharp$
	is the quiver constructed from $Q^{\rm dbl
	}$ by adding arrows $\varepsilon_i: i\rightarrow i' ,\varepsilon'_i: i'\rightarrow i$ for any $i\in\I$. 
	Let $I^{\sharp}$ be the admissible ideal of $\bfk Q^{\sharp}$ generated by
	\begin{itemize}
		\item
		(Nilpotent relations) $\varepsilon_i \varepsilon_i'$, $\varepsilon_i'\varepsilon_i$ for any $i\in Q_0$;
		\item
		(Commutative relations) $\varepsilon_j' \alpha' -\alpha\varepsilon_i'$, $\varepsilon_j \alpha -\alpha'\varepsilon_i$ for any $(\alpha:i\rightarrow j)\in Q_1$.
	\end{itemize}
	We define the algebra $\Lambda= \bfk Q^{\sharp} \big/ I^{\sharp}$. Then $\cs\cong \proj(\Lambda)$, and then $\mod(\cs)\simeq \mod(\Lambda)$; see \cite[Remark 3.7]{LW21b}.
	
	In Nakajima's original construction \cite{Na01,Na04}, the graded Nakajima category $\mcr^{\gr}$ is defined similarly with
	$C$  the set of all vertices of $\Z Q$. We denote by $\cs^{\gr}$ the full subcategory of $\mcr^\gr$ formed by all $\sigma c$, $c\in \Z Q$. The categories $\mcr^\gr$ and $\cs^\gr$ shall be used in \S\ref{subsec: QV C^*-action}.
	
	%%%%%%%%%
	\subsection{NKS quiver varieties}
	
	Let $\cs$ be a singular NKS category, and $\mcr$ its corresponding regular NKS category. An $\mcr$-module $M$ is \emph{stable}  if the support of $\soc(M)$ is contained in $\cs_0$. 
	
	Let $\bv \in \N^{\mcr_0-\cs_0}$ and $\bw \in \N^{\cs_0}$ be dimension vectors (with finite supports). Denote by $\rep(\bv,\bw,\mcr)$ the variety of $\mcr$-modules of dimension vector $(\bv,\bw)$. Let $\e_x$ denote the characteristic function of $x\in\mcr_0$, which is also viewed as the unit vector supported at $x$. Let $G_\bv:=\prod_{x\in \mcr_0-\cs_0}\GL({\bv(x)}, k)$.

	\begin{definition}
		\label{def:NKS}
		The {\em NKS quiver variety}, or simply {\em NKS variety}, $\cm(\bv,\bw)$ is the quotient $\cs t(\bv,\bw)/G_\bv$,  where $\cs t(\bv,\bw)$ is the subset of $\rep(\bv,\bw,\mcr)$ consisting of all stable $\mcr$-modules of dimension vector $(\bv,\bw)$. Define the affine variety
		\begin{equation}
			\label{eq:M0}
			\cm_0(\bv,\bw)=\cm_0(\bv,\bw, \mcr) :=\rep(\bv,\bw,\mcr)\sslash G_\bv
		\end{equation}
		to be the categorical quotient, whose coordinate algebra is $k[\rep(\bv,\bw,\mcr)]^{G_\bv}$.
	\end{definition}
	Then $\cm(\bv,\bw)$ is a pure dimensional smooth quasi-projective variety; see \cite[Theorem 3.2]{Sch19}. The projection map
	\begin{equation}   \label{eq:pi}
		\pi:\cm(\bv,\bw)\longrightarrow \cm_0(\bv,\bw),
	\end{equation}
	which sends the $G_\bv$-orbit of a stable $\mcr$-module $M$ to the unique closed $G_\bv$-orbit in the closure of $G_\bv M$, is proper; see \cite[Theorem 3.5]{Sch19}. Let $\rep(\bw,\cs)$ be the variety of $\cs$-modules of dimension vector $\bw$. 
	Define the affine variety
	\[
	\cm_0(\bw) =\cm_0(\bw, \mcr) :=\colim\limits_{\bv} \cm_0(\bv,\bw)
	\]
	to be the colimit of $\cm_0(\bv,\bw)$ along the inclusions. By \cite[Lemma 3.14]{LW21b}, we know that 
	$\cm_0(\bw)\cong \rep(\bw,\cs)$ for any dimension vector $\bw$, and identify them below.
	
	Denote by $\cm^{\text{reg}}(\bv,\bw)\subset \cm(\bv,\bw)$ the open subset consisting of the union of closed $G_\bv$-orbits of stable modules, and then
	\[
	\cm_0^{\text{reg}}(\bv,\bw):=\pi(\cm^{\text{reg}}(\bv,\bw))
	\]
	is an open subset of $\cm_0(\bv,\bw)$.

	Given $\bv \in \N^{\mcr_0-\cs_0}$,  we define a quantum Cartan matrix (cf. \cite{Sch19})
	%\red{[this explicit definition appeared right before \cite[Proposition 4.6]{Sch18}] }
	\begin{align}
		\label{def:Cq}
		\begin{split}
			{\mathcal C}_q \bv:  \mcr_0-\cs_0 & \longrightarrow\Z,
			\\
			({\mathcal C}_q\bv)(x)& =\bv(x)+\bv(\tau x) -\sum_{y\rightarrow x}\bv(y),
			\quad \text{ for }x \in\mcr_0-\cs_0,
		\end{split}
	\end{align}
	where the sum runs over all arrows $y\rightarrow x$ of $\mcr$ with $y\in\mcr_0-\cs_0$.
	Given $\bw \in \N^{\cs_0}$, define a dimension vector
	\[
	\sigma^*\bw:\mcr_0-\cs_0\longrightarrow\N,
	\qquad
	x \mapsto
	\begin{cases}
		\bw(\sigma x), & \text{ if } x\in C,
		\\
		0, & \text{otherwise}.
	\end{cases}
	\]
	Given $\bv \in \N^{\mcr_0-\cs_0}$, define the dimension vector
	\[
	\tau^*\bv: \mcr_0-\cs_0 \longrightarrow \N,
	\qquad
	x \mapsto \bv(\tau x).
	\]
	
	By \cite[Proposition 4.6]{Sch19}, %if the fibre of $\pi:\cm(\bv,\bw)\rightarrow \cm_0(\bw)$ is non-empty, then $\sigma^*\bw-{\mathcal C}_q\bv \ge \bf 0$. So 
	we obtain the following more precise form of a stratification of $\cm_0(\bw)$
	\begin{equation}   \label{eqn:stratification}
		\cm_0(\bw)=\bigsqcup_{\bv:\sigma^*\bw-{\mathcal C}_q\bv\geq0} \cm_0^{\text{reg}}(\bv,\bw).
	\end{equation}
	A pair $(\bv,\bw)$ of dimension vectors $\bv \in \N^{\mcr_0-\cs_0}$ and $\bw \in \N^{\cs_0}$ is called $l$\emph{-dominant} if $\sigma^*\bw-{\mathcal C}_q\bv\geq0$.
	
	For the categories $\mathcal{R}^{\gr}$ and $\mathcal{S}^{\gr}$, we can also define NKS quiver varieties
	$\mathcal{M}(\bv,\bw,\mathcal{R}^{\gr})$ and $\mathcal{M}_0(\bv,\bw,\mathcal{R}^{\gr})$.

	%%%%%%%%
	\subsection{Quantum Grothendieck rings}
	\label{subsec: graded Groth ring}

	We review the quantum Grothendieck ring and its convolution product, following \cite{VV}; also see \cite{Na01,Na04, Qin}. 
	
	For any two dimension vectors $\alpha,\beta$ of $\cs$, let $V_{\alpha+\beta}$ be a vector space of graded dimension $\alpha+\beta$. Fix a vector subspace $W_0\subset V_{\alpha+\beta}$ of graded dimension $\alpha$, and let
	\[
	F_{\alpha,\beta}:=\{y\in \rep(\alpha+\beta,\cs)\mid y(W_0)\subset W_0\}
	\]
	be the closed subset of $\rep(\alpha+\beta,\cs)$. 
	Then $y \in F_{\alpha,\beta}$ induces a natural linear map $y': V/W_0 \rightarrow V/W_0$, i.e., $y' \in \rep(\beta,\cs)$.
	Hence we obtain the following convolution diagram
	\[
	\rep(\alpha,\cs)\times \rep(\beta,\cs) \stackrel{p}{\longleftarrow} F_{\alpha,\beta}\stackrel{q}{\longrightarrow}\rep(\alpha+\beta,\cs),
	\]
	where $p(y):=(y|_{W_0},y')$ and $q$ is the natural closed embedding.

	Let $\cd_c(\rep(\alpha,\cs))$ be the derived category of constructible sheaves on $\rep(\alpha,\cs)$. We have the following restriction functor (called comultiplication),
	\begin{align}
		\label{eq:Delta1}
		\widetilde{\Delta}^{\alpha+\beta}_{\alpha,\beta}: \cd_c \big(\rep(\alpha+\beta,\cs) \big) \longrightarrow \cd_c \big(\rep(\alpha,\cs) \big)\times \cd_c \big(\rep(\beta,\cs) \big), \quad F\mapsto p_!q^*(F).
	\end{align}
	
	%For any $(\bv,\bw)$ such that $\cm(\bv,\bw)\neq \emptyset$, we do not know if  $\cm(\bv,\bw)$ is connected or not, however, it is pure dimensional by \cite[Theorem 3.2]{Sch18}. 
	Choose a set $\{\alpha_\bv\}$ such that it parameterizes the connected components 
	of $\cm_0(\bv,\bw,\mathcal{R})$. For any $l$-dominant pair $(\bv,\bw)$, since the restriction of $\pi$ on the regular stratum $\cm_0^{\text{reg}}(\bv,\bw,\mathcal{R})$ is a homeomorphism by \cite[Definition 3.6]{Sch19}, the set $\{\alpha_\bv\}$  naturally parameterizes the connected components of this regular stratum:
	\begin{align}
		\cm_0^{\text{reg}}(\bv,\bw,\mathcal{R})=\bigsqcup_{\alpha_\bv} \cm_0^{\reg;\alpha_\bv}(\bv,\bw,\mathcal{R}).
	\end{align}
	
	Let $\underline{k}_{\cm(\bv,\bw)}$ be the constant sheaf on $\cm(\bv,\bw)$. Denote by $\pi^{\cs}(\bv,\bw)\in \cd_c(\rep(\bw,\cs))$ (or $\pi(\bv,\bw)$ when there is no confusion) the pushforward along $\pi: \cm(\bv,\bw)\rightarrow \cm_0(\bw)\cong \rep(\bw,\cs)$ of $\underline{k}_{\cm(\bv,\bw)}$ with a grading shift:
	\begin{align}
		\label{eq:piPS}
		\pi(\bv,\bw):= \pi_!(\underline{k}_{\cm(\bv,\bw)}) [\dim \cm(\bv,\bw,\mathcal{R})].
	\end{align}
	
	For a strongly $l$-dominant pair $(\bv,\bw)$, let $\cl^{\cs}(\bv,\bw)$ (or $\cl(\bv,\bw)$ when there is no ambiguity) be the intersection cohomology (IC for short) complex associated to the stratum $\cm_0^{\reg}(\bv,\bw,\mathcal{R})$ with respect to the trivial local system, that is,
	\begin{align}
		\label{eq:decomp}
		\cl(\bv,\bw)=\IC(\cm_0^{\reg}(\bv,\bw,\mathcal{R}))=\bigoplus_{\alpha_\bv}\IC(\cm_0^{\reg;\alpha_\bv}(\bv,\bw,\mathcal{R})).
	\end{align}
	%Thanks to the transverse slice theorem \cite{Na01, Na04}  (also see \cite{Qin}), 
	%We expect the $\imath$NKS quiver varieties always satisfy the following property (the main point is that only IC sheaves with trivial local systems can show up in the decomposition). 
	
	%\begin{hypothesis}
	%	\label{hypothesis}
	We have a decomposition
	\begin{equation}
		\label{eqn:decomposition theorem}
		\pi(\bv,\bw)=\sum_{\bv':\sigma^*\bw-{\mathcal C}_q\bv'\geq0,\bv'\leq \bv} a_{\bv,\bv';\bw}(v)\cl(\bv',\bw),
	\end{equation}
	where we denote by $\cf^{\oplus m}[d]$ by $mv^{-d}\cf$ using an indeterminate $v$, for any sheaf $\cf$, $m\in\N$, and $d\in\Z$ (to compare with \cite{LW21b}, set $v=t^{-1}$). 
	Moreover, we have $a_{\bv,\bv';\bw}(v)\in\N[v, v^{-1}]$, $a_{\bv,\bv';\bw}(v^{-1})=a_{\bv,\bv';\bw}(v)$, and $a_{\bv, \bv;\bw}(v) =1$. (Any $f(v)\in \N[v, v^{-1}]$ such that $f(v^{-1})=f(v)$ is called \emph{bar-invariant}.)
	%\end{hypothesis}

	For each $\bw\in \N^{\cs_0}$, the Grothendieck group $K_\bw(\mod(\cs))$ is defined as the free abelian group generated by the perverse sheaves $\cl(\bv,\bw)$ appearing in (\ref{eqn:decomposition theorem}), for various $\bv$. It has a $\Z[v, v^{-1}]$-basis by \eqref{eqn:decomposition theorem}:
	\begin{align}
		\label{eq:bases}
		\{\cl(\bv,\bw) \mid \sigma^*\bw-{\mathcal C}_q\bv\geq0\}.
	\end{align}
	Consider the free $\Z[v, v^{-1}]$-module
	\begin{equation}
		\label{eq:Kgr}
		K^{\mathrm{gr}}(\mod(\cs)) := \bigoplus_\bw K_\bw(\mod(\cs)).
	\end{equation}
	Then $\{\widetilde{\Delta}^{\bw}_{\bw_1,\bw_2} \}$ induces a comultiplication $\widetilde{\Delta}$ on $K^{\mathrm{gr}}(\mod(\cs))$.
	
	Introduce a bilinear form $d(\cdot,\cdot)$ on $\N^{\mcr_0-\cs_0}$ by letting
	\begin{equation}\label{definition:d}
		d\big((\bv_1,\bw_1),(\bv_2,\bw_2) \big)=(\sigma^*\bw_1-{\mathcal C}_q\bv_1)\cdot \tau^* \bv_2+\bv_1\cdot \sigma^*\bw_2,
	\end{equation}
	where $\cdot$ denotes the standard inner product, i.e., $\bv' \cdot \bv'' =\sum_{x\in \mcr_0-\cs_0} \bv'(x)\bv''(x)$.
	\begin{proposition}[\text{\cite[Proposition 4.8]{SS16}; see also \cite[Lemma 4.1]{VV},\cite[eq. (11)]{Qin} }]
		The comultiplication $\widetilde{\Delta}$ is coassociative and given by
		\begin{align}
			\label{eqn:comultiplication}
			\widetilde{\Delta}^{\bw}_{\bw_1,\bw_2} \big(\pi(\bv,\bw) \big)=\bigoplus_{\stackrel{\bv_1+\bv_2=\bv}{\bw_1+\bw_2=\bw}} v^{d((\bv_1,\bw_1),(\bv_2,\bw_2))-d((\bv_2,\bw_2),(\bv_1,\bw_1))}\pi(\bv_1,\bw_1)\boxtimes\pi(\bv_2,\bw_2).
		\end{align}
	\end{proposition}

	Denote
	\begin{align}
		\label{eq:Kgr2}
		\begin{split}
			{\bf R}_\bw(\mod(\cs)) &=\Hom_{\Z[v, v^{-1}]} (K_\bw(\mod(\cs)),\Z[v, v^{-1}]),
			\\
			K^{\mathrm{gr}*}(\mod(\cs)) &=\bigoplus_{\bw\in \N^{\cs_0}}{\bf R}_\bw(\mod(\cs)).
		\end{split}
	\end{align}
	Then as the graded dual of a coalgebra, $K^{\mathrm{gr}*}(\mod(\cs) )$ becomes a $\Z[v, v^{-1}]$-algebra, whose multiplication is denoted by $\ast$.
	Note that $K^{\mathrm{gr}*}(\mod(\cs))$ is a $\N^{\cs_0}$-graded algebra (called the {\em quantum Grothendieck ring}).  It has a distinguished basis
	\begin{align}
		\label{eq:bases dual}
		\{L (\bv,\bw) \mid \sigma^*\bw-{\mathcal C}_q \bv\geq0\},
	\end{align}
	dual to the basis in \eqref{eq:bases}, called the dual IC basis.
	
	By the transversal slice Theorem \cite{Na01, LW21b}, we have
	\begin{equation}
		\label{equation multiplication}
		L(\bv_1,\bw_1)\ast  L(\bv_2,\bw_2)=\sum_{\bv \geq \bv_1+\bv_2} c_{\bv_1,\bv_2}^\bv(v) L(\bv,\bw_1+\bw_2),
	\end{equation}
	with a leading term $c_{\bv_1,\bv_2}^{\bv_1+\bv_2}(v) L(\bv_1+\bv_2,\bw_1+\bw_2)$; moreover, we have
	\begin{align}
		c_{\bv_1,\bv_2}^\bv(v) & \in \N[v, v^{-1}],
		\label{eq:positive}
		\\
		c_{\bv_1,\bv_2}^{\bv_1+\bv_2}(v) &= v^{d((\bv_1,\bw_1),(\bv_2,\bw_2))-d((\bv_2,\bw_2),(\bv_1,\bw_1))}.
		\label{eqn: leading term}
	\end{align}
	
	Recall $\e_x$ denotes the characteristic function of $x\in\mcr_0$.  %In this setting, the $\bw^i,\bv^i$ defined in  \eqref{eq:v^i} are
	Define
	\begin{align}
		%\label{eq:v^i}
		%\begin{split}
		\bw^{i}&=\e_{\sigma \tS_{ i}}+\e_{\sigma\Sigma\tS_{i}},
		\qquad
		\bv^{i}= \sum_{z\in\mcr_0-\cs_0} \dim\cp(\tS_i ,z)\e_z,\qquad
		\bv^{\Sigma i}=\Sigma^*\bv^i.
		%\end{split}
	\end{align}
	Denote
	\begin{align}
		&W^{+}=\bigoplus_{x\in\{\tS_i,i\in Q_0\}}\N \e_{\sigma x},&&
		V^{+}=\bigoplus_{x\in\Ind \mod(kQ),\, x\text{ is not injective}} \N \e_x,\\
		&W^{-}=\Sigma^*W^{+},&&
		V^{-}=\Sigma^*V^{+},\\
		&W^{0}=\bigoplus_{i\in Q_0} \N \bw^{i},&&
		V^{0}=\bigoplus_{i\in Q_0} \N \bv^{i}\oplus \N \bv^{\Sigma i}.
	\end{align}
	
	\begin{lemma}[{\cite{Qin}}]
		Let $\bv\in V^0$, $\bw\in W^0$ be such that $\sigma^*\bw-{\mathcal C}_q \bv=0$. Then we have $(\bv,\bw)\in\bigoplus_{i\in \I}\N(\bv^{i},\bw^{i})+\bigoplus_{i\in \I}\N(\bv^{\Sigma i},\bw^{i})$.
	\end{lemma}

	\begin{lemma} [$l$-dominant pair decomposition; see {\cite{Qin}}]
		\label{lem:Qin decomposition}
		If $(\bv,\bw)$ is an $l$-dominant pair for $\mcr$, then we have a unique decomposition of $(\bv,\bw)$ into $l$-dominant pairs $(\bv^+,\bw^+) \in (V^+, W^+)$, $(\bv^-,\bw^-) \in (V^-, W^-)$, $(\bv^0,\bw^0) \in (V^0, W^0)$ such that
		\[
		(\bv^+, \bw^+) + (\bv^-, \bw^-) + (\bv^0, \bw^0)= (\bv, \bw),
		\qquad
		\sigma^*\bw^0-{\mathcal C}_q \bv^0=0.
		\]
	\end{lemma}

	Let $Q^{\dbl} =Q\sqcup  Q'$,  where $Q'$ is an identical copy of $Q$. We can view $\bw\in\N^{\cs_0}$ as a dimension vector of $Q^{\rm dbl}$ by noting that $\cs\cong \proj(\Lambda)$; see \cite[\S4]{LP26}. With the help of the Euler form $\langle-,-\rangle_{Q^{\dbl}}$ of $kQ^{\dbl}$, we define the bilinear forms $\langle-,-\rangle_{Q^{\rm dbl},a}$ as follows: for any dimension vectors $\bw,\bw' \in \N^{\cs_0}$, let
	\begin{align}
		\langle \bw,\bw'\rangle_{Q^{\rm dbl},a}&= \langle \bw,\bw'\rangle_{Q^{\rm dbl}}-\langle \bw',\bw\rangle_{Q^{\rm dbl}}.\label{eqn: antisymmetric bilinear form}
	\end{align}

	Let us fix a square root $v^{1/2}$ of $v$ once for all, and set  $\cz:=\Z[v^{1/2}, v^{-1/2}]$.
	
	The coalgebra $K^{\mathrm{gr}}(\mod(\cs))$ (cf. \eqref{eq:Kgr}) and its graded dual (cf. \eqref{eq:Kgr2}) up to a base change here in the current setting read as follows:
	\begin{gather}
		K^{\mathrm{gr}}(\mod(\cs)) =\bigoplus_{\bw\in W^{+}+W^{-}} K_\bw(\mod(\cs)),
		\label{eq:KS}
		\\
		\hRZ=\bigoplus_{\bw\in W^{+}+W^{-}} \hR_{\cz,\bw},
		\quad \text{where } \hR_{\cz,\bw}  =   \Hom_{\cz} \big(K_\bw(\mod(\cs)),\cz\big).
		\label{eq:tR}
	\end{gather}
	
	Then $(\hRZ, \cdot)$ is the $\cz$-algebra corresponding to the coalgebra $K^{\mathrm{gr}}(\mod(\cs))$ with the {\em twisted} comultiplication
	\begin{equation}
		\label{eq:tw}
		\{\Res^{\bw}_{\bw^1,\bw^2} :=v^{\frac12\langle \bw^1,\bw^2\rangle_{Q^{\dbl},a}} \widetilde{\Res}^{\bw}_{\bw^1,\bw^2}\};
	\end{equation}
	in practice, the product sign $\cdot$ is often omitted.
	
	We shall need the $\Q(v^{1/2})$-algebra obtained by a base change below:
	\begin{align}
		\label{eq:hRQZ}
		\hR  =\Q(v^{1/2}) \otimes_{\cz} \hRZ.
	\end{align}
	We also set $\hR^+$ to be the submodule of $\hR$ generated by $L(\bv,\bw)$, $(\bv,\bw)\in (V^+,W^+)$. It also becomes an algebra under $\Delta^\bw_{\bw_1+\bw_2}$, and the inclusion $\hR^+\hookrightarrow\hR$ is an algebra homomorphism. Similarly we can define $\hR^-$.
	
	We define $\tRZ$ to be the localization of $\tRZ$ with respect to the multiplicatively closed subset generated by $L(\bv^i,\bw^i)$ and $L(\bv^{\Sigma i},\bw^i)$ ($i\in Q_0$). Define 
	\begin{align}
		\label{eq:tRQZ}
		\tR  =\Q(v^{1/2}) \otimes_{\cz} \tRZ.
	\end{align} 
	
	The following theorem recovers the main result of \cite{Qin}.
	
	\begin{theorem}
		\label{thm:Qin}
		Let $Q$ be a Dynkin $\imath$quiver. Then there exists an isomorphism of $\Q(v^{\frac{1}{2}})$-algebras $\tilde{\kappa}:\tU\stackrel{\simeq}{\rightarrow}\tR$ which sends
		\begin{align*}
			E_i\mapsto L(0,\e_{\sigma\tS_i}),\quad F_i\mapsto L(0,\e_{\sigma \Sigma \ts_i}),\quad K_i\mapsto L(\bv^i,\bw^i),\quad K_i'\mapsto L(\bv^{\Sigma i},\bw^i).
		\end{align*}
	\end{theorem}
	
	We know that $\widetilde{\kappa}$ induces an isomorphism of $\cz$-algebras $\widetilde{\kappa}:\tU_\cz\xrightarrow{\cong}\tR_\cz$; cf. \cite[Corollary 5.4]{LP26}.

	\begin{definition}[Dual canonical basis of quantum groups]
		\label{def:dual CB}
		The dual IC basis $\{\cl(\bv,\bw) \mid \sigma^*\bw-{\mathcal C}_q\bv\geq0\}$ for $\tR_\cz$ is transferred to a basis for $\tU_\cz$ via the isomorphism in Theorem~\ref{thm:Qin}, which are called the {\em dual canonical basis} for $\tU$.
	\end{definition}
	
	The Verdier duality defines an anti-automorphism $\bar{\cdot}$ on $\tR$ (and also $\hR$) by $\ov{L(\bv,\bw)}=L(\bv,\bw)$, $\ov{v}=v^{-1}$. This anti-automorphism coincides with the bar-involution of $\tU$ (and also $\hU$) via the isomorphism $\widetilde{\kappa}$.  
	We can define a $\diamond$-action of the subalgebra generated by $L(\bv^0,\bw^0)$, $(\bv^0,\bw^0)\in ({V^0},{W^0})$ on $\tR$, characterized by
	\begin{equation}\label{eq:diamond for L def}
		L(\bv^0,\bw^0)\diamond L(\bv,\bw) = L(\bv^0+\bv,\bw^0+\bw) = \ov{L(\bv^0,\bw^0)\diamond L(\bv,\bw)}
	\end{equation}
	for any $l$-dominant pair $(\bv,\bw)$; see \cite[(6.3)]{LP26}.
	
	\subsection{An algebraic construction of dual canonical bases}\label{dCB of HA subsec}
	
	In this subsection, we shall review the algebraic construction of dual canonical basis of $\hU$ given in \cite{LP25}. Recall the generators $\ce_i$, $\cf_i$ in \eqref{eq:Udj-gen}. Lusztig constructed canonical basis of $\U^{\pm}$ via the generators $\ce_i,\cf_i$ ($i\in\I$). Let $\mathbf{B}^-$ be Lusztig's canonical basis of $\U^-$ (via the generators $\cf_i$ ($i\in\I$)). Following \cite{Ka91}, we define a Hopf pairing on $\U^-\otimes\U^+$ by
	\begin{align} 
		\label{eq:hopf}(F_i,E_j)_{K}=\delta_{ij}(v-v^{-1}).
	\end{align}
	
	For $b\in\mathbf{B}^-$, we denote by $\delta_b\in\U^+$ the dual basis of $b$ under the pairing $(\cdot,\cdot)_K$. We define a norm function $N:\Z^\I\rightarrow\Z$ by
	\[N(\alpha)=\frac{1}{2}(\alpha,\alpha)_Q-\eta(\alpha).\]
	where $\eta:\Z^\I\rightarrow\Z$ is the augmentation map defined by $\eta(\sum_ia_i\alpha_i)=\sum_ia_i$ (cf. \cite{GLS13,HL15}). The rescaled dual canonical basis of $\U^+$ is then defined to be
	\[\widetilde{\mathbf{B}}^+:=\{v^{\frac{1}{2}N(-\deg(b))}\delta_b\mid b\in\mathbf{B}^-\}.\]
	The rescaled dual canonical basis of $\U^-$ is defined as $\varphi(\widetilde{\mathbf{B}}^+)$, where $\varphi$ is the isomorphism $\U^+\rightarrow\U^-$ defined by $E_i\mapsto F_i$.

	\begin{example}
		For $\tU=\tU_v(\mathfrak{sl}_2)$, we have $\widetilde{\mathbf{B}}^+=\{E_1^n\mid n\in\N\}$ and $\widetilde{\mathbf{B}}^-=\{F_1^n\mid n\in\N\}$.
	\end{example}

	\begin{theorem}[{\cite{HL15, Qin}}]\label{dCB of U^+ by L}
		Under the isomorphism $\kappa^+:\U^+\rightarrow\tR^+$, $E_i\mapsto L(0,\e_{\sigma\tS_i})$, the basis $\{L(\bv,\bw)\mid (\bv,\bw)\in(V^+,W^+)\}$ gets identified with the rescaled dual canonical basis of $\U^+$.
	\end{theorem}
	
	Denote by $\U^-\otimes\U^+$ the tensor algebra. 
	Using the $\imath$Hall algebra realization of $\hU$ in \cite{LW19}, we obtain an embedding of linear spaces
	\[\iota:\U^-\otimes\U^+\hookrightarrow\hU.\]
	We emphasize that $\iota$ is neither an algebra homomorphism nor induced by the triangular decomposition of $\hU$.
	In order to describe $\iota$, we shall recall the skew-derivations $\partial^L_i:\U^\pm\rightarrow \U^\pm$, for $i\in\I$ (cf. \cite{Lus93}). These are linear maps satisfying $\partial_i^L(1)=0$, $\partial_i^L(E_j)=\delta_{ij}=\partial_i^L(F_j)$, and
	\begin{align*}
		&\partial_i^L(fg)=v^{(\alpha_i,\nu)}\partial_i^L(f)g+f\partial_i^L(g),
	\end{align*}
	for any $j\in\I, f\in\U^\pm_{\pm\mu}, g\in\U^\pm_{\pm\nu}$. (Note that the operator $\partial_i^L:\U^-\rightarrow\U^-$ here corresponds to $\partial_i^R:\U^-\rightarrow\U^-$ in \cite{Lus93}.)
	By translating the $\imath$Hall product (see \cite[Proposition 3.3]{LR24}), we can describe the image $\iota(y\otimes x)$ for $x\in\U^+_\nu,y\in\U^-_{-\mu}$. It is the unique element determined by the following relations, for all $i\in\I$, 
	\begin{align}
		&\iota(y\otimes 1)=y,\qquad \iota(1\otimes x)=x,
		\\
		&  E_i\iota(y\otimes x)=\iota(y\otimes E_ix)+(v-v^{-1})v^{-(\alpha_i,\mu-\alpha_i)} K_i'\cdot\iota\big(\partial_{ i}^L(y)\otimes x\big),
		\\
		& F_i\iota(y\otimes x)=\iota(F_iy\otimes E_i)+(v-v^{-1})v^{-(\alpha_i,\nu-\alpha_i)}K_i\cdot\iota\big(y\otimes\partial_{ i}^L(x)\big).
	\end{align}

	In what follows, we shall study the image of elements of the form $F^\ba\otimes E^\bc$ under this embedding, where $F^\ba$ and $E^\bc$ are the PBW basis elements of $\hU$ defined in Proposition~\ref{prop:PBW-QG}.
	
	Recall from Lemma~\ref{QG bar-involution def} that $\hU$ admits a bar-involution determined by $\ov{v}=v^{-1}$, $\ov{E}_i=E_i$, $\ov{F}_i=F_i$ and $\ov{K}_i=K_i$, $\ov{K}'_i=K_i'$ for $i\in\I$.
	
	Let $\Gamma=\N^\I\times\N^\I$ and define
	\[\alpha_{+i}=(\alpha_i,0),\quad \alpha_{-i}=(0,\alpha_i).\] 
	We introduce a new degree function on $\hU$ by setting
	\[\deg_\Gamma(E_i)=\alpha_{+i},\quad \deg_\Gamma(F_i)=\alpha_{-i},\quad \deg_\Gamma(K_i)=\deg_\Gamma(K_i')=\alpha_{+i}+\alpha_{-i}.\]
	It is straightforward to verify that, with respect to this degree function, $\hU$ becomes a $\Gamma$-graded algebra. Using this grading, we define an action $\diamond$ of $\hU^0$ on $\hU$ by
	\begin{align*}
		K_i\diamond x=v^{-\frac{1}{2}\check{\alpha}_i(\deg_\Gamma(x))}K_ix,\quad K_i'\diamond x=v^{\frac{1}{2}\check{\alpha}_i(\deg_\Gamma(x))}K_i'x
	\end{align*}
	for homogeneous $x\in\hU$, where $\check{\alpha}_i\in\Hom_\Z(\Gamma,\Z)$ is defined by $\check{\alpha_i}(\alpha_{\pm i})=\pm c_{ij}$. This action is characterized by the compatibility condition
	\begin{equation}\label{QG diamond action char}
		\ov{K\diamond x}=K \diamond \ov{x},\quad  K\in\hU^0,x\in\hU.
	\end{equation}
	
	The following theorem gives an algebraic characterization of dual canonical basis of $\hU$.
	
	\begin{theorem}[{\cite[Proposition 4.13]{LP25}}]
		For any $\ba,\bc\in\N^l$ and $\alpha,\beta\in\N^\I$, there exists a unique element $C_{\alpha,\beta;\ba,\bc}\in\hU$ such that $\ov{C_{\alpha,\beta;\ba,\bc}}=C_{\alpha,\beta;\ba,\bc}$ and
		\[C_{\alpha,\beta;\ba,\bc}-K_\alpha K'_\beta\diamond \iota(F^\ba\otimes E^\bc)\in\sum v^{-1}\Z[v^{-1}]K_{\alpha'} K'_{\beta'}\diamond \iota(F^{\ba'}\otimes E^{\bc'})\]
		Moreover, $C_{\alpha,\beta;\ba,\bc}=K_\alpha K'_\beta\diamond C_{0,0;\ba,\bc}$ and $\{C_{\alpha,\beta;\ba,\bc}\mid\alpha,\beta\in\N^\I,\ba,\bc\in\N^l\}$ forms a $\Q(v^\frac{1}{2})$-basis of $\hU$, which coincides with the dual canonical basis of $\hU$.
	\end{theorem}
	By making $K_i,K_i'$ invertible, we also know that $\{C_{\alpha,\beta;\ba,\bc}\mid\alpha,\beta\in\Z^\I,\ba,\bc\in\N^l\}$ forms a $\Q(v^\frac{1}{2})$-basis of $\tU$, which coincides with the dual canonical basis of $\tU$.
	As stated in \cite{Qin}, the dual canonical basis of $\hU$ (and also $\tU$) contains the rescaled dual canonical bases $\widetilde{\mathbf{B}}^\pm$ of $\U^\pm$.  
	
	%%%%%%%%%%%%%%%5
	\section{Double canonical basis}\label{sec:QG double CB}
	
	In \cite{BG17} Berenstein and Greenstein defined a basis of $\hU$ from the dual canonical bases of $\U^+$ and $\U^-$, called the double canonical basis. In this section we recall their construction and prove that this basis coincides with the dual canonical basis of $\hU$. 
	
	%%%%%%%%%
	\subsection{Berenstein-Greenstein's constructions}
	\label{sec:BG}
	
	First, following \cite{BG17} we define the quantum Heisenberg algebras $\mathcal{H}^{\pm}$ by 
	\[\mathcal{H}^+=\hU/\langle K_i'\mid i\in\I\rangle,\quad \mathcal{H}^-=\hU/\langle K_i\mid i\in\I\rangle.\]
	Let $\mathbf{K}^+$ (resp. $\mathbf{K}^-$) be the submonoid of $\hU$ generated by the $K_i$ (respectively, the $K_i'$), $i\in\I$. Then we have triangular decompositions
	\[\mathcal{H}^+=\mathbf{K}^+\otimes\U^-\otimes \U^+,\quad \mathcal{H}^-=\mathbf{K}^-\otimes\U^-\otimes \U^+.\]
	which induce natural embeddings of vector spaces
	\begin{align*}
		\iota_+&:\mathcal{H}^+=\mathbf{K}^+\otimes\U^-\otimes \U^+\hookrightarrow \hU=\mathbf{K}^-\otimes(\mathbf{K}^+\otimes\U^-\otimes \U^+),\\
		\iota_-&:\mathcal{H}^-=\mathbf{K}^-\otimes\U^+\otimes \U^-\hookrightarrow \hU=\mathbf{K}^+\otimes(\mathbf{K}^-\otimes\U^+\otimes \U^-)
	\end{align*}
	splitting the canonical projections $\hU\rightarrow\mathcal{H}^+$ and $\hU\rightarrow\mathcal{H}^-$.
	
	Let $\widetilde{\mathbf{B}}^{\pm}$ be the rescaled dual canonical basis of $\hU^{\pm}$ defined in \S\ref{dCB of HA subsec} and recall that bar-involution and $\diamond$-action of $\tU$ defined therein. Note that the $\diamond$-action as well as the bar-involution factors through to a $\mathbf{K}^{\pm}$-action and an anti-involution on $\mathcal{H}^{\pm}$ via the canonical projection $\hU\rightarrow\mathcal{H}^{\pm}$, and (\ref{QG diamond action char}) still holds. 
	
	The following theorems are either proved in \cite{BG17}, or can be deduced from the methods therein.
	
	\begin{theorem}[\text{\cite[Main Theorem 1.3]{BG17}}]
		\label{thm:doubleCB-H+}
		For any $(b_-,b_+)\in\widetilde{\mathbf{B}}^-\times\widetilde{\mathbf{B}}^+$, there is a unique element $b_-\circ b_+\in\mathcal{H}^+$ fixed by $\bar{\cdot}$ and satisfying
		\[b_-\circ b_+-b_-b_+\in\sum v\Z[v]K\diamond(b'_-b'_+)\]
		where the sum is taken over $K\in\mathbf{K}^+\setminus\{1\}$ and $b'_{\pm}\in\widetilde{\mathbf{B}}^{\pm}$ such that $\deg_\Gamma(b_-b_+)=\deg_\Gamma(K)+\deg_\Gamma(b'_-b'_+)$. The basis $\{K\diamond(b_-\circ b_+)\mid K\in\mathbf{K}^+,b_{\pm}\in\widetilde{\mathbf{B}}^{\pm}\}$ is called the double canonical basis of $\mathcal{H}^+$.
	\end{theorem} 
	
	\begin{theorem}
		[cf. \text{\cite[Main Theorem 1.3]{BG17}}]
		\label{thm:doubleCB-H-}
		For any $(b_-,b_+)\in\widetilde{\mathbf{B}}^-\times\widetilde{\mathbf{B}}^+$, there is a unique element $b_+\circ b_-\in\mathcal{H}^-$ fixed by $\bar{\cdot}$ and satisfying
		\[b_+\circ b_--b_+b_-\in\sum v\Z[v]K\diamond(b'_+b'_-)\]
		where the sum is taken over $K\in\mathbf{K}^-\setminus\{1\}$ and $b'_{\pm}\in\widetilde{\mathbf{B}}^{\pm}$ such that $\deg_\Gamma(b_+b_-)=\deg_\Gamma(K)+\deg_\Gamma(b'_+b'_-)$. The basis $\{K\diamond(b_+\circ b_-)\mid K\in\mathbf{K}^-,b_{\pm}\in\widetilde{\mathbf{B}}^{\pm}\}$ is called the double canonical basis of $\mathcal{H}^-$.
	\end{theorem} 
	
	\begin{theorem}[\text{\cite[Main Theorem 1.5]{BG17}}]
		\label{thm:doubleCB-U +}
		For any $(b_-,b_+)\in\widetilde{\mathbf{B}}^-\times\widetilde{\mathbf{B}}^+$, there is a unique element $b_-\bullet b_+\in\hU$ fixed by $\bar{\cdot}$ and satisfying
		\[b_-\bullet b_+-\iota_+(b_-\circ b_+)\in \sum v^{-1}\Z[v^{-1}]K\diamond\iota_+(b'_-\circ b'_+)\]
		where the sum is taken over $K\in\hU^0\setminus\mathbf{K}^+$ and $b'_{\pm}\in\widetilde{\mathbf{B}}^{\pm}$ such that $\deg_\Gamma(b_-b_+)=\deg_\Gamma(K)+\deg_\Gamma(b'_-b'_+)$. The basis $\{K\diamond(b_-\bullet b_+)\mid K\in\hU^0,b_{\pm}\in\widetilde{\mathbf{B}}^{\pm}\}$ is called the positive double canonical basis of $\hU$.
	\end{theorem}
	
	\begin{theorem}[cf. \text{\cite[Main Theorem 1.5]{BG17}}]
		\label{thm:doubleCB-U -}
		For any $(b_-,b_+)\in\widetilde{\mathbf{B}}^-\times\widetilde{\mathbf{B}}^+$, there is a unique element $b_+\bullet b_-\in\hU$ fixed by $\bar{\cdot}$ and satisfying
		\[b_+\bullet b_--\iota_-(b_+\circ b_-)\in \sum v^{-1}\Z[v^{-1}]K\diamond\iota_-(b'_+\circ b'_-)\]
		where the sum is taken over $K\in\hU^0\setminus\mathbf{K}^-$ and $b'_{\pm}\in\widetilde{\mathbf{B}}^{\pm}$ such that $\deg_\Gamma(b_+b_-)=\deg_\Gamma(K)+\deg_\Gamma(b'_+b'_-)$. The basis $\{K\diamond(b_+\bullet b_-)\mid K\in\hU^0,b_{\pm}\in\widetilde{\mathbf{B}}^{\pm}\}$ is called the negative double canonical basis of $\hU$.
	\end{theorem}

	%%%%%%%%%%
	\subsection{$\C^*$-actions on NKS quiver varieties}\label{subsec: QV C^*-action}
	%%%%
	
	According to the results of \cite{Qin}, $\mathcal{M}(\bv,\bw,\mathcal{R})$ and $\mathcal{M}_0(\bv,\bw,\mathcal{R})$ are cyclic quiver varieties considered by Nakajima in \cite{Na04}. In {\em loc. cit.} Nakajima defined a $\C^*$-action on $\mathcal{M}(\bv,\bw,\mathcal{R})$ to relate it with graded quiver varieties. We will recall his construction in this subsection. We use $\Z Q$ to label objects in $k(\Z Q)$, see \S\ref{subsec:NKS} for the definition of the repetition quiver $\Z Q$.
	
	For $\mathbf{v}\in\N^{\mathcal{R}_0^{\gr}-\mathcal{S}_0^{\gr}}$, $\mathbf{w}\in\N^{\mathcal{S}_0^{\gr}}$, we define vector spaces 
	\[V(i,p)=\C^{\bv(i,p)},\quad W(\sigma(i,p))=\C^{\bw(\sigma(i,p))}.\]
	Then an element of $\rep(\bv,\bw,\mathcal{R}^{\gr})$ can be denoted by
	\begin{align*}
		(B,\alpha,\beta)&=\Big(\bigoplus_{h\in Q_1}B_h,\bigoplus_{\bar{h}\in\ov{Q}_1}B_{\bar{h}},\bigoplus_{i\in Q_0}\alpha_i,\bigoplus_{i\in Q_0}\beta_i\Big)\\
		&=\Big(\bigoplus_{h\in Q_1}\bigoplus_{p\in\Z}B_{h,p},\bigoplus_{\bar{h}\in\ov{Q}_1}\bigoplus_{p\in\Z}B_{\bar{h},p},\bigoplus_{i\in Q_0}\bigoplus_{p\in\Z}\alpha_{i,p},\bigoplus_{i\in Q_0}\bigoplus_{p\in\Z}\beta_{i,p}\Big)
	\end{align*}
	where
	\begin{equation}\label{eq:R^gr rep element}
		\begin{gathered}
			B_{h,p}:V(\tt(h),p)\rightarrow V(\ts(h),p),\quad B_{\bar{h},p}:V(\ts(h),p)\rightarrow V(\tt(h),p-1),\\
			\alpha_{i,p}:W(\sigma(i,p))\rightarrow V(i,p-1),\quad \beta_{i,p}:V(i,p)\rightarrow W(\sigma(i,p)).
		\end{gathered}
	\end{equation}
	%Also, the action of an element $g\in G_\mathbf{v}$ on $\rep(\bv,\bw,\mathcal{R}^{\gr})$ is given by
	%\[g\cdot(B,\alpha,\beta)=(gBg^{-1},g\alpha,\beta g^{-1}).\]
	
	Following \cite{Qin}, we choose a height function $\xi:Q_0\rightarrow\Z$ such that $\xi(i)=\xi(j)+1$ whenever there is an arrow from $i$ to $j$. Using this function we can embed $\mathcal{R}_0-\mathcal{S}_0$ into $\mathcal{R}_0^{\gr}-\mathcal{S}_0^{\gr}$ as the subset formed by elements $(i,p)$ such that $p-\xi(i)\in\{0,1,\dots,\tc-1\}$, where $\tc$ is the Coxeter number of the quiver $Q$. Similarly we can view $\mathcal{S}_0$ as a subset of $\mathcal{S}_0^{\gr}$. Via this labeling, elements of $\rep(\bv,\bw,\mathcal{R})$ can also be denoted by 
	\begin{align*}
		(B,\alpha,\beta)=\Big(\bigoplus_{h\in Q_1}B_h,\bigoplus_{\bar{h}\in\ov{Q}_1}B_{\bar{h}},\bigoplus_{i\in Q_0}\alpha_i,\bigoplus_{i\in Q_0}\beta_i\Big).
	\end{align*}
	Here each component of $(B,\alpha,\beta)$ has the same form given by \eqref{eq:R^gr rep element}, but now the indices are written periodically, see Example~\ref{CQV *-action eg} for an illustration.
	
	For $\bv\in\N^{\mathcal{R}_0-\mathcal{S}_0}$ and $\bw\in\N^{\mathcal{S}_0}$, we consider the following $\C^*$-action on $\rep(\bv,\bw,\mathcal{R})$ defined by
	\[t\ast\Big(\bigoplus_{h\in Q_1}B_h,\bigoplus_{\bar{h}\in\ov{Q}_1}B_{\bar{h}},\bigoplus_{i\in Q_0}\alpha_i,\bigoplus_{i\in Q_0}\beta_i\Big)=\Big(\bigoplus_{h\in Q_1}B_h,\bigoplus_{\bar{h}\in\ov{Q}_1}tB_{\bar{h}},\bigoplus_{i\in Q_0}t\alpha_is(t)^{-1},\bigoplus_{i\in Q_0}s(t)\beta_i\Big)\]
	where $s(t)\in G_\bw$ is given by 
	\[s(t)=\bigoplus_{\sigma(i,p)\in\mathcal{S}_0}t^p\Id_{W(\sigma(i,p))}.\]
	This action is compatible with the mesh relations and commutes with the action of $G_\mathbf{v}$. Therefore it induces an action on $\mathcal{M}_0(\mathbf{v},\mathbf{w},\mathcal{R})$. It also preserves the stability condition, hence induces an action on $\mathcal{M}(\mathbf{v},\mathbf{w},\mathcal{R})$. These induced actions are also denoted by $\ast$. The map $\pi:\mathcal{M}(\mathbf{v},\mathbf{w},\mathcal{R})\rightarrow\mathcal{M}_0(\mathbf{v},\mathbf{w},\mathcal{R})$ is equivariant. The following lemma is proved in \cite{Na04}.
	
	\begin{lemma}[{\cite[Lemma 7.1]{Na04}}]
		Let $[B,\alpha,\beta]\in\mathcal{M}(\mathbf{v},\mathbf{w},\mathcal{R})$ and consider the flow $t\ast[B,\alpha,\beta]$ for $t\in\C^*$. It has a limit when $t\rightarrow 0$.
	\end{lemma}
	
	\begin{example}\label{CQV *-action eg}
		Let $Q_0=\{i,j\}$, $Q_1=\{j\rightarrow i\}$, and take $\xi(i)=0$, $\xi(j)=1$. The space $\rep(\bv,\bw,\mathcal{R})$ then has the form
		\begin{center}
			\vspace{-0.5cm}
			\begin{equation*}
				\begin{tikzpicture}[scale=1.5]
					%\fill[opacity=0.5,fill=purple!60] 
					(1.9,0) --(5,3.1) -- (5.1,3) -- (2,-0.1)--(1.9,0);
					%\fill[opacity=0.5,fill=purple!20] (3.4,-0.2) --(7.2,3.6) -- (7.6,3.2) -- (3.8,-0.6)--(3.4,-0.2);
					%\fill[opacity=0.5,fill=purple!60] 
					(3.9,0) --(7,3.1) -- (7.1,3) -- (4,-0.1)--(3.9,0);
					
					%\fill[opacity=0.5,fill=black!60]  
					(1.9,0) --(3,1.1)--(4.1,0)--(4,-0.1)--(3,0.9)--(2,-0.1)--(1.9,0);
					%\fill[opacity=0.5,fill=black!60] 
					(7,3.1) --(7.1,3)--(6,1.9)--(4.9,3)--(5,3.1)--(6,2.1)--(7,3.1);
					
					%\fill[opacity=0.5,fill=purple!60] 
					(4.9,3) --(5,3.1) -- (8.1,0) -- (8,-0.1)--(4.9,3);
					%\fill[opacity=0.5,fill=purple!40] (6.4,3.2) --(6.8,3.6) -- (10.6,-0.2) -- (10.2,-0.6)--(6.4,3.2);
					%\fill[opacity=0.5,fill=purple!60] 
					(6.9,3) --(7,3.1) -- (10.1,0) -- (10,-0.1)--(6.9,3);
					
					\node at (2,0) {\scriptsize$W(\sigma(i,0))$};
					\node at (4,0) {\scriptsize$W(\sigma(i,1))$};
					\node at (3,1) {\scriptsize$V(i,0)$};
					\node at (4,2) {\scriptsize$V(j,1)$};
					\node at (5,3) {\scriptsize$W(\sigma(j,2))$};
					\node at (5,1) {\scriptsize$V(i,1)$};
					\node at (6,2) {\scriptsize$V(j,2)$};
					\node at (7,3) {\scriptsize$W(\sigma(j,3))$};
					\node at (7,1) {\scriptsize$V(i,2)$};
					\node at (8,0) {\scriptsize$W(\sigma(i,0))$};
					\node at (10,0) {\scriptsize$W(\sigma(i,1))$};
					\node at (9,1) {\scriptsize$V(i,0)$};
					\node at (8,2) {\scriptsize$V(j,3)$};
					
					\draw[->] (2.85,0.85) -- (2.15,0.15);
					\draw[->] (3.85,0.11) -- (3.15,0.85);
					\draw[->] (3.85,1.85) -- (3.15,1.15) node[midway,rectangle,fill=white] {\scriptsize$\blue{t}$};
					\draw[->] (4.85,1.15) -- (4.15,1.85);
					\draw[->] (4.85,2.85) -- (4.15,2.15) node[midway,rectangle,fill=white] {\scriptsize$\blue{t^{-1}}$};
					\draw[->] (5.85,2.15) -- (5.15,2.85) node[midway,rectangle,fill=white] {\scriptsize$\blue{t^2}$};
					
					\draw[->] (6.85,2.85) -- (6.15,2.15) node[midway,rectangle,fill=white] {\scriptsize$\blue{t^{-2}}$};
					\draw[->] (5.85,1.85) -- (5.15,1.15) node[midway,rectangle,fill=white] {\scriptsize$\blue{t}$};
					\draw[->] (4.85,0.85) -- (4.15,0.15) node[midway,rectangle,thick, fill=white] {\scriptsize$\blue{t}$};
					
					\draw[->] (8.85,0.85) -- (8.15,0.15);
					\draw[->] (7.85,1.85) -- (7.15,1.15) node[midway,rectangle,fill=white] {\scriptsize$\blue{t}$};
					\draw[->] (8.85,1.15) -- (8.15,1.85);
					\draw[->] (9.85,0.15) -- (9.15,0.85);
					\draw[->] (7.85,2.15) -- (7.15,2.85) node[midway,rectangle,thick, fill=white] {\scriptsize$\blue{t^3}$};
					\draw[->] (6.85,1.15) -- (6.15,1.85) ;
					\draw[->] (7.85,0.15) -- (7.15,0.85) node[midway,rectangle,fill=white] {\scriptsize$\blue{t}$};
				\end{tikzpicture}
			\end{equation*}
		\end{center}
		where the $\ast$-action by $t\in\C^*$ is indicated by the labeling on each arrow. The induced $\C^*$-action on $\mathcal{M}_0(\bw,\mathcal{R})$ can be easily read off:
		\begin{center}
			\vspace{-0.5cm}
			\begin{equation*}
				\begin{tikzpicture}[scale=1.2]
					\node at (0,0) {\scriptsize$W(\sigma(i,0))$};
					\node at (2,0) {\scriptsize$W(\sigma(i,1))$};
					\node at (0,1.5) {\scriptsize$W(\sigma(j,2))$};
					\node at (2,1.5) {\scriptsize$W(\sigma(j,3)$};
					
					\draw[->] (1.3,0) -- (0.7,0);
					\draw[->] (1.3,1.5) -- (0.7,1.5);
					\draw[->] (-0.15,0.2) -- (-0.15,1.3) node[midway,rectangle,thick, fill=white] {\scriptsize$\blue{t^3}$};
					\draw[->] (0.15,1.3) -- (0.15,0.2);
					\draw[->] (1.85,0.2) -- (1.85,1.3) node[midway,rectangle,thick, fill=white] {\scriptsize$\blue{t^3}$};
					\draw[->] (2.15,1.3) -- (2.15,0.2);
				\end{tikzpicture}
			\end{equation*}
			\vspace{-0.85cm}
		\end{center}
	\end{example}
	
	We want to relate the fixed point subvariety of $\mathcal{M}(\mathbf{v},\mathbf{w},\mathcal{R})$ with some graded quiver varieties. For this consider vectors $\mathbf{v}_{\tg}\in\N^{\mathcal{R}_0^{\gr}-\mathcal{S}_0^{\gr}}$ and $\mathbf{w}_{\tg}\in\N^{\mathcal{S}_0^{\gr}}$ such that $\mathbf{w}_{\tg}$ is supported on $\mathcal{S}_0$ (no condition for $\mathbf{v}_{\tg}$). We can then view $\mathbf{w}_{\tg}$ as an element in $\N^{\mathcal{S}_0}$. Now define a vector $\mathbf{v}\in\N^{\mathcal{R}_0-\mathcal{S}_0}$ by
	\[\mathbf{v}(i,p)=\sum_{n\equiv p\mod \texttt{c}} \mathbf{v}_{\tg}(i,n),\quad \forall (i,p)\in\mathcal{R}_0-\mathcal{S}_0.\]
	There is a map $\rep(\mathbf{v}_{\tg},\mathbf{w}_{\tg},\mathcal{R}^{\gr})\rightarrow\rep(\mathbf{v},\mathbf{w},\mathcal{R})$ defined in the obvious way. It is equivariant under the $G_{\mathbf{v}_{\tg}}$-action, where $G_{\mathbf{v}_{\tg}}\rightarrow G_\mathbf{v}$ is an obvious homomorphism, and preserves the stability condition. Therefore, we obtain a morphism 
	\begin{equation}\label{QV morphism from graded}
		\mathcal{M}(\mathbf{v}_{\tg},\mathbf{w}_{\tg},\mathcal{R}^{\gr})\longrightarrow\mathcal{M}(\mathbf{v},\mathbf{w},\mathcal{R}).
	\end{equation}
	Note that $\bw_{\tg}$ is uniquely determined by $\bw$, but $\bv_{\tg}$ is not determined by $\bv$.
	
	\begin{lemma}[{\cite[Lemma 7.3]{Na04}}]\label{QV fixed point}
		A point $[B,\alpha,\beta]\in\mathcal{M}(\mathbf{v},\mathbf{w},\mathcal{R})$ is fixed by the $\ast$-action if and only if it is contained in the image of (\ref{QV morphism from graded}). Moreover, the map (\ref{QV morphism from graded}) is a closed embedding.
	\end{lemma}
	\begin{proof}
		Fix a representative $(B,\alpha,\beta)$ of $[B,\alpha,\beta]$. Then $[B,\alpha,\beta]$ is fixed if and only if there exists, for each $t\in\C^*$, an element $\lambda(t)\in G_\mathbf{v}$ such that 
		\[t\ast(B,\alpha,\beta)=\lambda(t)^{-1}\cdot(B,\alpha,\beta).\]
		Such $\lambda(t)$ is unique since the action of $G_\mathbf{v}$ is free. In particular, $\lambda:\C^*\rightarrow G_\mathbf{v}$ is a group homomorphism.
		
		Let $V(i,p)[n]$ be the weight space of $V(i,p)$ with eigenvalue $t^n$. The above equation means that
		\begin{gather*}
			B_{h,p}(V(\tt(h),p)[n])\subseteq V(\ts(h),p)[n],\quad B_{\bar{h},p}(V(\ts(h),p)[n])\subseteq V(\textup{t}(h),p-1)[n-1],\\
			\alpha_{i,p}(W(\sigma(i,p)))\subseteq V(i,p-1)[p-1],\quad \beta_{i,p}(V(i,p)[n])=0\quad\text{if $n\neq p$}.
		\end{gather*}
		Let us define a subspace $S$ of $V$ by
		\[S(i,p)=\bigoplus_{n\not\equiv p\mod \tc}V(i,p)[n].\]
		Then $S(i,p)$ gives a sub-$\mathcal{R}$-module of $(B,\alpha,\beta)$ supported on $\mathcal{R}_0-\mathcal{S}_0$, so $S=0$ by stability condition. This means that $[B,\alpha,\beta]$ is in the image of (\ref{QV morphism from graded}) if we set 
		\begin{equation}\label{eq: QV morphism from graded bv def}
			\mathbf{v}_{\tg}(i,n)=\dim V(i,n)[n].
		\end{equation}
		Conversely, a point in the image is a fixed point. Since $\lambda$ is unique, the map (\ref{QV morphism from graded}) is injective.
		Finally, the fact that (\ref{QV morphism from graded}) is a closed embedding can be seen by calculating the differential at $(B,\alpha,\beta)$; see \cite[Lemma 7.3]{Na04} for details.
	\end{proof}
	
	Let $\mathcal{M}(\mathbf{v},\mathbf{w},\mathcal{R})^{\ast,\C^*}$ be the fixed pointed subvariety of the $\ast$-action. From the proof of Lemma~\ref{QV fixed point} we see that different $\mathcal{M}(\mathbf{v}_{\tg},\mathbf{w}_{\tg},\mathcal{R}^{\gr})$ have disjoint image under (\ref{QV morphism from graded}), i.e.
	\[\mathcal{M}(\mathbf{v},\mathbf{w},\mathcal{R})^{\ast,\C^*}=\bigsqcup_{\bv_{\tg}}\mathcal{M}(\mathbf{v}_{\tg},\mathbf{w}_{\tg},\mathcal{R}^{\gr}).\]
	The graded quiver varieties $\mathcal{M}(\mathbf{v}_{\tg},\mathbf{w}_{\tg},\mathcal{R}^{\gr})$ are connected by \cite[Proposition 7.3.4]{Na01}, so they constitute the connected components of $\mathcal{M}(\mathbf{v},\mathbf{w},\mathcal{R})^{\ast,\C^*}$. Now consider the Bialynicki-Birula decomposition of $\mathcal{M}(\mathbf{v},\mathbf{w},\mathcal{R})$:
	\begin{equation}\label{eq:BB decomposition}
		\mathcal{M}(\mathbf{v},\mathbf{w},\mathcal{R})=\bigsqcup_{\mathbf{v}_{\tg}} S(\mathbf{v}_{\tg},\mathbf{w}_{\tg})
	\end{equation}
	where we define
	\[S(\mathbf{v}_{\tg},\mathbf{w}_{\tg}):=\{x\in\mathcal{M}(\mathbf{v},\mathbf{w},\mathcal{R})\mid \lim_{t\rightarrow 0}t\ast x\in \mathcal{M}(\mathbf{v}_{\tg},\mathbf{w}_{\tg},\mathcal{R}^{\gr})\}\]
	By the general theory, each $S(\mathbf{v}_{\tg},\mathbf{w}_{\tg})$ is a locally closed subvariety of $\mathcal{M}(\mathbf{v},\mathbf{w},\mathcal{R})$, and the natural map $S(\mathbf{v}_{\tg},\mathbf{w}_{\tg})\rightarrow\mathcal{M}(\mathbf{v}_{\tg},\mathbf{w}_{\tg},\mathcal{R}^{\gr})$ is an affine fiber bundle of rank, say, $r(\mathbf{v}_{\tg},\mathbf{w}_{\tg})$.
	
	\begin{lemma}\label{QV fixed point restriction on +}
		Let $i^+:\mathcal{M}_0(\mathbf{w}_{\tg},\mathcal{R}^{\gr})\rightarrow\mathcal{M}_0(\mathbf{w},\mathcal{R})$ be the natural inclusion. Then
		\begin{align*}
			(i^+)^*(\pi(\mathbf{v},\mathbf{w}))&=\bigoplus_{\mathbf{v}_{\tg}}\pi(\mathbf{v}_{\tg},\mathbf{w}_{\tg})[-a(\bv,\bw;\bv_{\tg},\bw_{\tg})],\\
			(i^+)^!(\pi(\mathbf{v},\mathbf{w}))&=\bigoplus_{\mathbf{v}_{\tg}}\pi(\mathbf{v}_{\tg},\mathbf{w}_{\tg})[a(\bv,\bw;\bv_{\tg},\bw_{\tg})],
		\end{align*}
		where $a(\bv,\bw;\bv_{\tg},\bw_{\tg})=\dim\mathcal{M}(\mathbf{v},\mathbf{w},\mathcal{R})-\dim\mathcal{M}(\mathbf{v}_{\tg},\mathbf{w}_{\tg},\mathcal{R}^{\gr})-2r(\bv_{\tg},\bw_{\tg})$.
	\end{lemma}
	
	\begin{proof}
		Let $p:\mathcal{M}_0(\mathbf{w},\mathcal{R})\rightarrow \mathcal{M}_0(\mathbf{w}_{\tg},\mathcal{R}^{\gr})$ be the map defined by taking limit $t\rightarrow 0$. Then by \cite[Theorem 2.10.3]{Ach21} (note that the $\C^\ast$-action on $\mathcal{M}_0(\bw,\mathcal{R})$ is attracting, and each $\pi(\bv,\bw)$ is equivariant under the $\ast$-action), we have
		\[(i^+)^!(\pi(\mathbf{v},\mathbf{w}))\cong p_!(\pi(\mathbf{v},\mathbf{w})),\quad (i^+)^*(\pi(\mathbf{v},\mathbf{w}))\cong p_*(\pi(\mathbf{v},\mathbf{w})).\]
		Now consider the diagram
		\[\begin{tikzcd}
			\bigsqcup\mathcal{M}(\bv_{\tg},\bw_{\tg},\mathcal{R}^{\gr})\ar[d,swap,"\pi_{\bv_{\tg},\bw_{\tg}}"]\ar[r,hook] &\mathcal{M}(\bv,\bw,\mathcal{R})\ar[d,"\pi_{\bv,\bw}"]\\
			\mathcal{M}_0(\bv_{\tg},\bw_{\tg},\mathcal{R}^{\gr})\ar[r,hook]&\mathcal{M}_0(\bv,\bw,\mathcal{R})
		\end{tikzcd}\]
		Let $p_{\bv_{\tg}}:S(\bv_{\tg},\bw_{\tg})\rightarrow \mathcal{M}(\bv_{\tg},\bw_{\tg},\mathcal{R}^{\gr})$ denote the projection map, which is an affine fiber bundle. Then by applying \cite[\S 8.1.6]{Lus93} to \eqref{eq:BB decomposition}, we see that 
		\begin{align*}
			p_!(\pi_{\bv,\bw})_!(\underline{k}_{\mathcal{M}(\bv,\bw,\mathcal{R})})&=\bigoplus_{\bv_{\tg}}(\pi_{\bv_{\tg},\bw_{\tg}}\circ p_{\bv_{\tg}})_!(\underline{k}_{\mathcal{M}(\bv,\bw,\mathcal{R})})=\bigoplus_{\bv_{\tg}}(\pi_{\bv_{\tg},\bw_{\tg}})_!\circ (p_{\bv_{\tg}})_!(\underline{k}_{S(\bv_{\tg},\bw_{\tg})})\\
			&=\bigoplus_{\bv_{\tg}}(\pi_{\bv_{\tg},\bw_{\tg}})_!(\underline{k}_{\mathcal{M}(\bv_{\tg},\bw_{\tg},\mathcal{R}^{\gr})})[-2r(\bv_{\tg},\bw_{\tg})].
		\end{align*}
		This proves the assertion for $(i^+)^!(\pi(\mathbf{v},\mathbf{w}))$. The case for $(i^+)^*(\pi(\mathbf{v},\mathbf{w}))$ can be deduced by taking Verdier duality.
	\end{proof}
	
	For any dimension vector $\mathbf{w}\in\N^{\mathcal{S}_0}$ as above, we denote by $\mathcal{K}_\mathbf{w}$ the category of direct sums of shifts of $\mathcal{L}(\mathbf{v},\mathbf{w})$ for $(\bv,\bw)$ strongly $l$-dominant, which are complexes on $\mathcal{M}_0(\mathbf{w},\mathcal{R})$. We denote by $\mathcal{K}_\mathbf{w}^+$ the similar category over $\mathcal{M}_0(\mathbf{w}_{\tg},\mathcal{R}^{\gr})$ (recall that $\bw_{\tg}$ is uniquely determined by $\bw$, whence the notation). 
	
	\begin{lemma}\label{QV fixed point for affine}
		Under the isomorphism $\mathcal{M}_0(\bw,\mathcal{R})\cong\rep(\bw,\mathcal{S})$, the $\ast$-action acts by 
		\begin{equation}\label{QV ast action on affine}
			t\ast (x_h,\eps_i,\eps_{i^\diamond})=(x_h,t^{\tc}\eps_i,\eps_{i^\diamond}).
		\end{equation}
		Therefore $\mathcal{M}_0(\bw,\mathcal{R})^{\ast,\C^*}=\mathcal{M}_0(\bw_{\tg},\mathcal{R}^{\gr})\cong \rep(\mathbf{w}_{\tg},\Lambda^+)$, where $\Lambda^+$ is quotient algebra of $\Lambda$ by the ideal generated by $(\varepsilon_{i}\mid i\in Q_0)$. Moreover, under the decomposition $(\bv,\bw)=(\bv^0,\bw^0)+(\bv^+,\bw^+)+(\bv^-,\bw^-)$ of Lemma~\ref{lem:Qin decomposition}, we have
		\begin{align}
			\label{eq:decompRgr+}
			\mathcal{M}_0(\mathbf{w}_{\tg},\mathcal{R}^{\gr})=\bigsqcup_{(\bv^0,\bw^0)\in\bigoplus_i\N(\bv^i,\bw^i)}\mathcal{M}_0^{\reg}(\bv,\mathbf{w},\mathcal{R}).
		\end{align}
	\end{lemma}
	\begin{proof}
		By \cite{Sch19}, the coordinate ring of $\mathcal{M}_0(\bw,\mathcal{R})$ is generated by coordinate maps to paths from a frozen vertex to a frozen vertex. Now consider such a path from $W(\sigma(j,k+N+1))$ to $W(\sigma(i,k))$, indicated as follows:
		\[\begin{tikzcd}
			W(\sigma(j,k+N+1))\ar[r]&V(j,k+N)\ar[r]&\cdots\ar[r]&V(i,k)\ar[r]&W(\sigma(i,k))
		\end{tikzcd}\]
		For convenience we may assume that $(i,k)\in\mathcal{R}_0-\mathcal{S}_0$. Then under the $\ast$-action, any coordinate function of the path is multiplied by
		\[t^{-r+1}t^{N}t^{k}=t^{k+N+1-r}\]
		where $r$ is the unique integer such that $(j,r)\in\mathcal{R}_0-\mathcal{S}_0$. From the definition of $r$ and our assumption on $(i,k)$, we see that $k+N+1-r\geq 0$ and is a multiple of $\tc$. However, for such a path in $\mathcal{R}$ to be nonzero, the only possibilities are $k+N+1-r=0$ or $\tc$, so we get \eqref{QV ast action on affine}. The rest of the lemma is now clear.
	\end{proof}
	
	In particular, we see that $(i^+)_*(\mathcal{L}(\bv_{\tg},\bw_{\tg}))\in\mathcal{K}_{\bw}$ for each strongly $l$-dominant pair $(\bv_{\tg},\bw_{\tg})$. On the other hand, from Lemma~\ref{QV fixed point restriction on +} we also have $(i^+)^*(\mathcal{L}(\bv,\bw))\in\mathcal{K}_{\bw}^+$ for any strongly $l$-dominant pair $(\bv,\bw)$. Since $(i^+)_*$ is fully faithful, we can (and will) identify $\mathcal{K}_\mathbf{w}^+$ with its image under $(i^+)_*$. The following result is immediate.
	
	\begin{proposition}\label{dCB on + char}
		For any strongly $l$-dominant pair $(\bv,\bw)$ for $\mathcal{R}$, the complex $\mathcal{L}(\mathbf{v},\mathbf{w})$ belongs to $\mathcal{K}_\mathbf{w}^+$ if and only if $(\mathbf{v}^0,\mathbf{w}^0)\in\bigoplus_i\N(\mathbf{v}^i,\mathbf{w}^i)$ for the decomposition $(\mathbf{v},\mathbf{w})=(\mathbf{v}^0,\mathbf{w}^0)+(\mathbf{v}^+,\mathbf{w}^+)+(\mathbf{v}^-,\mathbf{w}^-)$.
	\end{proposition}
	
	We also need a variation of the $\C^*$-action $\ast$. For this, recall the shift functor $\Sigma$ induces an isomorphism on quiver varieties:
	\begin{equation}\label{QV shift functor automorphism}
		\begin{tikzcd}
			\mathcal{M}(\mathbf{v},\mathbf{w},\mathcal{R})\ar[r,"\Sigma^*"]\ar[d]&\mathcal{M}(\Sigma^*(\mathbf{v}),\Sigma^*(\mathbf{w}),\mathcal{R})\ar[d]\\
			\mathcal{M}_0(\mathbf{v},\mathbf{w},\mathcal{R})\ar[r,"\Sigma^*"]&\mathcal{M}_0(\Sigma^*(\mathbf{v}),\Sigma^*(\mathbf{w}),\mathcal{R})
		\end{tikzcd}
	\end{equation}
	Using this we can define a new $\C^*$-action on $\mathcal{M}(\mathbf{v},\mathbf{w},\mathcal{R})$ such that
	\[\Sigma^*(t\star[B,\alpha,\beta])=t\ast\Sigma^*([B,\alpha,\beta]).\]
	One can consider $\star$ as the conjugation of $\ast$ under $\Sigma^*$.
	
	\begin{example}
		Retain the notations in Example~\ref{CQV *-action eg}. Then the $\star$-action on $\mathcal{M}_0(\bv,\bw,\mathcal{R})$ is given by
		
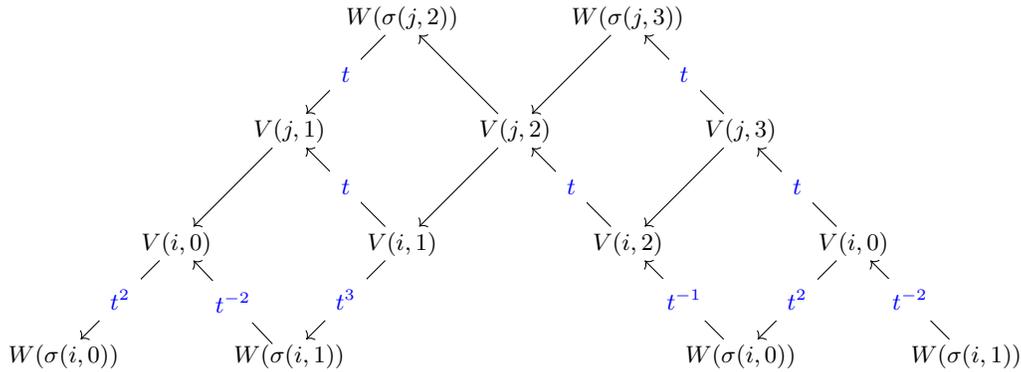
\begin{center}
			\vspace{-0.5cm}
			\begin{equation*}
				\begin{tikzpicture}[scale=1.5]
					%\fill[opacity=0.5,fill=purple!60] 
					(1.9,0) --(5,3.1) -- (5.1,3) -- (2,-0.1)--(1.9,0);
					%\fill[opacity=0.5,fill=purple!20] (3.4,-0.2) --(7.2,3.6) -- (7.6,3.2) -- (3.8,-0.6)--(3.4,-0.2);
					%\fill[opacity=0.5,fill=purple!60] 
					(3.9,0) --(7,3.1) -- (7.1,3) -- (4,-0.1)--(3.9,0);
					
					%\fill[opacity=0.5,fill=black!60]  
					(1.9,0) --(3,1.1)--(4.1,0)--(4,-0.1)--(3,0.9)--(2,-0.1)--(1.9,0);
					%\fill[opacity=0.5,fill=black!60] 
					(7,3.1) --(7.1,3)--(6,1.9)--(4.9,3)--(5,3.1)--(6,2.1)--(7,3.1);
					
					%\fill[opacity=0.5,fill=purple!60] 
					(4.9,3) --(5,3.1) -- (8.1,0) -- (8,-0.1)--(4.9,3);
					%\fill[opacity=0.5,fill=purple!40] (6.4,3.2) --(6.8,3.6) -- (10.6,-0.2) -- (10.2,-0.6)--(6.4,3.2);
					%\fill[opacity=0.5,fill=purple!60] 
					(6.9,3) --(7,3.1) -- (10.1,0) -- (10,-0.1)--(6.9,3);
					
					\node at (2,0) {\scriptsize$W(\sigma(i,0))$};
					\node at (4,0) {\scriptsize$W(\sigma(i,1))$};
					\node at (3,1) {\scriptsize$V(i,0)$};
					\node at (4,2) {\scriptsize$V(j,1)$};
					\node at (5,3) {\scriptsize$W(\sigma(j,2))$};
					\node at (5,1) {\scriptsize$V(i,1)$};
					\node at (6,2) {\scriptsize$V(j,2)$};
					\node at (7,3) {\scriptsize$W(\sigma(j,3))$};
					\node at (7,1) {\scriptsize$V(i,2)$};
					\node at (8,0) {\scriptsize$W(\sigma(i,0))$};
					\node at (10,0) {\scriptsize$W(\sigma(i,1))$};
					\node at (9,1) {\scriptsize$V(i,0)$};
					\node at (8,2) {\scriptsize$V(j,3)$};
					
					\draw[->] (2.85,0.85) -- (2.15,0.15) node[midway,rectangle,fill=white] {\scriptsize$\blue{t^2}$};
					\draw[->] (3.85,0.11) -- (3.15,0.85) node[midway,rectangle,fill=white] {\scriptsize$\blue{t^{-2}}$};
					\draw[->] (3.85,1.85) -- (3.15,1.15);
					\draw[->] (4.85,1.15) -- (4.15,1.85) node[midway,rectangle,fill=white] {\scriptsize$\blue{t}$};
					\draw[->] (4.85,2.85) -- (4.15,2.15) node[midway,rectangle,fill=white] {\scriptsize$\blue{t}$};
					\draw[->] (5.85,2.15) -- (5.15,2.85);
					
					\draw[->] (6.85,2.85) -- (6.15,2.15);
					\draw[->] (5.85,1.85) -- (5.15,1.15);
					\draw[->] (4.85,0.85) -- (4.15,0.15) node[midway,rectangle,thick, fill=white] {\scriptsize$\blue{t^3}$};
					
					\draw[->] (8.85,0.85) -- (8.15,0.15) node[midway,rectangle,fill=white] {\scriptsize$\blue{t^2}$};
					\draw[->] (7.85,1.85) -- (7.15,1.15);
					\draw[->] (8.85,1.15) -- (8.15,1.85) node[midway,rectangle,thick, fill=white] {\scriptsize$\blue{t}$};
					\draw[->] (9.85,0.15) -- (9.15,0.85) node[midway,rectangle,fill=white] {\scriptsize$\blue{t^{-2}}$};
					\draw[->] (7.85,2.15) -- (7.15,2.85) node[midway,rectangle,thick, fill=white] {\scriptsize$\blue{t}$};
					\draw[->] (6.85,1.15) -- (6.15,1.85) node[midway,rectangle,thick, fill=white] {\scriptsize$\blue{t}$};
					\draw[->] (7.85,0.15) -- (7.15,0.85) node[midway,rectangle,thick, fill=white] {\scriptsize$\blue{t^{-1}}$};
				\end{tikzpicture}
			\end{equation*}
		\end{center}
		and the induced action on $\mathcal{M}_0(\bw,\mathcal{R})$ has the form
		\begin{center}
			\vspace{-0.4cm}
			\begin{equation*}
				\begin{tikzpicture}[scale=1.2]
					\node at (0,0) {\scriptsize$W(\sigma(i,0))$};
					\node at (2,0) {\scriptsize$W(\sigma(i,1))$};
					\node at (0,1.5) {\scriptsize$W(\sigma(j,2))$};
					\node at (2,1.5) {\scriptsize$W(\sigma(j,3)$};
					
					\draw[->] (1.3,0) -- (0.7,0);
					\draw[->] (1.3,1.5) -- (0.7,1.5);
					
					\draw[->] (-0.15,0.2) -- (-0.15,1.3);
					\draw[->] (0.15,1.3) -- (0.15,0.2) node[midway,rectangle,thick, fill=white] {\scriptsize$\blue{t^3}$};
					\draw[->] (1.85,0.2) -- (1.85,1.3);
					\draw[->] (2.15,1.3) -- (2.15,0.2) node[midway,rectangle,thick, fill=white] {\scriptsize$\blue{t^3}$};
				\end{tikzpicture}
			\end{equation*}
			\vspace{-0.85cm}
		\end{center}
	\end{example}
	
	For any vectors $\ov{\bv}_{\tg}\in\N^{\mathcal{R}^{\gr}_0-\mathcal{S}^{\gr}_0}$ and $\ov{\bw}_{\tg}\in\N^{\mathcal{S}^{\gr}_0}$ such that $\ov{\bw}_{\tg}$ is supported on $\mathcal{S}_0$, by \eqref{QV morphism from graded} there is an embedding
	\[\mathcal{M}(\ov{\bv}_{\tg},\ov{\bw}_{\tg},\mathcal{R}^{\gr})\longrightarrow \mathcal{M}(\Sigma^*(\mathbf{v}),\Sigma^*(\mathbf{w}),\mathcal{R})\]
	where $\Sigma^*(\bw)$ corresponds to $\ov{\bw}_{\tg}$ and $\Sigma^*(\bv)$ is defined as \eqref{eq: QV morphism from graded bv def}. Composing with the isomorphism $(\Sigma^{-1})^*$ gives us a new embedding
	\begin{equation}\label{QV morphism from graded shifted}
		\mathcal{M}(\ov{\bv}_{\tg},\ov{\bw}_{\tg},\mathcal{R}^{\gr})\longrightarrow \mathcal{M}(\mathbf{v},\mathbf{w},\mathcal{R}),
	\end{equation}
	From the results of Lemma~\ref{QV fixed point} and Lemma~\ref{QV fixed point for affine}, it is clear that 
	\[\cm(\bv,\bw,\mcr)^{\star,\C^*}=\bigsqcup_{\ov{\bv}_{\tg}}\mathcal{M}(\ov{\bv}_{\tg},\ov{\bw}_{\tg},\mathcal{R}^{\gr}),\]
	and the $\star$-action on $\mathcal{M}_0(\bw,\mathcal{R})$ is given by 
	\[t\star(x_h,\eps_i,\eps_{i^\diamond})=(x_h,\eps_i,t^{\tc}\eps_{i^\diamond}).\]
	In particular, we have $\cm_0(\bw,\mcr)^{\star,\C^*}=\cm_0(\ov{\bw}_{\tg},\mathcal{R}^{\gr})\cong\rep(\ov{\bw}_{\tg},\Lambda^-)$ where $\Lambda^-$ is quotient algebra of $\Lambda$ by $(\varepsilon_{i^\diamond}\mid i\in Q_0)$, and
	\[
	\cm_0(\ov{\bw}_{\tg},\mathcal{R}^{\gr})=\bigsqcup_{(\bv^0,\bw^0)\in\bigoplus_i\N(\bv^{\Sigma i},\bw^i)}\mathcal{M}_0^{\reg}(\bv,\bw,\mathcal{R})
	\]
	
	\begin{lemma}\label{QV fixed point restriction on -}
		Let $i^-:\mathcal{M}_0(\ov{\mathbf{w}}_{\tg},\mathcal{R}^{\gr})\rightarrow\mathcal{M}_0(\mathbf{w},\mathcal{R})$ be the natural inclusion. Then
		\begin{align*}
			(i^-)^*(\pi(\mathbf{v},\mathbf{w}))&=\bigoplus_{\ov{\mathbf{v}}_{\tg}}\pi(\ov{\mathbf{v}}_{\tg},\ov{\mathbf{w}}_{\tg})[-a(\Sigma^*(\bv),\Sigma^*(\bw);\ov{\bv}_{\tg},\ov{\bw}_{\tg})],\\
			(i^-)^!(\pi(\mathbf{v},\mathbf{w}))&=\bigoplus_{\ov{\mathbf{v}}_{\tg}}\pi(\ov{\mathbf{v}}_{\tg},\ov{\mathbf{w}}_{\tg})[a(\Sigma^*(\bv),\Sigma^*(\bw);\ov{\bv}_{\tg},\ov{\bw}_{\tg})].
		\end{align*}
	\end{lemma}
	\begin{proof}
		Since $i^-$ is noting but the composition of $i^+$ and $\Sigma^*$, these can be deduced from Lemma~\ref{QV fixed point restriction on +}.
	\end{proof}
	
	Again we denote by $\mathcal{K}_\mathbf{w}^-$ the category of direct sums of shifts of $\mathcal{L}(\ov{\mathbf{v}}_{\tg},\ov{\mathbf{w}}_{\tg})$ for strongly $l$-dominant pairs $(\ov{\mathbf{v}}_{\tg},\ov{\mathbf{w}}_{\tg})$ for $\mathcal{R}^{\gr}$, which are complexes on $\mathcal{M}_0(\ov{\mathbf{w}}_{\tg},\mathcal{R}^{\gr})$. Similar to \eqref{eq:decompRgr+}, we can write
	\[
	\mathcal{M}_0(\ov{\mathbf{w}}_{\tg},\mathcal{R}^{\gr})=\bigsqcup_{(\bv^0,\bw^0)\in\bigoplus_i\N(\bv^{\Sigma i},\bw^i)}\mathcal{M}_0(\bv,\bw,\mathcal{R}),
	\]
	so there are functors
	\[(i^-)_*:\mathcal{K}_\mathbf{w}^-\longrightarrow \mathcal{K}_\mathbf{w},\quad (i^-)^*:\mathcal{K}_\mathbf{w}\longrightarrow \mathcal{K}_\mathbf{w}^-.\]
	The following characterization of the image of $(i^-)_*$ is immediate.
	
	\begin{proposition}\label{dCB on - char}
		For any strongly $l$-dominant pair $(\bv,\bw)$, the complex $\mathcal{L}(\mathbf{v},\mathbf{w})$ belongs to $\mathcal{K}_\mathbf{w}^-$ if and only if $(\mathbf{v}^0,\mathbf{w}^0)\in\bigoplus_i\N(\mathbf{v}^{\Sigma i},\mathbf{w}^i)$ for the decomposition $(\mathbf{v},\mathbf{w})=(\mathbf{v}^0,\mathbf{w}^0)+(\mathbf{v}^+,\mathbf{w}^+)+(\mathbf{v}^-,\mathbf{w}^-)$.
	\end{proposition}

	\subsection{Coincidence with dual canonical basis}\label{subsec: compare cb}
	
	In this subsection, we shall prove that the dual canonical basis coincides with the double canonical basis for $\hU$. By Definition \ref{def:dual CB}, we only need to relate the double canonical basis with the dual canonical basis $L(\mathbf{v},\mathbf{w})$ of $\hR$. 
	
	We recall the set up of \S\ref{subsec: QV C^*-action}. Consider the inclusions
	\[\begin{tikzcd}
		&\mathcal{M}_0(\bw_{\tg},\mathcal{R}^{\gr})\ar[rd,"i^+"]&\\
		\mathcal{M}_0(\mathbf{w}^-,\mathcal{R})\times \mathcal{M}_0(\mathbf{w}^+,\mathcal{R})\ar[ru,"j^+"]\ar[rd,swap,"j^-"]&&\mathcal{M}_0(\mathbf{w},\mathcal{R})\\
		&\mathcal{M}_0(\ov{\bw}_{\tg},\mathcal{R}^{\gr})\ar[ru,swap,"i^-"]&
	\end{tikzcd}\] 
	Let $\widehat{\mathbf{H}}^\pm=\bigoplus_{\mathbf{w}}\widehat{\mathbf{H}}^\pm_{\mathbf{w}}$, where $\widehat{\mathbf{H}}_{\mathbf{w}}^\pm$ is the dual of the Grothendieck group generated by the perverse sheaves $\mathcal{L}(\mathbf{v}_{\tg},\mathbf{w}_{\tg})$ (resp. $\mathcal{L}(\ov{\bv}_{\tg},\ov{\bw}_{\tg})$) on $\mathcal{M}_0(\bw_{\tg},\mathcal{R}^{\gr})$ (resp. $\mathcal{M}_0(\ov{\bw},\mathcal{R}^{\gr})$) (recall that $\bw_{\tg}$ and $\ov{\bw}_{\tg}$ are uniquely determined by $\bw$). The algebra structure of $\widehat{\mathbf{H}}^\pm$ is defined similarly to $\hR$.  In view of Lemma~\ref{QV fixed point restriction on +}, we can use $(j^\pm)^!$ and $(i^\pm)^*$ to obtain linear embeddings of dual Grothendieck rings:
	\begin{equation}\label{eq:embedding by j and i}
		\begin{tikzcd}
			\hR^-\otimes \hR^+\ar[r,hook,"m_+"]&\widehat{\mathbf{H}}^+\ar[rd,hook,"\iota_+"]&\\
			&&\hR\\
			\hR^+\otimes \hR^-\ar[r,hook,swap,"m_-"]&\widehat{\mathbf{H}}^-\ar[ru,hook,swap,"\iota_-"]&
		\end{tikzcd}
	\end{equation}
	Recall that by \cite[Theorem 2.10.3]{Ach21}, the restriction functors of $\mathcal{M}_0(\bw,\mathcal{R})$ satisfy
	\begin{equation}\label{eq:res of tRi by pullback}
		\Res^{\bw}_{\bw^-,\bw^+}\cong(j^+)^!(i^+)^*,\quad \Res^{\bw}_{\bw^+,\bw^-}\cong(j^-)^!(i^-)^*,
	\end{equation}
	so the map $m_+$ (resp. $m_-$) sends $b_-\otimes b_+$ to the product $b_-b_+$ (resp. $b_+\otimes b_-$ to $b_+b_-$), computed in $\widehat{\mathbf{H}}^+$ (resp. in $\widehat{\mathbf{H}}^-$), and $\iota_\pm$ is the map induced from the triangular decomposition.
	
	\begin{lemma}
		The algebra $\widehat{\mathbf{H}}^\pm$ is isomorphic to $\mathcal{H}^\pm$ and we have the following commutative diagrams
		\[\begin{tikzcd}
			\U^-\otimes\U^+\ar[r,hook,"m_+"]\ar[d]&\mathcal{H}^+\ar[r,hook,"\iota_+"]\ar[d]&\hU\ar[d]\\
			\hR^-\otimes \hR^+\ar[r,hook,"m_+"]&\widehat{\mathbf{H}}^+\ar[r,hook,"\iota_+"]&\hR
		\end{tikzcd}\quad \begin{tikzcd}
			\U^+\otimes\U^-\ar[r,hook,"m_-"]\ar[d]&\mathcal{H}^-\ar[r,hook,"\iota_-"]\ar[d]&\hU\ar[d]\\
			\hR^+\otimes \hR^-\ar[r,hook,"m_-"]&\widehat{\mathbf{H}}^-\ar[r,hook,"\iota_-"]&\hR
		\end{tikzcd}\]
		where the map $m_+:\U^-\otimes\U^+\rightarrow \mathcal{H}^+$ (resp. $m_-:\U^+\otimes\U^-\rightarrow \mathcal{H}^-$) is defined by sending $b_-\otimes b_+$ to $b_-b_+$ (resp. $b_+\otimes b_-$ to $b_+b_-$).
	\end{lemma}
	\begin{proof}
		We only prove the claim for $\mathcal{H}^+$, the case for $\mathcal{H}^-$ is similar. Using the functor $(i^+)_*$ we can define a quotient map
		\[\pi^+:\hR\longrightarrow\widehat{\mathbf{H}}^+\]
		which acts by sending $L(\mathbf{v}^{\Sigma i},\mathbf{w}^i)$ ($i\in Q_0$) to zero in view of Proposition~\ref{dCB on + char}. Since $(i^+)_*$ commutes with the restriction functors, this is a homomorphism of algebras. Moreover, the embedding $\iota:\widehat{\mathbf{H}}^+\rightarrow\hR$ splits $\pi^+$. Since $\hR$ is isomorphic to $\hU$, this implies $\widehat{\mathbf{H}}^+\cong\mathcal{H}^+$. The commutativity of the diagram is now clear.
	\end{proof}
	
	To distinguish between different quantum Grothendieck rings, for $(\bv^\pm,\bw^\pm)\in(V^\pm,W^\pm)$, let us denote by $\mathcal{P}(\bv^\pm,\bw^\pm)$ the IC complex associated to the constant local system on $\mathcal{M}_0^\reg(\bv^\pm,\bw^\pm)$, and $B(\bv^\pm,\bw^\pm)$ by the corresponding dual basis, which lies in $\hR^\pm$. Similarly, for $(\mathbf{v},\mathbf{w})$ such that $(\mathbf{v}^0,\mathbf{w}^0)\in\bigoplus_i\N(\mathbf{v}^i,\mathbf{w}^i)$ (resp. $(\mathbf{v}^0,\mathbf{w}^0)\in\bigoplus_i\N(\mathbf{v}^{\Sigma i},\mathbf{w}^i)$), we denote by $\mathcal{L}^+(\mathbf{v},\mathbf{w})$ (resp. $\mathcal{L}^-(\bv,\bw)$) the restriction of $\mathcal{L}(\mathbf{v},\mathbf{w})$ to $\mathcal{M}_0(\bw_{\tg},\mathcal{R}^{\gr})$ (resp. $\mathcal{M}_0(\ov{\bw}_{\tg},\mathcal{R}^{\gr})$), and let $L^+(\mathbf{v},\mathbf{w})$ (resp. $L^-(\bv,\bw)$ be its dual basis in $\widehat{\mathbf{H}}^+$ (resp. $\widehat{\mathbf{H}}^-$). 
	
	\begin{lemma}
		For $(\bv^\pm,\bw^\pm)\in(V^\pm,W^\pm)$, we have
		\begin{align}
			\label{dCB + embedding-1}
			\begin{split}&B(\mathbf{v}^-,\mathbf{w}^-)\cdot B(\mathbf{v}^+,\mathbf{w}^+)\\
				&=L^+(\mathbf{v}^++\bv^-,\mathbf{w}^++\bw^-)+\sum_{\bv^++\bv^-<\mathbf{v}'}a^+_{\mathbf{v}^++\bv^-,\mathbf{v}'}L^+(\mathbf{v}',\mathbf{w}^++\bw^-)\in \widehat{\mathbf{H}}^+,\end{split}
			\\ 
			\begin{split}
				&B(\mathbf{v}^+,\mathbf{w}^+)\cdot B(\mathbf{v}^-,\mathbf{w}^-)
				\\
				&=L^-(\mathbf{v}^++\bv^-,\mathbf{w}^++\bw^-)+\sum_{\bv^++\bv^-<\mathbf{v}'}a^-_{\mathbf{v}^++\bv^-,\mathbf{v}'}L^-(\mathbf{v}',\mathbf{w}^++\bw^-)\in \widehat{\mathbf{H}}^-
				\label{dCB - embedding-1}
			\end{split}
		\end{align}
		where $a^\pm_{\mathbf{v}^++\bv^-,\mathbf{v}'}\in v\N[v]$. On the other hand, for $(\mathbf{v},\mathbf{w})$ such that $(\mathbf{v}^0,\mathbf{w}^0)\in\bigoplus_i\N(\mathbf{v}^i,\mathbf{w}^i)$ (resp. $(\mathbf{v}^0,\mathbf{w}^0)\in\bigoplus_i\N(\mathbf{v}^{\Sigma i},\mathbf{w}^i)$), we have
		\begin{align}
			\iota_+(L^+(\mathbf{v},\mathbf{w}))&=L(\mathbf{v},\mathbf{w})+\sum_{\mathbf{v}<\mathbf{v}'}b^+_{\mathbf{v},\mathbf{v}'}L(\mathbf{v}',\mathbf{w}),
			\label{dCB + embedding-2}\\
			\iota_-(L^-(\mathbf{v},\mathbf{w}))&=L(\mathbf{v},\mathbf{w})+\sum_{\mathbf{v}<\mathbf{v}'}b^-_{\mathbf{v},\mathbf{v}'}L(\mathbf{v}',\mathbf{w}),
			\label{dCB - embedding-2}
		\end{align}
		where $b^\pm_{\mathbf{v},\mathbf{v}'}\in v^{-1}\N[v^{-1}]$.
	\end{lemma}
	\begin{proof}
		Let us deal with $B(\mathbf{v}^-,\mathbf{w}^-)\cdot B(\mathbf{v}^+,\mathbf{w}^+)$, the other equalities can be proved similarly. Now from the arguments above we know that $B(\mathbf{v}^-,\mathbf{w}^-)\cdot B(\mathbf{v}^+,\mathbf{w}^+)=m_+(B(\mathbf{v}^-,\mathbf{w}^-)\otimes B(\mathbf{v}^+,\mathbf{w}^+))$, so it suffices to consider $(j^+)^!\mathcal{L}(\bv,\bw)$ where $(\bv,\bw)$ is such that $(\mathbf{v}^0,\mathbf{w}^0)\in\bigoplus_i\N(\mathbf{v}^i,\mathbf{w}^i)$. Let us write 
		\[(j^+)^!\mathcal{L}(\bv,\bw)=\sum_{\bv'^++\bv'^-<\bv}a^+_{\bv'^+\bv'^-,\bv}\mathcal{P}(\bv'^-,\bw^-)\boxtimes\mathcal{P}(\bv'^+,\bw^+).\]
		From the definition it is clear that $a^+_{\bv^+\bv^-,\bv^++\bv^-}=1$ if $\bv^0=\bw^0=0$. Otherwise by taking a generic point in $\mathcal{M}_0^\reg(\bv^-,\bw^-)\times\mathcal{M}_0^\reg(\bv^+,\bw^+)$ and using \cite[2.1.9]{BBD}, it is easy to see that $a^+_{\bv^+\bv^-,\bv}\in v\N[v]$. This proves the assertion.
	\end{proof}

	\begin{theorem}\label{iCB is DCB}
		The dual canonical bases of $\mathcal{H}^\pm$ coincide with the double canonical bases, and the positive (resp. negative) double canonical basis for $\hU$ coincides with the dual canonical basis.
	\end{theorem}
	
	\begin{proof}
		We only consider the case for $\mathcal{H}^+$, since that for $\mathcal{H}^-$ is similar. Using (\ref{dCB + embedding-1}), (\ref{dCB + embedding-2}) and \eqref{eq:diamond for L def}, for $(\bv,\bw)=(\bv^0,\bw^0)+(\bv^+,\bw^+)+(\bv^-,\bw^-)$, we can write
		\begin{align*}
			L^+(\mathbf{v},\mathbf{w})&\in L^+(\mathbf{v}^0,\mathbf{w}^0)\diamond (B(\mathbf{v}^-,\mathbf{w}^-)\cdot B(\mathbf{v}^+,\mathbf{w}^+))\\
			&+\sum_{\substack{ \mathbf{w}=\tilde{\mathbf{w}}^0+\tilde{\mathbf{w}}^++\tilde{\mathbf{w}}^-\\(\mathbf{v}^0,\mathbf{w}^0)\prec(\tilde{\mathbf{v}}^0,\tilde{\mathbf{w}}^0)}} v\Z[v]L^+(\tilde{\mathbf{v}}^0,\tilde{\mathbf{w}}^0) \diamond (B(\tilde{\mathbf{v}}^-,\tilde{\mathbf{w}}^-)\cdot B(\tilde{\mathbf{v}}^+,\tilde{\mathbf{w}}^+)),\\
			L(\mathbf{v},\mathbf{w})&\in L(\mathbf{v}^0,\mathbf{w}^0)\diamond \iota(L^+(\mathbf{v}^++\mathbf{v}^-,\mathbf{w}^++\mathbf{w}^-))\\
			&+\sum_{\substack{ \mathbf{w}=\tilde{\mathbf{w}}^0+\tilde{\mathbf{w}}^++\tilde{\mathbf{w}}^-\\(\mathbf{v}^0,\mathbf{w}^0)\prec(\tilde{\mathbf{v}}^0,\tilde{\mathbf{w}}^0)}} v^{-1}\Z[v^{-1}] L(\tilde{\mathbf{v}}^0,\tilde{\mathbf{w}}^0) \diamond \iota(L^+(\tilde{\mathbf{v}}^++\tilde{\mathbf{v}}^-,\tilde{\mathbf{w}}^++\tilde{\mathbf{w}}^-))
		\end{align*}
		Since $\{B(\bv^\pm,\bw^\pm)\mid (\bv^\pm,\bw^\pm)\in(V^\pm,W^\pm)\}$ gives rise to the rescaled dual canonical basis of $\U^\pm$ in view of Theorem~\ref{dCB of U^+ by L}, the result follows by comparing these with the defining relations of double canonical basis of $\mathcal{H}^+$ and $\hU$ in Theorem \ref{thm:doubleCB-H+} and Theorem \ref{thm:doubleCB-U +}.
	\end{proof}

	\begin{remark}
		It can be seen from the proof of Theorem~\ref{iCB is DCB} that 
		\begin{align}
			\label{eq:left-double-dCB}
			b_+\bullet b_-=b_-\bullet b_+,\qquad  \forall (b_-,b_+)\in\widetilde{\mathbf{B}}^-\times\widetilde{\mathbf{B}}^+. 
		\end{align}
		%for any $(b_-,b_+)\in\widetilde{\mathbf{B}}^-\times\widetilde{\mathbf{B}}^+$ we have $b_+\bullet b_-=b_-\bullet b_+$.
	\end{remark}

	Using Theorem~\ref{iCB is DCB}, we can answer many conjectures proposed by \cite{BG17}. For example, the following results have already been proved for the dual canonical basis of $\tU$ (viewed as iquantum group of diagonal type), so they are also valid for the double canonical basis.
	
	\begin{corollary}[{\cite[Conjecture 1.15]{BG17}; see \cite[Theorem 5.16]{LP25}}]\label{coro: QG double CB braid}
		The double ($=$ dual) canonical basis of $\tU$ is invariant under braid group actions.
	\end{corollary}
	
	\begin{corollary}[{\cite[Conjecture 1.21]{BG17}}%, see \eqref{dCB + embedding-1}--\eqref{dCB - embedding-2}
		]\label{coro: QG double CB positive}
		For any $(b_-,b_+)\in\widetilde{\mathbf{B}}^-\times\widetilde{\mathbf{B}}^+$, the transition coefficients of $b_-b_+$ (or $b_+b_-$) with respect to the double canonical basis of $\hU$ (and also $\tU$) belong to $\N[v,v^{-1}]$.
	\end{corollary}
	\begin{proof}
		It suffices to consider $b_-b_+$. From \eqref{dCB + embedding-1} and \eqref{dCB + embedding-2} we see that
		\begin{equation*}
			\iota_+(L(\mathbf{v}^-,\mathbf{w}^-)\cdot L(\mathbf{v}^+,\mathbf{w}^+))=L(\mathbf{v}^++\bv^-,\mathbf{w}^++\bw^-)+\sum_{\bv'}\N[v,v^{-1}] L(\mathbf{v}',\mathbf{w}^++\bw^-).
		\end{equation*}
		From the definition we know that $\iota_+(L(\mathbf{v}^-,\mathbf{w}^-)\cdot L(\mathbf{v}^+,\mathbf{w}^+))$ equals to $L(\mathbf{v}^-,\mathbf{w}^-)\cdot L(\mathbf{v}^+,\mathbf{w}^+)$ calculated in $\hU$, the assertion therefore follows.
	\end{proof}
	
	\begin{corollary}
		The structure constants of double ($=$ dual) canonical basis are in $\N[v^{\frac{1}{2}},v^{-\frac{1}{2}}]$. 
	\end{corollary}

	In \cite{BG17} the authors defined $\Q(v^{\frac{1}{2}})$-linear anti-involutions ${\cdot}^*:\hU\rightarrow\hU$ and ${\cdot}^t:\hU\rightarrow\hU$ via
	\begin{gather*}
		E_i^*=E_i,\quad F_i^*=F_i,\quad K_i^*=K_i',\quad (K_i')^*=K_i,\\
		E_i^t=F_i,\quad F_i^t=E_i,\quad K_i^t=K_i,\quad (K_i')^t=K_i'.
	\end{gather*}
	Then $(\widetilde{\mathbf{B}}^{\pm})^t=\widetilde{\mathbf{B}}^{\mp}$, and it is proved in \cite[Lemma 3.5]{BG17} that ${\cdot}^*$ preserves $\widetilde{\mathbf{B}}^{\pm}$ as sets. It is not hard to see that ${\cdot}^t$ preserves the basis $\{K\diamond(b_-\circ b_+)\}$, so it preserves the double canonical basis (as a set) of $\tU$; see  \cite[Proposition 2.10]{BG17}. It is conjectured in \cite{BG17} that ${\cdot}^*$ also preserves the double canonical basis. Using Theorem~\ref{iCB is DCB} we can now verify this conjecture.

	\begin{proposition}[{\cite[Conjecture 1.11]{BG17}}]
		\label{QG dCB invariant under *}
		The double ($=$ dual) canonical basis of $\hU$ (and also $\tU$) is invariant under ${\cdot}^*$. More precisely, we have
		\[(K\diamond(b_-\bullet b_+))^*=K^*\diamond (b_+^*\bullet b_-^*)=K^*\diamond (b_-^*\bullet b_+^*).\]
	\end{proposition}
	
	\begin{proof}
		Note that for $(b_-,b_+)\in \widetilde{\mathbf{B}}^-\times\widetilde{\mathbf{B}}^+$ we have $\deg_\Gamma(b_-b_+)=\deg_\Gamma(b_+b_-)$, so
		\begin{align*}
			(K_i\diamond b_-b_+)^*&=v^{-\frac{1}{2}\check{\alpha}_i(\deg_\Gamma(b_-b_+)}(K_ib_-b_+)^*=v^{-\frac{1}{2}\check{\alpha}_i(\deg_\Gamma(b_-b_+)}b_+^*b_-^*K_i'\\
			&=v^{\frac{1}{2}\check{\alpha}_i(\deg_\Gamma(b_-b_+)}K_i'b_+^*b_-^*=K_i'\diamond(b_+^*b_-^*),
		\end{align*}
		and in general we can write
		\[(K\diamond(b_-b_+))^*=K^*\diamond (b_+^*b_-^*),\quad \forall K\in\hU^0.\]
		
		It is clear that ${\cdot}^*$ induces an anti-automorphism between $\mathcal{H}^\pm$ and is compatible with the inclusions $\iota_\pm$, that is
		\[\iota_\mp(x^*)=(\iota_\pm(x))^*,\quad \forall x\in\mathcal{H}^\pm.\]
		Now apply ${\cdot}^*$ to the defining equation of $b_-\circ b_+$, we find 
		\begin{equation}\label{QG dCB invariant under *-1}
			(b_-\circ b_+)^*-b_+^*b_-^*\in\sum v\Z[v]K^*\diamond({'b_+^*}{'b_-^*})
		\end{equation}
		where the sum is taken over $K\in\mathbf{K}^+\setminus\{1\}$ and ${'b_{\pm}}\in\widetilde{\mathbf{B}}^{\pm}$ such that $\deg_\Gamma(b_-b_+)=\deg_\Gamma(K)+\deg_\Gamma({'b_-}{'b_+})$. Moreover, since ${\cdot}^*$ commutes with $\ov{\cdot}$, the element $(b_-\circ b_+)^*$ is bar-invariant, so it must equal to $b_+^*\circ b_-^*$. The same arguments apply to $b_+\circ b_-$, in which case we find $(b_+\circ b_-)^*=b_-^*\circ b_+^*$.
		
		Finally, applying ${\cdot}^*$ to the defining equation of $b_-\bullet b_+$ and using \eqref{QG dCB invariant under *-1}, we get
		\[(b_-\bullet b_+)^*-\iota_-(b_+^*\circ b_-^*)\in\sum v^{-1}\Z[v^{-1}]K^*\diamond \iota_-({'b_+^*}\circ {'b_-^*})\]
		where the sum is taken over $K^*\in\hU^0\setminus\mathbf{K}^+$ and ${'b_{\pm}}\in\widetilde{\mathbf{B}}^{\pm}$ such that $\deg_\Gamma(b_-b_+)=\deg_\Gamma(K)+\deg_\Gamma({'b_-}{'b_+})$. Again since $(b_-\bullet b_+)^*$ is bar-invariant, it must equal to $b_+^*\bullet b_-^*$, and our assertion follows from \eqref{eq:left-double-dCB}.
	\end{proof}
	
	Recall that the Chevalley involution $\omega$ is a $\Q(v^{\frac12})$-algebra involution of $\hU$ (and also $\tU$) sending $E_i\mapsto F_i$, $F_i\mapsto E_i$, $K_i\mapsto K_i'$, $K_i'\mapsto K_i$. It is the composition of the anti-involution $\cdot^*$ and $\cdot ^t$, so we have the following corollary.
	
	\begin{corollary}
		\label{cor:dCB-chevalley}
		The double ($=$ dual) canonical basis of $\hU$ (and also $\tU$) is preserved by the Chevalley involution $\omega$. 
	\end{corollary}
	
	Recall the PBW basis of $\hU$ (and $\tU$) given in \eqref{PBW}. The following is a corollary of \cite[Proposition 6.8]{LP26}. 
	
	\begin{corollary}
		\label{cor:positive-trans-PBW-dCB}
		The transition matrix coefficients of the PBW basis $\{F^{\ba}E^{\bc}K_\mu K'_{\nu}\mid \ba,\bc\in\N^{|\Phi^+|},\mu,\nu\in\Z^\I\}$ with respect to the double ($=$dual) canonical basis of $\tU$ belong to $\N[v^{\frac{1}{2}},v^{-\frac{1}{2}}]$.
	\end{corollary}
	
	\begin{remark}
		Bridgland \cite{Br13} gave a Hall algebra realization of $\tU$ (and also $\hU$); see \cite[Theorem 8.5]{LW19} (or \cite[Lemma 3.10]{LP25}) for a reformulation via the generic semi-derived Hall algebra $\cs\cd \widetilde{\ch}(\Lambda)$ associated to the algebra $\Lambda$.  The Hall algebra $\cs\cd \widetilde{\ch}(\Lambda)$ has a natural basis (called Hall basis), which is different to the PBW basis. As a corollary of \cite[Theorem 6.6]{LP26}, we know the transition matrix coefficients of the Hall basis with respect to the double ($=$dual) canonical basis of $\tU$ belong to $\N[v^{\frac{1}{2}},v^{-\frac{1}{2}}]$.

		In  \cite{Lus90}, Lusztig proved that the transition matrix coefficients of the canonical basis of $\U^+$ (resp. $\U^-$) with respect to the PBW basis $\{E^\bc\mid \bc\in \N^{|\Phi^+|}\}$ (resp. $\{F^\ba\mid \ba\in \N^{|\Phi^+|}\}$) belong to $\N[v^{\frac{1}{2}},v^{-\frac{1}{2}}]$. Then Corollary \ref{cor:positive-trans-PBW-dCB} generalizes (the dual version of) Lusztig's result on $\U^\pm$ to the entire quantm group $\tU$.
	\end{remark}
	
	%%%%%%%%%%%
	%%%%%%%%%%%
	\section{Dual canonical basis of $\tU_v(\mathfrak{sl}_2)$}\label{sec:dCB rank I}
	
	In this subsection we determine the dual canonical basis of $\tU_v(\mathfrak{sl}_2)$. For this we need to consider the double of the quiver with one vertex $1$ and no edges. The quiver of the regular NKS category $\mcr$ is 
	\[
	\xymatrix{\tS_1\ar[d]^{\alpha_1} & \sigma(\tS_1) \ar[l]_{\beta_1} \\
		\sigma(\Sigma \tS_1) \ar[r]^{\beta_2}& \Sigma \tS_1\ar[u]^{\alpha_2} }
	\]
	subject to $\beta_2\alpha_1=0,\beta_1\alpha_2=0$.
	We have 
	\[\N^{\mathcal{R}_0-\mathcal{S}_0}=\N\e_{\tS_1}+\N\e_{\Sigma\tS_1},\quad \N^{\mathcal{S}_0}=\N\e_{\sigma\tS_1}+\N\e_{\sigma\Sigma\tS_1}.\]
	For vectors $\bv\in\N^{\mathcal{R}_0-\mathcal{S}_0}$ and $\bw\in\N^{\mathcal{S}_0}$, we will write
	\[\bv_1=\bv(\tS_1),\quad \bv_2=\bv(\Sigma\tS_1),\quad \bw_{1}=\bw(\sigma\tS_1),\quad \bw_2=\bw(\sigma\Sigma\tS_1).\]
	For any $\bw\in\N^{\mathcal{S}_0}$, let us fix a graded vector space $\bW=\bW_1\oplus \bW_2$ of dimension vector $(\bw_1,\bw_2)$. Then the quiver varieties are given by 
	\begin{align*}
		\mathcal{M}(\bv,\bw,\mathcal{R})&=\{(x,y,\bV\subseteq\bW)\mid 
		\dimv\bV=(\bv_1,\bv_2),x|_{\bV_1}=y|_{\bV_2}=0,x(\bW_1)\subseteq\bV_2,y(\bW_2)\subseteq\bV_1\},\\
		\mathcal{M}_0(\bw,\mathcal{R})&=\{(x,y) \mid xy=yx=0\},\\
		\mathcal{M}_0^{\reg}(\bw,\mathcal{R})&=\{(x,y)\in\mathcal{M}_0(\bv,\bw)\mid \rank(y)=\bv_1,\rank(x)=\bv_2\},
	\end{align*}
	where $x:\bW_1\rightarrow\bW_2,y:\bW_2\rightarrow\bW_1$ are linear maps. Note that for $\mathcal{M}_0^{\reg}(\bv,\bw,\mathcal{R})\neq\emptyset$ we must have $\bv_1+\bv_2\leq\min\{\bw_1,\bw_2\}$. Also, the variety $\mathcal{M}(\bv,\bw,\mathcal{R})$ is smooth and
	\[\dim\mathcal{M}(\bv,\bw,\mathcal{R})=(\bv_1+\bv_2)(\bw_1+\bw_2-\bv_1-\bv_2).\]
	
	The image of the map $\pi_{\bv,\bw}:\mathcal{M}(\bv,\bw,\mathcal{R})\rightarrow \mathcal{M}_0(\bw)$ consists of $\mathcal{M}_0^{\reg}(\bv',\bw,\mathcal{R})$ such that 
	\begin{equation}\label{eq:rank I pi_vw image char}
		\bv'_1\leq
		\min\{\bv_1,\bw_2-\bv_2'\},\quad \bv'_2\leq\min\{\bv_2,\bw_1-\bv_1'\}.
	\end{equation}
	and we denote by $I_\bv(\bw)$ the set of vectors $\bv'$ satisfying the above conditions.
	
	Let $F_{\bv,\bv'}$ denote the fiber of $\pi_{\bv,\bw}$ over a generic point of $\mathcal{M}_0^{\reg}(\bv',\bw,\mathcal{R})$. An immediate calculation shows that 
	\begin{equation}\label{eq:rank I defection number}
		2\dim F_{\bv,\bv'} + \dim\mathcal{M}_0^{\reg}(\bv',\bw,\mathcal{R}) - \dim\mathcal{M}(\bv,\bw,\mathcal{R}) = \tr_{\bv,\bv'}(\bw_1-\bw_2-\tr_{\bv,\bv'}),
	\end{equation}
	where $\tr_{\bv,\bv'}:=(\bv_1-\bv_2)-(\bv'_1-\bv'_2)$. In particular, if $|\bw_1-\bw_2|\leq 1$ then $\pi_{\bv,\bw}$ is semi-small, so by \cite{dM02} (cf. \cite[Theorem 7.1]{LP26}) we obtain an explicit decomposition
	\begin{equation}\label{eq:rank I pi_vw decomp diagonal case}
		\pi(\bv,\bw)=\begin{dcases}
			\bigoplus_{\substack{\bv'\in I_\bv(\bw)\\ \tr_{\bv,\bv'}=0}}\mathcal{L}(\bv',\bw,\mathcal{R})& \text{if $|\bw_1-\bw_2|=0$},\\
			\bigoplus_{\substack{\bv'\in I_\bv(\bw)\\ \text{$\tr_{\bv,\bv'}=0$ or $\bw_1-\bw_2$}}}\mathcal{L}(\bv',\bw,\mathcal{R})& \text{if $|\bw_1-\bw_2|=1$}.
		\end{dcases}
	\end{equation}
	
	Given vectors $\bw',\bw''$ in $\N^{\mathcal{S}_0}$, we consider the restriction functor $\Res^{\bw}_{\bw',\bw''}$, where $\bw=\bw'+\bw''$. In our case \eqref{eqn:comultiplication} reads
	\begin{equation}\label{eq:rank I res formula}
		\Res^{\bw}_{\bw',\bw''}(\pi(\bv,\bw))=\bigoplus_{\bv'+\bv''=\bv}\pi(\bv',\bw')\boxtimes\pi(\bv'',\bw'')[d((\bv',\bw'),(\bv'',\bw''))],
	\end{equation}
	where
	\[d((\bv',\bw'),(\bv'',\bw''))=(\bw'_1-\bw'_2)(\bv''_1-\bv''_2)-(\bw''_1-\bw''_2)(\bv'_1-\bv'_2).\]
	
	\begin{example}\label{eg:rank I rest to EF formula}
		For $\bw\in\N^{\mathcal{S}_0}$, we define vectors $\bw^{(1)}=\bw_1\e_{\sigma\tS_1}$ and $\bw^{(2)}=\bw_2\e_{\sigma\Sigma\tS_1}$. Then \eqref{eq:rank I res formula} implies
		\begin{equation*}
			\Res^{\bw}_{\bw^{(1)},\bw^{(2)}}(\pi(\bv,\bw))=\pi(\bv^{(1)},\bw^{(1)})\boxtimes \pi(\bv^{(2)},\bw^{(2)})[\bv_1\bw_2-\bv_2\bw_1]
		\end{equation*}
		where $\bv^{(1)}=\bv_1\e_{\tS_1}$ and $\bv^{(2)}=\bv_2\e_{\Sigma\tS_1}$. From the definition we easily derive that
		\[\pi(\bv^{(1)},\bw^{(1)})=\qbinom{\bw_1}{\bv_1}\mathcal{L}(0,\bw^{(1)}),\quad \pi(\bv^{(2)},\bw^{(2)})=\qbinom{\bw_2}{\bv_2}\mathcal{L}(0,\bw^{(2)})\]
		so 
		\[\Res^{(\bw_1,\bw_2)}_{(\bw_1,0),(0,\bw_2)}(\pi(\bv,\bw))=\qbinom{\bw_1}{\bv_1}\qbinom{\bw_2}{\bv_2} \mathcal{L}(0,\bw^{(1)}) \boxtimes \mathcal{L}(0,\bw^{(2)} [\bv_1\bw_2-\bv_2\bw_1].\]
	\end{example}
	%%%%%%%
	
	We now assume $\bw_1\geq\bw_2$ and set $k=\bw_1-\bw_2$. Then we may write 
	\begin{equation}\label{eq:rank I res of IC to diagonal-1}
		\begin{aligned}
			\Res^{\bw}_{k\e_{\sigma\tS_1},\bw-k\e_{\sigma\tS_1}} (\mathcal{L}(\bv,\bw))=\bigoplus_{\bv'}b_{\bv,\bv'}\cdot \mathcal{L}(0,k\e_{\sigma\tS_1}) \boxtimes \mathcal{L}(\bv',\bw-k\e_{\sigma\tS_1})[k(\bv'_1-\bv'_2)]
		\end{aligned}
	\end{equation}
	for some $b_{\bv,\bv'}\in\N[v,v^{-1}]$. Note that $b_{\bv,\bv}=1$ by \eqref{eqn: leading term}, and $b_{\bv,\bv'}=0$ if the condition $\bv'\leq\bv$ is not satisfied (this can be seen from \eqref{eqn:comultiplication}).
	
	Using the decomposition \eqref{eq:rank I pi_vw decomp diagonal case}, we find that 
	\begin{equation}\label{eq:rank I res of IC to diagonal-2}
		\begin{aligned}
			&\Res^{\bw}_{k\e_{\sigma\tS_1},\bw-k\e_{\sigma\tS_1}} (\pi(\bv,\bw))= \bigoplus_{r=0}^{k} \qbinom{k}{r} \pi(0,k\e_{\sigma\tS_1}) \boxtimes \pi(\bv-r\e_{\tS_1},\bw-k\e_{\sigma\tS_1})[k(\bv_1-\bv_2-r)]\\
			=&\bigoplus_{r=0}^{k}\bigoplus_{\substack{\bv'\in I_{\bv-r\e_{\tS_1}}(\bw-k\e_{\sigma\tS_1})\\ \tr_{\bv,\bv'}=r}} \qbinom{k}{r} \mathcal{L}(0,k\e_{\sigma\tS_1}) \boxtimes \mathcal{L}(\bv',\bw-k\e_{\sigma\tS_1})[k(\bv_1-\bv_2-r)].
		\end{aligned}
	\end{equation}
	It is easy to check that
	%\[\bigcup_{r=0}^{k}\{\bv'\in I_{\bv-r\e_{\tS_1}}(\bw-k\e_{\sigma\tS_1}) \mid \tr_{\bv,\bv'}=r\} = \bigcup_{r=0}^{k}\{\bv' \in I_{\bv}(\bw) \mid \tr_{\bv,\bv'}=r\}\]
	\[\{\bv'\in I_{\bv-r\e_{\tS_1}}(\bw-k\e_{\sigma\tS_1}) \mid \tr_{\bv,\bv'}=r\} = \{\bv' \in I_{\bv}(\bw) \mid \tr_{\bv,\bv'}=r\},\quad \forall 0\leq r\leq k.\]
	so (\ref{eq:rank I res of IC to diagonal-2}) can be rewritten as 
	\begin{equation}\label{eq:rank I res of IC to diagonal-3}
		\begin{aligned}
			\Res^{\bw}_{k\e_{\sigma\tS_1},\bw-k\e_{\sigma\tS_1}} (\pi(\bv,\bw)) = \bigoplus_{\substack{\bv'\in I_{\bv}(\bw)\\ 0\leq \tr_{\bv,\bv'}\leq k}} \qbinom{k}{\tr_{\bv,\bv'}} \mathcal{L}(0,k\e_{\sigma\tS_1}) \boxtimes \mathcal{L}(\bv',\bw-k\e_{\sigma\tS_1})[k(\bv_1-\bv_2-r)].
		\end{aligned}
	\end{equation}
	For any given $\bv'\in I_{\bv}(\bw)$, if $\mathcal{L}(\bv',\bw)$ apprears in the decomposition of $\pi(\bv',\bw)$ then $\Res^{\bw}_{k\e_{\sigma\tS_1},\bw-k\e_{\sigma\tS_1}}(\mathcal{L}(\bv',\bw))$ is a direct summand of $\Res^{\bw}_{k\e_{\sigma\tS_1},\bw-k\e_{\sigma\tS_1}} (\pi(\bv,\bw))$. In this case we will have
	\begin{equation*}%\label{eq:rank I res of IC to diagonal-4}
		\begin{aligned}
			\Res^{\bw}_{k\e_{\sigma\tS_1},\bw-k\e_{\sigma\tS_1}}(\mathcal{L}(\bv',\bw))=\bigoplus_{\substack{\bv''\leq\bv' \\ 0\leq \tr_{\bv,\bv''}\leq k}} b_{\bv',\bv''}\cdot \mathcal{L}(0,k\e_{\sigma\tS_1}) \boxtimes \mathcal{L}(\bv'',\bw-k\e_{\sigma\tS_1})[k(\bv'_1-\bv'_2-r)]
		\end{aligned}
	\end{equation*}
	
	On the other hand, if we write $\pi(\bv,\bw)=\bigoplus_{\bv'\in I_\bv(\bw)}a_{\bv,\bv'}\mathcal{L}(\bv',\bw)$, then
	\begin{align*}
		&\Res^{\bw}_{k\e_{\sigma\tS_1},\bw-k\e_{\sigma\tS_1}} (\pi(\bv,\bw) )= \bigoplus_{\bv'\in I_\bv(\bw)}a_{\bv,\bv'}\cdot\Res^{\bw}_{k\e_{\sigma\tS_1},\bw-k\e_{\sigma\tS_1}} \mathcal{L}(\bv',\bw)\\
		&=\bigoplus_{\bv'\in I_\bv(\bw)}\bigoplus_{\substack{\bv''\leq\bv'\\ 0\leq \tr_{\bv,\bv''}\leq k}} a_{\bv,\bv'}\cdot b_{\bv',\bv''} \mathcal{L}(0,k\e_{\sigma\tS_1}) \boxtimes \mathcal{L}(\bv'',\bw-k\e_{\sigma\tS_1})[k(\bv'_1-\bv'_2-r)].
	\end{align*}
	This implies for any fixed $\bv''\in I_{\bv}(\bw)$ with $0\leq \tr_{\bv,\bv''}\leq k$,
	\begin{equation}\label{eq:rank I res of IC to diagonal-5}
		\qbinom{k}{\tr_{\bv,\bv''}} v^{-k(\bv_1-\bv_2-\tr_{\bv,\bv''})} = \sum_{\substack{\bv''\leq\bv'\\ 0\leq \tr_{\bv,\bv''}\leq k}} a_{\bv,\bv'} b_{\bv',\bv''}\cdot v^{-k(\bv_1-\bv_2-\tr_{\bv,\bv'})}
	\end{equation}
	In particular, as $b_{\bv'',\bv''}=1$, the coefficient $a_{\bv,\bv'}$ is nonzero only for $0\leq \tr_{\bv,\bv'}\leq k$. Also recall that $\tr_{\bv,\bv'}-\tr_{\bv,\bv''}=r_{\bv'',\bv'}=-\tr_{\bv',\bv''}$, so (\ref{eq:rank I res of IC to diagonal-5}) gives
	\begin{equation}\label{eq:rank I res of IC to diagonal-6}
		\qbinom{k}{\tr_{\bv,\bv''}}=\sum_{\substack{\bv''\leq\bv'\\ 0\leq \tr_{\bv,\bv'}\leq k}} a_{\bv,\bv'} b_{\bv,\bv''}\cdot v^{-\tr_{\bv',\bv''}}.
	\end{equation}
	
	To proceed further we consider the automorphism $c:\mathcal{M}_0(\bw)\rightarrow\mathcal{M}_0(\bw)$ defined by $(x,y)\mapsto(y^t,x^t)$, where $\cdot^t$ denotes transpose. It is clear that this map permutes the stratification of $\mathcal{M}_0(\bw)$: it sends $\mathcal{M}_0^{\reg}(\bv,\bw,\mathcal{R})$ to $\mathcal{M}_0^{\reg}(\bv^c,\bw,\mathcal{R})$, where $\bv^c=\bv_2\e_{\tS_1}+\bv_1\e_{\Sigma\tS_1}$.
	
	\begin{lemma}\label{eq:rank I coefficient b symmetry}
		$\overline{b_{\bv',\bv''}}=b_{\bv'^c,\bv''^c}$.
	\end{lemma}
	\begin{proof}
		Consider the following commutative convolution diagram for $\Res^{\bw}_{k\e_{\sigma\tS_1},\bw-k\e_{\sigma\tS_1}}$ (cf. \S\S~\ref{subsec: compare cb}):
		\[\begin{tikzcd}
			&F_{k\e_{\sigma\tS_1},\bw-k\e_{\sigma\tS_1}}\ar[rd,"j^+"]&\\
			\mathcal{M}_0(\bw-k\e_{\sigma\tS_1},\mathcal{R})\ar[ru,"i^+"]\ar[rd,swap,"i^-"]&&\mathcal{M}_0(\bw,\mathcal{R})\\
			&F_{\bw-k\e_{\sigma\tS_1},k\e_{\sigma\tS_1}}\ar[ru,swap,"j^-"]&
		\end{tikzcd}\]
		By \cite[Theorem 2.10.3]{Ach21}, we have the following isomorphism for restriction functors:
		\[\Res^{\bw}_{k\e_{\sigma\tS_1},\bw-k\e_{\sigma\tS_1}}=(i^+)^!(j^+)^*,\quad \Res^{\bw}_{\bw-k\e_{\sigma\tS_1},k\e_{\sigma\tS_1}}=(i^-)^!(j^-)^*.\]
		In particular, from the definition of $b_{\bv',\bv''}$ we have
		\[(i^+)^!(j^+)^*\mathcal{L}(\bv',\bw)=\bigoplus_{\bv''\leq\bv'} b_{\bv',\bv''}\cdot \mathcal{L}(0,k\e_{\sigma\tS_1}) \boxtimes \mathcal{L}(\bv'',\bw-k\e_{\sigma\tS_1})[k(\bv''_1-\bv''_2)].\]
		Taking Verdier dual, this gives
		\begin{equation}\label{eq:rank I coefficient b symmetry-1}
			(i^+)^*(j^+)^!\mathcal{L}(\bv',\bw)=\bigoplus_{\bv''\leq\bv'} \ov{b_{\bv',\bv''}}\cdot \mathcal{L}(0,k\e_{\sigma\tS_1}) \boxtimes \mathcal{L}(\bv'',\bw-k\e_{\sigma\tS_1})[k(\bv''_2-\bv''_1)].
		\end{equation}
		On the other hand, by the hyperbolic localization theorem \cite{Bra03}, we know that 
		\[(i^+)^!(j^+)^*\mathcal{L}(\bv',\bw) \cong (i^-)^*(j^-)^!\mathcal{L}(\bv',\bw)\]
		so taking Verdier dual gives
		\begin{equation}\label{eq:rank I coefficient b symmetry-2}
			\begin{aligned}
				(i^-)^!(j^-)^*\mathcal{L}(\bv',\bw) &= (i^+)^*(j^+)^!\mathcal{L}(\bv',\bw)\\
				&=\bigoplus_{\bv''\leq\bv'}\overline{b_{\bv',\bv''}}\cdot \mathcal{L}(0,k\e_{\sigma\tS_1}) \boxtimes \mathcal{L}(\bv'',\bw-k\e_{\sigma\tS_1})[k(\bv''_2-\bv''_1)].
			\end{aligned}
		\end{equation}
		Now we have a commutative diagram 
		\[\begin{tikzcd}
			\mathcal{M}_0(\bw-k\e_{\sigma\tS_1},\mathcal{R})\ar[r,"i^+"]\ar[d,shift left=2pt,"c"] & F_{k\e_{\sigma\tS_1},\bw-k\e_{\sigma\tS_1}}\ar[r,"j^+"]\ar[d,shift left=2pt,"c"] & \mathcal{M}_0(\bw,\mathcal{R})\ar[d,shift left=2pt,"c"]\\
			\mathcal{M}_0(\bw-k\e_{\sigma\tS_1},\mathcal{R})\ar[r,"i^-"]\ar[u,shift left=2pt,"c^{-1}"] & F_{\bw-k\e_{\sigma\tS_1},k\e_{\sigma\tS_1}}\ar[r,"j^-"]\ar[u,shift left=2pt,"c^{-1}"] & \mathcal{M}_0(\bw,\mathcal{R})\ar[u,shift left=2pt,"c^{-1}"]
		\end{tikzcd}\]
		Since $c$ is an isomorphism, we then have
		\[(c^{-1})^*(i^-)^!(j^-)^*c^*=(c^{-1})^*(i^-)^!c^*(j^+)^*=(c^{-1})^*c^*(i^+)^!(j^+)^*=(i^+)^!(j^+)^*\]
		whence by (\ref{eq:rank I coefficient b symmetry-2}),
		\begin{equation*}
			\begin{aligned}
				(i^+)^!(j^+)^*\mathcal{L}(\bv',\bw)&=(c^{-1})^*(i^-)^!(j^-)^*c^*\mathcal{L}(\bv',\bw)=(c^{-1})^*(i^-)^!(j^-)^*\mathcal{L}(\bv'^c,\bw)\\
				&=(c^{-1})^*\bigoplus_{\bv''^c\leq\bv'^c}\overline{b_{\bv'^c,\bv''^c}}\cdot \mathcal{L}(0,k\e_{\sigma\tS_1}) \boxtimes \mathcal{L}(\bv''^c,\bw-k\e_{\sigma\tS_1})[k(\bv''_1-\bv''_2)]\\
				&=\bigoplus_{\bv''\leq\bv'}\overline{b_{\bv'^c,\bv''^c}}\cdot \mathcal{L}(0,k\e_{\sigma\tS_1}) \boxtimes \mathcal{L}(\bv'',\bw-k\e_{\sigma\tS_1})[k(\bv''_1-\bv''_2)]
			\end{aligned}
		\end{equation*}
		Comparing with (\ref{eq:rank I coefficient b symmetry-1}) gives $\overline{b_{\bv',\bv''}}=b_{\bv'^c,\bv''^c}$.
	\end{proof}

	We may henceforth assume that $\bv_1+\bv_2\leq\min\{\bw_1,\bw_2\}$, so that $\bv$ is the maximal element in $I_{\bv}(\bw)$. We denote by $v^{-\dim\mathcal{M}_0^{\reg}(\bv',\bw)}P_{\bv',\bv''}$ the Poincar\'e polynomial of the stalk of $\mathcal{L}(\bv',\bw)$ at a generic point of $\mathcal{M}_0^{\reg}(\bv',\bw,\mathcal{R})$. It is well-known that
	\[\deg(P_{\bv',\bv''})<\dim\mathcal{M}_0^{\reg}(\bv',\bw,\mathcal{R})-\dim\mathcal{M}_0^{\reg}(\bv'',\bw,\mathcal{R}).\]
	Consider the decomposition of $\pi(\bv,\bw)$. Let $\bv''\in I_{\bv}(\bw)$ be such that $0\leq \tr_{\bv,\bv''}\leq k$; by taking stalk at $x_{\bv'',\bw}$, we find that 
	\begin{equation}\label{eq:rank I Poincare poly-1}
		\begin{aligned}
			v^{-\dim\mathcal{M}(\bv,\bw,\mathcal{R})} P(F_{\bv,\bv''}) =& a_{\bv,\bv} \cdot v^{-\dim\mathcal{M}(\bv,\bw,\mathcal{R})} P_{\bv,\bv''} + a_{\bv,\bv''} v^{-\dim\mathcal{M}_0^{\reg}(\bv'',\bw,\mathcal{R})} \\
			&+ \sum_{\bv'\in I_{\bv}(\bw)\setminus\{\bv\},\bv'>\bv''} a_{\bv,\bv'} \cdot v^{-\dim\mathcal{M}_0^{\reg}(\bv',\bw,\mathcal{R})} P_{\bv',\bv''}.
		\end{aligned}
	\end{equation}
	where $P(F_{\bv,\bv''})$ is the Poincar\'e polynomial of the fiber $F_{\bv,\bv''}$. By our assumption on $\bv$, it is easy to see that $a_{\bv,\bv}=1$. Therefore, applying \eqref{eq:rank I defection number}, the equation (\ref{eq:rank I Poincare poly-1}) gives 
	\begin{equation}\label{eq:rank I Poincare poly-2}
		\begin{aligned}
			P(F_{\bv,\bv''}) =& P_{\bv,\bv''} + v^{2\dim F_{\bv,\bv''}-\tr_{\bv,\bv''}(k-\tr_{\bv,\bv''})} a_{\bv,\bv''} \\
			&+ \sum_{\bv'\in I_{\bv}(\bw)\setminus\{\bv\},\bv'>\bv''} v^{2\dim F_{\bv,\bv'}-\tr_{\bv,\bv'}(k-\tr_{\bv,\bv'})} a_{\bv,\bv'}P_{\bv',\bv''}.
		\end{aligned}
	\end{equation}
	Note that the polynomial $P_{\bv,\bv''}$ satisfies
	\begin{equation}\label{eq:rank I Poincare poly-3}
		\deg(P_{\bv,\bv''}) < \dim\mathcal{M}_0^{\reg}(\bv,\bw,\mathcal{R}) - \dim\mathcal{M}_0^{\reg}(\bv'',\bw,\mathcal{R}) = 2\dim F_{\bv,\bv''}-\tr_{\bv,\bv''}(k-\tr_{\bv,\bv''}).
	\end{equation}
	Also, from \eqref{eq:rank I res of IC to diagonal-6} we see that $a_{\bv,\bv'}$ has degree $\leq \tr_{\bv,\bv'}(k-\tr_{\bv,\bv'})$, so 
	\begin{equation}\label{eq:rank I Poincare poly-4}
		\begin{aligned}
			&\deg\big(a_{\bv,\bv'} \cdot v^{2\dim F_{\bv,\bv'}-\tr_{\bv,\bv'}(k-\tr_{\bv,\bv'})} P_{\bv',\bv''}\big) \\
			&\leq \tr_{\bv,\bv'}(k-\tr_{\bv,\bv'})+2\dim F_{\bv,\bv'}-\tr_{\bv,\bv'}(k-\tr_{\bv,\bv'})+\deg(P_{\bv',\bv''})\\
			&=\tr_{\bv,\bv'}(k-\tr_{\bv,\bv'})+\dim\mathcal{M}(\bv,\bw,\mathcal{R})-\dim\mathcal{M}_0^{\reg}(\bv',\bw,\mathcal{R})+\deg(P_{\bv',\bv''})\\
			&=2\dim F_{\bv,\bv'}+\deg(P_{\bv',\bv''}).
		\end{aligned}
	\end{equation}
	
	\begin{lemma}\label{lem:rank I IC stalk inequality}
		For any $\bv''\leq\bv'$, we have
		\begin{equation}\label{IC sheaf stalk Poincare inequality-1}
			\deg(P_{\bv',\bv''})<\dim\mathcal{M}_0^{\reg}(\bv',\bw,\mathcal{R})-\dim\mathcal{M}_0^{\reg}(\bv'',\bw,\mathcal{R})-|\tr_{\bv',\bv''}|(k-|\tr_{\bv',\bv''}|).
		\end{equation}
	\end{lemma}
	\begin{proof}
		Since $\bv''\leq\bv'$, we can write $\bv''=\bv'-s\e_{\tS_1}-t\e_{\Sigma\tS_1}$. Moreover, as $P_{\bv',\bv''}=P_{\bv'^c,\bv''^c}$, we may assume that $s>t$ (the case $s=t$ being trivial). In this case, $\mathcal{L}(\bv''^c,\bw)$ does not appear in the decomposition of $\pi(\bv'^c,\bw)$ (this can be seen from \eqref{eq:rank I res of IC to diagonal-3}), so taking stalk of $\mathcal{L}(\bv''^c,\bw,\mathcal{R})$ at $\mathcal{M}_0^{\reg}(\bv''^c,\bw,\mathcal{R})$ shows that
		\begin{align*}
			\deg(P_{\bv',\bv''})&=\deg(P_{\bv'^c,\bv''^c})\leq 2\dim F_{\bv'^c,\bv''^c}\\
			&=\dim\mathcal{M}_0^{\reg}(\bv'^c,\bw,\mathcal{R})-\dim\mathcal{M}_0^{\reg}(\bv''^c,\bw,\mathcal{R})+(t-s)(k-(t-s))\\
			&=\dim\mathcal{M}_0^{\reg}(\bv',\bw,\mathcal{R})-\dim\mathcal{M}_0^{\reg}(\bv'',\bw,\mathcal{R})-(s-t)(k+(s-t))\\
			&<\dim\mathcal{M}_0^{\reg}(\bv',\bw,\mathcal{R})-\dim\mathcal{M}_0^{\reg}(\bv'',\bw,\mathcal{R})-(s-t)(k-(s-t))\\
			&=\dim\mathcal{M}_0^{\reg}(\bv',\bw,\mathcal{R})-\dim\mathcal{M}_0^{\reg}(\bv'',\bw,\mathcal{R})-\tr_{\bv',\bv''}(k-\tr_{\bv',\bv''}),
		\end{align*}
		where we have used \eqref{eq:rank I defection number}.
	\end{proof}
	
	\begin{proposition}
		The coefficient $a_{\bv,\bv''}$ is nonzero for any $\bv''\in I_\bv(\bw)$ such that $0\leq \tr_{\bv,\bv''}\leq k$.
	\end{proposition}
	\begin{proof}
		Combined with Lemma~\ref{lem:rank I IC stalk inequality}, the equation (\ref{eq:rank I Poincare poly-4}) implies that
		\begin{equation*}
			\begin{aligned}
				&\deg\big(v^{2\dim F_{\bv,\bv'}-\tr_{\bv,\bv'}(k-\tr_{\bv,\bv'})} a_{\bv,\bv'} P_{\bv',\bv''}\big)\\
				&<2\dim F_{\bv,\bv'} + \dim\mathcal{M}_0^{\reg}(\bv',\bw,\mathcal{R}) - \dim\mathcal{M}_0^{\reg}(\bv'',\bw,\mathcal{R}) - |\tr_{\bv',\bv''}|(k-|\tr_{\bv',\bv''}|)\\
				&=\dim\mathcal{M}(\bv,\bw,\mathcal{R}) + \tr_{\bv,\bv'}(k-\tr_{\bv,\bv'}) - \dim\mathcal{M}_0^{\reg}(\bv'',\bw,\mathcal{R}) - |\tr_{\bv',\bv''}|(k - |\tr_{\bv',\bv''}|)\\
				&=2\dim F_{\bv,\bv''} + \tr_{\bv,\bv'}(k - \tr_{\bv,\bv'}) - \tr_{\bv,\bv''}(k - \tr_{\bv,\bv''}) - |\tr_{\bv',\bv''}|(k - |\tr_{\bv',\bv''}|).
			\end{aligned}
		\end{equation*}
		Recall that $\tr_{\bv',\bv''}=\tr_{\bv,\bv''}-\tr_{\bv,\bv'}$. If $\tr_{\bv',\bv''}\geq 0$, then 
		\begin{align*}
			&\tr_{\bv,\bv'}(k - \tr_{\bv,\bv'}) - \tr_{\bv,\bv''}(k - \tr_{\bv,\bv''}) - |\tr_{\bv',\bv''}|(k-|\tr_{\bv',\bv''}|)\\
			&=\tr_{\bv,\bv'}(k - \tr_{\bv,\bv'}) - \tr_{\bv,\bv''}(k - \tr_{\bv,\bv''}) - \tr_{\bv',\bv''}(k-\tr_{\bv',\bv''})\\
			&=k(\tr_{\bv,\bv'} - \tr_{\bv,\bv''} - \tr_{\bv',\bv''}) + \tr_{\bv,\bv''}^2-\tr_{\bv,\bv'}^2 + \tr_{\bv',\bv''}^2\\
			&=-2k\tr_{\bv',\bv''} + 2\tr_{\bv,\bv''} - 2\tr_{\bv,\bv'}\tr_{\bv,\bv''}\\
			&=2\tr_{\bv,\bv''}(\tr_{\bv',\bv''} - k)\leq 0.
		\end{align*}
		On the other hand, for $\tr_{\bv',\bv''}\leq 0$, we have
		\begin{align*}
			&\tr_{\bv,\bv'}(k - \tr_{\bv,\bv'}) - \tr_{\bv,\bv''}(k - \tr_{\bv,\bv''}) - |\tr_{\bv',\bv''}|(k - |\tr_{\bv',\bv''}|)\\
			&=\tr_{\bv,\bv'}(k - \tr_{\bv,\bv'}) - \tr_{\bv,\bv''}(k - \tr_{\bv,\bv''}) + \tr_{\bv',\bv''}(k + \tr_{\bv',\bv''})\\
			&=\tr_{\bv,\bv''}^2 - \tr_{\bv,\bv'}^2 + \tr_{\bv',\bv''}^2 = 2\tr_{\bv,\bv''}^2 - 2\tr_{\bv,\bv'}\tr_{\bv,\bv''}\\
			&=2\tr_{\bv,\bv''}\tr_{\bv',\bv''}\leq 0.
		\end{align*}
		So the following inequality is valid:
		\[\deg\big(v^{2\dim F_{\bv,\bv'}-\tr_{\bv,\bv'}(k-\tr_{\bv,\bv'})} a_{\bv,\bv'} P_{\bv',\bv''}\big)< 2\dim F_{\bv,\bv''}.\]
		Since the polynomial $P(F_{\bv,\bv''})$ has degree $2\dim F_{\bv,\bv''}$, we conclude from \eqref{eq:rank I Poincare poly-2} that $a_{\bv,\bv''}\neq 0$.
	\end{proof}

	\begin{corollary}\label{eq:rank I b_vv only nonzero}
		$b_{\bv',\bv''}\neq 0$ only if $\tr_{\bv',\bv''}=0$.
	\end{corollary}
	\begin{proof}
		Since $a_{\bv,\bv''}\neq 0$ for any $\bv''\in I_{\bv}(\bw)$, $0\leq \tr_{\bv,\bv''}\leq k$, the equation \eqref{eq:rank I res of IC to diagonal-6} implies that $b_{\bv',\bv''}\neq 0$ only if $0\leq \tr_{\bv,\bv''}\leq k$, where $\bv$ is any vector satisfying $\bv'\in I_{\bv}(\bw)$ and $0\leq \tr_{\bv,\bv'}\leq k$. In particular, taking $\bv=\bv'$ shows that $b_{\bv',\bv''}\neq 0$ only if $0\leq \tr_{\bv',\bv''}\leq k$. But the same arguments apply to $b_{\bv'^c,\bv''^c}$ and we have $r_{\bv'^c,\bv''^c}=-\tr_{\bv',\bv''}$, so Lemma~\ref{eq:rank I coefficient b symmetry} implies the claim.
	\end{proof}
	
	With Corollary~\ref{eq:rank I b_vv only nonzero}, equation \eqref{eq:rank I res of IC to diagonal-6} can be simplified as 
	\begin{equation}\label{eq:rank I res of IC to diagonal-7}
		\qbinom{k}{\tr_{\bv,\bv''}} = \sum_{\substack{ \bv''\leq\bv' \\ \tr_{\bv',\bv''}=0} } a_{\bv,\bv'} b_{\bv',\bv''}.
	\end{equation}
	
	\begin{proposition}\label{prop:rank I res to diagonal coefficient}
		$a_{\bv,\bv''}=\qbinom{k}{\tr_{\bv,\bv''}}$ and $b_{\bv',\bv''}=\delta_{\bv',\bv''}$.
	\end{proposition}
	\begin{proof}
		Write $\bv''=\bv-s\e_{\tS_1}-t\e_{\Sigma\tS_1}$ for $0\leq s-t\leq k$, we prove the proposition by induction on $t$. For the case $t=0$, the only sequence $\bv'\in I_{\bv}(\bw)$ such that $\bv'\geq\bv''$, $\tr_{\bv',\bv''}=0$ is $\bv'=\bv''$, so (\ref{eq:rank I res of IC to diagonal-7}) gives $a_{\bv,\bv''}=\qbinom{k}{s}$ and $b_{\bv',\bv''}=\delta_{\bv',\bv''}$. Now assume that the assertion is proved for any $\bv''=\bv-a\e_{\tS_1}-b\e_{\Sigma\tS_1}$ with $0\leq b<t$. By \eqref{eq:rank I res of IC to diagonal-7}, we have 
		\begin{align*}
			\qbinom{k}{s-t}&=\sum_{i=1}^{t}a_{\bv,\bv-(s-i)\e_{\tS_1}-(t-i)\e_{\Sigma\tS_1}} b_{\bv-(s-i)\e_{\tS_1}-(t-i)\e_{\Sigma\tS_1},\bv-s\e_{\tS_1}-t\e_{\Sigma\tS_1}}+a_{\bv,\bv-s\e_{\tS_1}-t\e_{\Sigma\tS_1}}\\
			&=\sum_{i=1}^{t}\qbinom{k}{s-t} b_{\bv-(s-i)\e_{\tS_1}-(t-i)\e_{\Sigma\tS_1},\bv-s\e_{\tS_1}-t\e_{\Sigma\tS_1}}+a_{\bv,\bv-s\e_{\tS_1}-t\e_{\Sigma\tS_1}}.
		\end{align*}
		Since each $b_{\bv-(s-i)\e_{\tS_1}-(t-i)\e_{\Sigma\tS_1},\bv-s\e_{\tS_1}-t\e_{\Sigma\tS_1}}$ and $a_{\bv,\bv-s\e_{\tS_1}-t\e_{\Sigma\tS_1}}$ belong to $\N[v,v^{-1}]$ and the latter is nonzero, the only possibility is that $b_{\bv-(s-i)\e_{\tS_1}-(t-i)\e_{\Sigma\tS_1},\bv-s\e_{\tS_1}-t\e_{\Sigma\tS_1}}=0$ for $1\leq i\leq t$. This completes the induction step.
	\end{proof}

	\begin{proposition}
		For any pair $(\bv,\bw)$, we have
		\[\pi(\bv,\bw)=\begin{dcases}
			\bigoplus_{\substack{\bv'\in I_{\bv}(\bw)\\ 0\leq \tr_{\bv,\bv'}\leq \bw_1-\bw_2}} \qbinom{\bw_1-\bw_2}{\tr_{\bv,\bv'}} \mathcal{L}(\bv',\bw)&\text{if $\bw_1-\bw_2\geq 0$},\\
			\bigoplus_{\substack{\bv'\in I_{\bv}(\bw)\\ \bw_1-\bw_2\leq \tr_{\bv,\bv'}\leq 0}} \qbinom{-\bw_1+\bw_2}{-\tr_{\bv,\bv'}} \mathcal{L}(\bv',\bw)&\text{if $\bw_1-\bw_2\leq 0$}.
		\end{dcases}\]
	\end{proposition}
	\begin{proof}
		Since we have already determined the coefficients $b_{\bv,\bv'}$, the first case follows from \eqref{eq:rank I res of IC to diagonal-6}. The case $\bw_1\leq\bw_2$ can be treated in a similar way.
	\end{proof}
	
	With the decomposition of $\pi(\bv,\bw)$, we can now deduce the explicit expression of $L(\bv,\bw)$ in terms of the generators of $\tR$. To simplify our notation we set
	\[E:=L(0,\e_{\sigma\tS_1}),\quad F:=L(0,\e_{\sigma\Sigma\tS_1}),\quad K:=L(\bv^1,\bw^1), \quad K'=L(\bv^{\Sigma 1},\bw^1).\]
	
	From Proposition~\ref{prop:rank I res to diagonal coefficient} (and its analogue) we immediately get the following multiplication formulas:
	
	\begin{corollary}\label{prop:rank I mult by EF formula}
		For any strongly $l$-dominant pair $(\bv,\bw)$ such that $\bw_1=\bw_2$, we have
		\begin{alignat*}{2}
			E^k\cdot L(\bv,\bw)&=v^{k(\bv_2-\bv_1)}L(\bv,\bw+k\e_{\sigma\tS_1}),&&\quad L(\bv,\bw)\cdot E^k=v^{k(\bv_1-\bv_2)}L(\bv,\bw+k\e_{\tS_1}),\\
			F^k\cdot L(\bv,\bw)&=v^{k(\bv_1-\bv_2)}L(\bv,\bw+k\e_{\sigma\Sigma\tS_1}),&&\quad L(\bv,\bw)\cdot F^k=v^{k(\bv_2-\bv_1)}L(\bv,\bw+k\e_{\sigma\Sigma\tS_1}).
		\end{alignat*}
	\end{corollary}

	\begin{proposition}
		For $a,b\geq 0$, we have
		\begin{align}
			\label{eq:rank I EF multiply}
			E^aF^b=\begin{dcases}
				\sum_{0\leq \bv_1+\bv_2\leq b}v^{a(\bv_2-\bv_1)}\qbinom{b+1}{\bv_1}\qbinom{b+1}{\bv_2}\frac{[b+1-\bv_1-\bv_2]}{[b+1]}L(\bv,a\e_{\sigma\tS_1}+b\e_{\sigma\Sigma\tS_1})&\text{if $a\geq b$},\\
				\sum_{0\leq \bv_1+\bv_2\leq a}v^{b(\bv_2-\bv_1)}\qbinom{a+1}{\bv_1}\qbinom{a+1}{\bv_2}\frac{[a+1-\bv_1-\bv_2]}{[a+1]}L(\bv,a\e_{\sigma\tS_1}+b\e_{\sigma\Sigma\tS_1})&\text{if $b\geq a$}.
			\end{dcases}
		\end{align}
	\end{proposition}
	\begin{proof}
		In view of Corollary~\ref{prop:rank I mult by EF formula}, we may assume $a=b$. In this case we note
		\begin{equation}\label{eq:rank I EF multiply-1}
			\pi(\bv,b\e_{\sigma\tS_1}+b\e_{\sigma\Sigma\tS_1})=\bigoplus_{i=0}^{\min\{\bv_1,\bv_2\}}\mathcal{L}(\bv-i\e_{\sigma\tS_1}-i\e_{\sigma\Sigma\tS_1},b\e_{\sigma\tS_1}+b\e_{\sigma\Sigma\tS_1}).
		\end{equation}
		Then Example~\ref{eg:rank I rest to EF formula} gives 
		\[ \Res^{b\e_{\sigma\tS_1}+b\e_{\sigma\Sigma\tS_1}}_{b\e_{\sigma\tS_1},b\e_{\sigma\Sigma\tS_1}}(\pi(\bv,b\e_{\sigma\tS_1}+b\e_{\sigma\Sigma\tS_1})) = v^{b(\bv_2-\bv_1)} \qbinom{b}{\bv_1}\qbinom{b}{\bv_2} \mathcal{L}(0,b\e_{\sigma\tS_1})\boxtimes \mathcal{L}(0,b\e_{\sigma\Sigma\tS_1}),
		\]
		so by \eqref{eq:rank I EF multiply-1}, we get
		\begin{align*}
			&\Res^{b\e_{\sigma\tS_1}+b\e_{\sigma\Sigma\tS_1}}_{b\e_{\sigma\tS_1},b\e_{\sigma\Sigma\tS_1}}(\mathcal{L}(\bv,b\e_{\sigma\tS_1}+b\e_{\sigma\Sigma\tS_1}))\\
			&=v^{b(\bv_2-\bv_1)} \Big(\qbinom{b}{\bv_1}\qbinom{b}{\bv_2} - \qbinom{b}{\bv_1-1}\qbinom{b}{\bv_2-1}\Big) \mathcal{L}(0,b\e_{\sigma\tS_1})\boxtimes \mathcal{L}(0,b\e_{\sigma\Sigma\tS_1})\\
			&=v^{b(\bv_2-\bv_1)}\qbinom{b+1}{\bv_1}\qbinom{b+1}{\bv_2}\frac{[b+1-\bv_1-\bv_2]}{[b+1]} \mathcal{L}(0,b\e_{\sigma\tS_1})\boxtimes \mathcal{L}(0,b\e_{\sigma\Sigma\tS_1}).
		\end{align*}
		Taking dual, this gives the desired formula.
	\end{proof}
	
	To deduce the explicit formula for $L(\bv,\bw)$, we need to express $E^aF^bK^cK'^d$ in terms of the dual canonical basis. The following lemma tells us how to derive this from \eqref{eq:rank I EF multiply}.
	
	\begin{lemma}[{cf. \cite[Lemma 4.19]{LP26}}]\label{lem: multiply by L^0}
		For any $l$-dominant pairs $(\bv,\bw)$ and $(\bv^0,\bw^0)\in ({V^0},{W^0})$ such that $\sigma^*\bw^0-\mathcal{C}_q\bv^0=0$, we have 
		\[L(\bv^0,\bw^0)\cdot L(\bv,\bw)\in v^{\frac{1}{2}\Z}L(\bv+\bv^0,\bw+\bw^0).\]
	\end{lemma}
	
	\begin{lemma}\label{prop:rank I mult by K formula}
		For any strongly $l$-dominant pair $(\bv,\bw)$, we have 
		\begin{alignat*}{2}
			K\cdot L(\bv,\bw)&=v^{\bw_1-\bw_2}L(\bv+\e_{\sigma\tS_1},\bw),\quad\quad &&
			L(\bv,\bw)\cdot K=v^{\bw_2-\bw_1}L(\bv+\e_{\sigma\tS_1},\bw)\\
			K'\cdot L(\bv,\bw)&=v^{\bw_2-\bw_1}L(\bv+\e_{\sigma\Sigma\tS_1},\bw),\quad\quad &&
			L(\bv,\bw)\cdot K'=v^{\bw_1-\bw_2}L(\bv+\e_{\sigma\Sigma\tS_1},\bw).
		\end{alignat*}
	\end{lemma}
	\begin{proof}
		By Lemma~\ref{lem: multiply by L^0} we have
		\[K\cdot L(\bv,\bw)=v^{a_\bw}L(\bv+\e_{\sigma\tS_1},\bw)=\ov{L(\bv,\bw)\cdot K},\]
		and similarly for $K'$. The integer $a_\bw$ can be determined using the decomposition of $\pi(\bv,\bw)$ and \eqref{eq:rank I res formula}.
	\end{proof}

	\begin{proposition}\label{prop:rank I L(vw) formula}
		For any strongly $l$-dominant pair $(\bv,\bw)$, we have
		\begin{align}
			\label{eq:rank I L formula}
			L(\bv,\bw)=\sum_{k=l}^{n}(-1)^{k-l} \sum_{\substack{\bv'_1\geq\bv_1,\bv'_2\geq\bv_2 \\ \bv'_1+\bv'_2=k}} 
			v^{f_{\bw}(\bv,\bv')} \qbinom{n-\bv'_2-\bv_1}{\bv'_1-\bv_1} \qbinom{n-\bv'_1-\bv_2}{\bv'_2-\bv_2} E^{\bw_1-k}F^{\bw_2-k}K^{\bv'_1}K'^{\bv'_2},
		\end{align}
		where $n=\min\{\bw_1,\bw_2\}$, $l=\bv_1+\bv_2$, and
		\[f_{\bw}(\bv,\bv')=(\bw_1-\bw_2)(\bv_1-\bv_2)+(n+1-\bv'_1-\bv'_2)[(\bv_1-\bv_2)-(\bv'_1-\bv'_2)].\]
	\end{proposition}
	
	\begin{proof}
		We need to prove that \eqref{eq:rank I L formula} is the inverse formula to \eqref{eq:rank I EF multiply}. In view of Corollary~\ref{prop:rank I mult by EF formula}, it suffices to consider the case $a=b$. Plugging \eqref{eq:rank I L formula} into \eqref{eq:rank I EF multiply} and applying Lemma~\ref{prop:rank I mult by K formula}, we obtain
		\begin{align*}
			E^bF^bK'^cK^d&=\sum_{0\leq\bv_1+\bv_2\leq b}\sum_{\substack{\bv'_1+\bv'_2\leq b+c+d\\ \bv'_1\geq\bv_1+c,\bv'_2\geq\bv_2+d}} (-1)^{(\bv_1'+\bv_2')-(\bv_1+c+\bv_2+d)} v^{b(\bv_2-\bv_1)+f_{\bw}(\bv+\mathbf{u},\bv')}\\
			&\ \cdot \qbinom{b+1}{\bv_1}\qbinom{b+1}{\bv_2}\frac{[b+1-\bv_1-\bv_2]}{[b+1]}\qbinom{b+d-\bv'_2-\bv_1}{\bv'_1-\bv_1-c} \qbinom{b+c-\bv'_1-\bv_2}{\bv'_2-\bv_2-d}\\
			&\ \cdot E^{b+c+d-k}F^{b+c+d-k} K'^{\bv'_1}K^{\bv'_2}\\
			&=\sum_{0\leq \bv_1+\bv_2\leq b}\sum_{\substack{\bv'_1+\bv'_2\leq d\\ \bv'_1\geq\bv_1,\bv'_2\geq\bv_2}} (-1)^{(\bv_1'+\bv_2')-(\bv_1+\bv_2)} v^{b(\bv_2-\bv_1)+f_{\bw}(\bv+\mathbf{u},\bv'+\mathbf{u})}\\
			&\ \cdot \qbinom{b+1}{\bv_1}\qbinom{b+1}{\bv_2}\frac{[b+1-\bv_1-\bv_2]}{[b+1]}\qbinom{b-\bv'_2-\bv_1}{\bv'_1-\bv_1} \qbinom{b-\bv'_1-\bv_2}{\bv'_2-\bv_2}\\
			&\ \cdot E^{b+c+d-k}F^{b+c+d-k} K'^{\bv'_1+c}K^{\bv'_2+d}
		\end{align*}
		where we write
		\[\bw=(b+c+d)\e_{\sigma\tS_1}+(b+c+d)\e_{\sigma\Sigma\tS_1},\quad \mathbf{u}=c\e_{\tS_1}+d\e_{\Sigma\tS_1}.\]
		The coefficient of $E^bF^bK'^cK^d$ in the right-hand side is easily seen to be $1$, so it remains to prove that for any (fixed) $\bv'_1,\bv'_2\geq 0$ such that $0<\bv_1'+\bv'_2\leq b$, the following equation holds:
		\begin{align*}
			\sum_{0\leq \bv\leq\bv'}
			&(-1)^{\bv_1+\bv_2} \cdot v^{(\bv'_1+\bv'_2-1)(\bv_1-\bv_2)} \qbinom{b+1}{\bv_1} \qbinom{b+1}{\bv_2}  \\
			&\frac{[b+1-\bv_1-\bv_2]}{[b+1]}\qbinom{b-\bv'_2-\bv_1}{\bv'_1-\bv_1} \qbinom{b-\bv'_1-\bv_2}{\bv'_2-\bv_2}=0.
		\end{align*}
		On the other hand, recall that in the proof of (\ref{eq:rank I EF multiply}) we have seen the following identity:
		\[\qbinom{b}{\bv_1}\qbinom{b}{\bv_2}-\qbinom{b}{\bv_1-1}\qbinom{b}{\bv_2-1}=\qbinom{b+1}{\bv_1}\qbinom{b+1}{\bv_2}\frac{[b+1-\bv_1-\bv_2]}{[b+1]}.\]
		So it suffices to prove the following lemma.
	\end{proof}
	
	\begin{lemma}
		For any non-negative integers $\bv'_1,\bv'_2$ such that $0<\bv'_1+\bv'_2\leq b$, set
		\begin{align*}
			J_1=\sum_{0\leq \bv\leq\bv'}
			&(-1)^{\bv_1+\bv_2} \cdot v^{(\bv'_1+\bv'_2-1)(\bv_1-\bv_2)} \qbinom{b}{\bv_1} \qbinom{b}{\bv_2} \qbinom{b-\bv'_2-\bv_1}{\bv'_1-\bv_1} \qbinom{b-\bv'_1-\bv_2}{\bv'_2-\bv_2},\\
			J_2=\sum_{0\leq \bv\leq\bv'}
			&(-1)^{\bv_1+\bv_2} \cdot v^{(\bv'_1+\bv'_2-1)(\bv_1-\bv_2)} \qbinom{b}{\bv_1-1} \qbinom{b}{\bv_2-1} \qbinom{b-\bv'_2-\bv_1}{\bv'_1-\bv_1} \qbinom{b-\bv'_1-\bv_2}{\bv'_2-\bv_2}.
		\end{align*}
		Then $J_1=J_2$.
	\end{lemma}
	\begin{proof}
		We shall make use the following standard binomial identities:
		\begin{gather}
			\qbinom{x}{n}=(-1)^n \qbinom{n-x-1}{n}, \label{binomial identity-1}\\
			\sum_{k=0}^n v^{xk -y(n-k)} \qbinom{x}{n-k} \qbinom{y}{k} = \qbinom{x+y}{n}, \label{binomial identity-2}
		\end{gather}
		where $x,y,m$ are integers and $n\geq 0$. Applying \eqref{binomial identity-1} to $J_1$, we find that 
		\begin{align*}
			&\sum_{\bv_1=0}^{\bv'_1}(-1)^{\bv_1}  v^{\bv_1(\bv'_1+\bv'_2-1)} \qbinom{b}{\bv_1} \qbinom{b-\bv'_2-\bv_1}{\bv'_1-\bv_1} \sum_{\bv_2=0}^{\bv'_2}(-1)^{\bv_2} v^{-\bv_2(\bv'_1+\bv'_2-1)} \qbinom{b}{\bv_2} \qbinom{b-\bv'_1-\bv_2}{\bv'_2-\bv_2}\\
			=&\sum_{\bv_1=0}^{\bv'_1}(-1)^{\bv'_1} v^{\bv_1(\bv'_1+\bv'_2-1)} \qbinom{b}{\bv_1} \qbinom{\bv'_1+\bv'_2-b-1}{\bv'_1-\bv_1} \\
			&\cdot\sum_{\bv_2=0}^{\bv'_2}(-1)^{\bv'_2} v^{-\bv_2(\bv'_1+\bv'_2-1)} \qbinom{b}{\bv_2} \qbinom{\bv'_1+\bv'_2-b-1}{\bv'_2-\bv_2}.
		\end{align*}
		Now we apply (\ref{binomial identity-2})  to give
		\begin{align*}
			\sum_{\bv_1=0}^{\bv'_1}v^{(b+1)\bv'_1-\bv_1(\bv'_1+\bv'_2)} \qbinom{b}{\bv_1} \qbinom{\bv'_1+\bv'_2-b-1}{\bv'_1-\bv_1}=\qbinom{\bv'_1+\bv'_2-1}{\bv'_1},\\
			\sum_{\bv_2=0}^{\bv'_2}v^{(b+1)\bv'_2-\bv_2(\bv'_1+\bv'_2)} \qbinom{b}{\bv_2} \qbinom{\bv'_1+\bv'_2-b-1}{\bv'_2-\bv_2}=\qbinom{\bv'_1+\bv'_2-1}{\bv'_2},
		\end{align*}
		from which we conclude
		\begin{align*}
			J_1&=(-1)^{\bv'_1}v^{b\bv'_1}\qbinom{\bv'_1+\bv'_2-1}{\bv'_1}\cdot (-1)^{\bv'_2}v^{-b\bv'_2}\qbinom{\bv'_1+\bv'_2-1}{\bv'_2}\\
			&=(-1)^{\bv'_1+\bv'_2} v^{b(\bv'_1-\bv'_2)} \qbinom{\bv'_1+\bv'_2-1}{\bv'_1} \qbinom{\bv'_1+\bv'_2-1}{\bv'_2}.
		\end{align*}
		A similar argument applies to $J_2$, in which case we get 
		\begin{align*}
			J_2=&\sum_{\bv_1=1}^{\bv'_1} (-1)^{\bv_1}  v^{\bv_1(\bv'_1+\bv'_2-1)} \qbinom{b}{\bv_1-1} \qbinom{b-\bv'_2-\bv_1}{\bv'_1-\bv_1}\\
			&\cdot \sum_{\bv_2=1}^{\bv'_2} (-1)^{\bv_2} v^{-\bv_2(\bv'_1+\bv'_2-1)} \qbinom{b}{\bv_2-1} \qbinom{b-\bv'_1-\bv_2}{\bv'_2-\bv_2}\\
			=&\sum_{\bv_1=0}^{\bv'_1-1}(-1)^{\bv'_1} v^{(\bv_1+1)(\bv'_1+\bv'_2-1)} \qbinom{b}{\bv_1} \qbinom{\bv'_1+\bv'_2-b-1}{\bv'_1-\bv_1-1}\\
			&\cdot \sum_{\bv_2=0}^{\bv'_2-1}(-1)^{\bv'_2}v^{-(\bv_2+1)(\bv'_1+\bv'_2-1)} \qbinom{b}{\bv_1} \qbinom{\bv'_1+\bv'_2-b-1}{\bv'_2-\bv_2-1}\\
			=&(-1)^{\bv'_1} v^{(b+1)\bv'_1+\bv'_2-b-1} \qbinom{\bv'_1+\bv'_2-1}{\bv'_1-1} \cdot (-1)^{\bv'_2} v^{-(b+1)\bv'_2-\bv'_1+b+1} \qbinom{\bv'_1+\bv'_2-1}{\bv'_2-1}\\
			=&(-1)^{\bv'+\bv'_2}v^{b(\bv'_1-\bv'_2)}\qbinom{\bv'_1+\bv'_2-1}{\bv'_1-1}\qbinom{\bv'_1+\bv'_2-1}{\bv'_2-1}.
		\end{align*}
		It is clear that $J_1=J_2$ in any case (if $\bv'_1=0$ or $\bv'_2=0$ then $J_1$ and $J_2$ both reduce to $0$).
	\end{proof}
	
	\begin{remark}
		In \cite{BG17} the authors have given an explicit formula for the double canonical basis of $\tU_v(\mathfrak{sl}_2)$, where the quantum Casimir element
		\[C=EF-v^{-1}K-vK'=FE-vK-v^{-1}K'\]
		plays an important role. From \eqref{eq:rank I L formula} we see that $C=L(0,\e_{\sigma\tS_1}+\e_{\sigma\Sigma\tS_1})$. In fact, the elements $C^{(m)}$ defined inductively by
		\[C\cdot C^{(m)}=C^{(m+1)}+KK'C^{(m-1)},\quad\forall m\geq 1\]
		in \cite[Section 4.1]{BG17} is nothing but $L(0,m\e_{\sigma\tS_1}+m\e_{\sigma\Sigma\tS_1})$ in our notation, and in particular the double canonical basis of $\tU_v(\mathfrak{sl}_2)$ coincides with the dual canonical basis. It also coincides with the formula of the basis $\Theta$ constructed in \cite{Sh22} for $\tU_v(\mathfrak{sl}_2)$. 
	\end{remark}


\begin{thebibliography}{DDPW01}
		
		%\bibitem[As11]{As} H. Asashiba,
		%{\em A generalization of Gabriel's Galois covering functors and derived equivalences}, J. Algebra {\bf 334} (2011), 109--149.
		\bibitem[Ach21]{Ach21} P. Achar, Perverse sheaves and applications to representation theory, vol. {\bf258}, American Mathematical Society, 2021.
		
		%\bibitem[AB89]{AB89} M. Auslander and R. Buchweitz,
		%{\em The homological theory of maximal Cohen-Macaulay approximation}, Colloque en l'honneur de Pierre Samuel (Orsay, 1987). M\'{e}m. Soc. Math. France (N.S.) {\bf 38} (1989), 5--37.
		
		%\bibitem[BKLW18]{BKLW} H. Bao, J. Kujawa, Y. Li and W. Wang,
		%{\it  Geometric Schur duality of classical type}, Transform. Groups {\bf 23} (2018), 329--389.
		%\href{https://arxiv.org/abs/1404.4000}{arXiv:1404.4000v3}
		
		\bibitem[BBD82]{BBD} 
		A.A. Beilinson, J. Bernstein, P. Deligne, \emph{Faisceaux pervers}, Ast\'{e}risque {\bf100}, 1982.
		
		
		\bibitem[BG17]{BG17} A. Berenstein and J. Greenstein, {\em Double canonical bases}, Adv. Math. {\bf 316} (2017), 381--468.
		
		
		
		\bibitem[Bra03]{Bra03} T. Braden, {\em Hyperbolic localization of intersection cohomology}, Transform. Groups {\bf8} (2003), no. 3, 209--216.
		
		\bibitem[Br13]{Br13} T. Bridgeland, {\em Quantum groups via Hall algebras of complexes}, Ann. Math. {\bf 177} (2013), 739--759.
		
		\bibitem[CZ26]{CZ26} X. Chen, X. Zhou, {\em Bases of the quantum group and $\imath$quantum group of $\mathfrak{sl}_2$}, J. Algebra {\bf690} (2026), 425--474.
		
		
		
		%\bibitem[CM06]{CM06} C. Cibils, E. Marcos, {\em Skew category, Galois covering and smash product of a $k$-category}, Proc. Amer. Math. Soc. {\bf134}
		%(2006), 39--50.
		
		\bibitem[DDPW08]{DDPW}
		B.~Deng, J.~Du, B.~Parshall and J.~Wang,
		Finite dimensional algebras and quantum groups,
		Mathematical Surveys and Monographs {\bf 150}.
		AMS, Providence, RI, 2008.
		
		\bibitem[dCM02]{dM02} M. A. A. de Cataldo and L. Migliorini, {\em The Hard Lefschetz Theorem and the topology of semismall maps}, Ann. Sci. \'Ecole Norm. Sup. (4) {\bf35} (2002), no. 5, 759–772.
		
		
		
		%\bibitem[FLX23]{FLX23} J. Fang, Y. Lan and J. Xiao, {\em Sheaf realization of Bridgeland's Hall algebra of Dynkin type
		%}, \href{arXiv:2303.0499}{https://arxiv.org/abs/2303.04993}
		
		%\bibitem[FK88]{FK88} E. Freitag and R. Kiehl, Etale Cohomology and the Weil Conjecture, Springer, 1988.
		
		
		%\bibitem[Fu17]{Fu17} R. Fujita,
		%{\em Affine highest weight categories and quantum affine Schur-Weyl duality of Dynkin quiver types}, 	Represent. Theory {\bf26} (2022), 211--263.  %\href{https://arxiv.org/abs/1710.11288}{arXiv:1710.11288}
		
		%\bibitem[Ga79]{Ga79} P. Gabriel, {\em Auslander-Reiten sequences and representation-finite algebras}, Representation theory, I (Proc.
		%Workshop, Carleton Univ., Ottawa, Ont., 1979), Springer, Berlin, 1980,  1–-71.
		
		
		%\bibitem[Ga81]{Ga} P. Gabriel, {\em The universal cover of a representation finite algebra}, in: Representation of Algebras, in: Lecture Notes in Math., vol. {\bf903} (1981), 65--105.
		
		
		%\bibitem[GLS12]{GLS12}  C. Geiss, B. Leclerc, and J. Schr\"{o}er, {\em Generic bases for cluster algebras and the Chamber Ansatz}, J.
		%Amer. Math. Soc. {\bf25} (2012), no. 1, 21–76.
		
		\bibitem[GLS13]{GLS13}  C. Geiss, B. Leclerc, and J. Schr\"{o}er, {\em Cluster structures on quantum coordinate rings}, Selecta Math. (N.S.) {\bf19} (2013), 337--397.
		
		
		
		
		
		%\bibitem[Ha87]{Ha1} D. Happel, {\em On the deived category of a finite-dimesional algebra}, Comment. Math. Helv. {\bf62} (3)(1987), 339--389.
		
		\bibitem[Ha88]{Ha2} D. Happel, Triangulated Categories in the Representation Theory of Finite Dimensional Algebras. London Math. Soc. Lecture Notes Ser. {\bf119}, Cambridge Univ. Press, Cambridge, 1988.
		
		\bibitem[HL15]{HL15} D. Hernandez and B. Leclerc, {\em Quantum Grothendieck rings and derived Hall algebras}, J. Reine Angew. Math. {\bf701} (2015), 77--126.
		
		%\bibitem[I05]{I05} O. Iyama, $\tau$-categories. I: Ladders, Algebr. Represent. Theory {\bf8} (2205), no.3, 297--321.
		
		\bibitem[KKKO18]{KKKO18} S.-J. Kang, M. Kashiwara, M. Kim, and S.-j. Oh, {\em Monoidal categorification of cluster algebras}, J. Amer. Math. Soc. 31 (2018), 349--426.
		
		
		\bibitem[Ka91]{Ka91} M. Kashiwara,
		{\em On crystal bases of the $Q$-analogue of universal enveloping algebras}, Duke Math.~J.~{\bf 63} (1991), 456--516.
		
		%\bibitem[Ka14]{Ka14} S. Kato, {\em Poincar\'{e}-Birkhoff-Witt bases and Khovanov-Lauda-Rouquier algebras}, Duke Math. J. {\bf163} (2014), 619--663.
		
		%\bibitem[Ke90]{Ke1} B. Keller, {\em Chain complexes and stable categories}, Manus. Math. {\bf67} (1990), 379--417.
		
		
		\bibitem[KS16]{KS16} B. Keller and S. Scherotzke, {\em Graded quiver varieties and derived categories}, J. Reine Angew. Math. {\bf713} (2016), 85--127.
		
		%\bibitem[KKK15]{KKK15} S.-J. Kang, M. Kashiwara, M. Kim,
		%{\em Symmetric quiver Hecke algebras and $R$-matrices of quantum affine algebras, II}, Duke Math.~J.~{\bf 164} (2015), 1549--1602.
		
		%\bibitem[KL09]{KL09}
		%M. Khovanov and A. Lauda,
		%\emph{A diagrammatic approach to  categorification of quantum groups. {I}}, Represent. Theory \textbf{13}  (2009), 309--347.
		
		
		
		
		
		\bibitem[Ko14]{Ko14} S. Kolb, {\em Quantum symmetric Kac-Moody pairs}, Adv. Math. {\bf267} (2014), 395--469.
		
		
		
		%Publications Math\'ematiques de L’Institut des Hautes Scientifiques {\bf 65}, 131--210 (1987).
		
		
		
		
		
		
		
		
		%\bibitem[LeBP90]{LBP} L. Le Bruyn and C. Procesi, {\em Semisimple representations of quivers}, Trans.
		%Amer. Math. Soc. {\bf 317} (1990), 585--598.
		
		\bibitem[Let99]{Let99} G. Letzter, {\em Symmetric pairs for quantized enveloping algebras}, J. Algebra {\bf220} (1999), 729--767.
		\bibitem[LP25]{LP25} M. Lu and X. Pan, {\em Dual canonical bases of quantum groups and $\imath$quantum groups I: Hall algebras}, \href{https://arxiv.org/abs/2504.19073v2}{arXiv:2504.19073v2}.
		
		
		\bibitem[LP26]{LP26} M. Lu and X. Pan, {\em Dual canonical bases of quantum groups and $\imath$quantum groups II: geometry}, \href{http://arxiv.org/abs/2603.01350}{arXiv:2603.01350}.
		
		
		\bibitem[LR24]{LR24} M. Lu and S. Ruan, $\imath$Hall algebras of weighted projective lines and quantum symmetric pairs III:
		quasi-split type, \href{http://arxiv.org/abs/2411.13078}{arXiv:2411.13078}.
		
		
		
		
		%\bibitem[LW19b]{LW19b} M. Lu and W. Wang, {\em Hall algebras and quantum symmetric pairs II: Reflection functors}, \href{http://export.arxiv.org/abs/1904.01621}{arXiv:1901.01621}
		
		%\bibitem[LW19e]{LW19e} M. Lu and W. Wang, {\em Hall algebras and quantum symmetric pairs V: Bases}, in preparation.
		
		%\bibitem[Lu19]{Lu19} M. Lu,
		%{\em Modified Ringel-Hall algebras of 1-Gorenstein algebras},
		%Appendix A to \cite{LW19}, \href{http://arxiv.org/abs/1901.11446}{arXiv:1901.11446}
		
		\bibitem[LW21a]{LW21a} M. Lu and W. Wang, {\em Hall algebras and quantum symmetric pairs II: reflection functors},  Commun. Math. Phys.  {\bf 381} (2021), 799--855. %, \href{http://arxiv.org/abs/1904.01621}{arXiv:1904.01621}
		
		
		\bibitem[LW21b]{LW21b} M. Lu and W. Wang,
		{\em Hall algebras and quantum symmetric pairs III: Quiver varieties}, Adv. Math. {\bf393} (2021), 108071.
		
		
		\bibitem[LW22a]{LW19} M. Lu and W. Wang, {\em Hall algebras and quantum symmetric pairs I: Foundations}, Proc. London Math. Soc. {\bf 124} (1):  1--82.
		
		
		
		
		
		\bibitem[Lus90a]{Lus90a} G.~Lusztig,
		\textit{Finite-dimensional Hopf algebras arising from quantized universal enveloping algebra}, J. Amer. Math. Soc. {\bf 3} (1990), 257--296.
		
		\bibitem[Lus90b]{Lus90} G. Lusztig, {\em Canonical bases arising from quantized enveloping algebras}, J. Amer. Math.
		Soc. {\bf 3} (1990),  447--498.
		
		\bibitem[Lus91]{Lus91} G. Lusztig,  {\em Quivers, perverse sheaves, and quantized enveloping algebras}, J. Amer.
		Math. Soc. {\bf4} (2)(1991), 365--421, 1991.
		
		\bibitem[Lus93]{Lus93} G. Lusztig, Introduction to Quantum Groups, Birkh\"{a}user, Boston, 1993.
		
		
		\bibitem[Na01]{Na01} H. Nakajima, {\em Quiver varieties and finite-dimensional representations of quantum affine algebras}, J. Amer. Math. Soc. {\bf14} (2001), 145--238 (electronic).
		
		\bibitem[Na04]{Na04} H. Nakajima, {\em Quiver varieties and $t$-analogs of $q$-characters of quantum affine algebras}, Ann. Math. {\bf160} (2004), 1057--1097.
		
		%\bibitem[NV04]{NV04} C. N\v{a}st\v{a}sescu and F. Van Oystaeyen, Methods of Graded Rings, Lecture Notes in
		%Mathematics, {\bf1836}, Springer-Verlag, Berlin, 2004.
		
		
		\bibitem[Q16]{Qin} F. Qin, {\em Quantum groups via cyclic quiver varieties I}, Compos. Math. {\bf152} (2016), 299--326.
		
		%\bibitem[Q21]{Qin21} F. Qin, {\em Dual canonical bases and quantum cluster algebras}, \href{https://arxiv.org/abs/2003.13674v3}{arxiv:2003.13674}
		
		
		%\bibitem[Qui73]{Q} D. Quillen, {\em Higher algebraic K-theory, I}, Springer Lecture Notes in Mathematics {\bf 341} (1973), 85--147.
		
		%\bibitem[Rie86]{Rie86} C. Riedtmann, {\em Degenerations for representations of quivers with relations}, Ann. scient. \'{E}c. Norm. Sup. {\bf 19} (1986), 275--301.
		
		\bibitem[Rin90]{Rin90} C. Ringel, {\em Hall algebras and quantum groups}, Invent. Math. {\bf101} (1990), 583--591.
		
		\bibitem[Rin96]{Rin3} C.M. Ringel,
		{\em PBW-bases of quantum groups}, J. reine angrew. Math. {\bf 470} (1996), 51--88.
		
		%\bibitem[Rin96b]{Rin4}
		%C.M. Ringel, {\em Green's theorem on Hall algebras} in Representations of Algebras and Related Topics, CMS Conference Proceedings {\bf 19}, Providennce, 1996,185--245.
		
		
		
		
		%\bibitem[R12]{R12}
		%\bysame,
		%{\em  Quiver Hecke algebras and 2-Lie algebras},  Algebra Colloq. \textbf{19} (2012),  359--410.
		
		%\bibitem[Sch17]{Sch17} S. Scherotzke, {\em Desingularization of Quiver Grassmannians via Nakajima Categories}, Algebr. Represent. Theor. {\bf 20} (2017), 231--243.
		
		\bibitem[Sch19]{Sch19} S. Scherotzke, {\em Generalized quiver varieties and triangulated categories}, Math. Z. {\bf 292} (2019), 1453--1478. %,  \href{https://arxiv.org/abs/1405.4729}{arXiv:1405.4729v6}
		
		\bibitem[SS16]{SS16}  S. Scherotzke and N. Sibilla, {\em Quiver varieties and Hall algebras}, Proc. London Math. Soc. {\bf 112} (2016), 1002--1018.
		
		
		%\bibitem[Sh22]{Sh22} L. Shen, {\em Duals of Semisimple Poisson–Lie Groups and Cluster Theory of Moduli Spaces of G-local
		% Systems},  
		% IMRN {\bf 18} (2022), 14295--14318.
		
		\bibitem[Sh22]{Sh22} L. Shen, {\em Cluster nature of quantum groups}, \href{https://arxiv.org/abs/2209.06258}{arXiv:2209.06258}
		
		\bibitem[VV03]{VV} M. Varagnolo and E. Vasserot, {\em Perverse sheaves and quantum Grothendieck rings}, Studies in memory
		of Issai Schur (Chevaleret/Rehovot, 2000), Progress in Mathematics {\bf210} (Birkhauser Boston, Boston, MA,
		2003), 345--365.
		
		
		
		
		
		%\bibitem[XZ05]{XZ05} J. Xiao and B. Zhu, {\em Locally finite triangulated categories}, J. Algebra {\bf 290} (2005), 473--490.
	\end{thebibliography}
\end{document}